\documentclass{amsart}

\usepackage[letterpaper,top=2cm,bottom=2cm,left=3cm,right=3cm,marginparwidth=1.75cm]{geometry}

\usepackage[utf8]{inputenc}
\usepackage{amsmath}
\usepackage{mathtools}   
\usepackage{amsfonts}
\usepackage{listings}
\usepackage{cancel}
\usepackage{comment}
\usepackage{mathabx}
\usepackage{amssymb}
\usepackage{subfig}

\usepackage{hyperref}
\usepackage{verbatim}

\usepackage{todonotes}

\theoremstyle{plain}   						
\newtheorem{theorem}{Theorem}[section]			
\newtheorem{lemma}[theorem]{Lemma} 		 	
\newtheorem{proposition}[theorem]{Proposition}	
\newtheorem{corollary}[theorem]{Corollary}

\DeclareMathOperator{\sign}{sign}
\theoremstyle{remark} 
\newtheorem{remark}{Remark}[section]

\theoremstyle{definition}
\newtheorem{definition}{Definition}[section]

\title[On the Gross-Pitaevskii model with a moving impurity]{On the Gross-Pitaevskii model with a moving impurity: Cauchy problem and superfluidity criterion}

\author[]{Paolo Antonelli}
\address{Gran Sasso Science Institute, Viale Francesco Crispi, 7, 67100, L'Aquila, Italy}
\email{paolo.antonelli@gssi.it} 
\author[]{Martino Caliaro}
\address{Gran Sasso Science Institute, Viale Francesco Crispi, 7, 67100, L'Aquila, Italy}
\email{martino.caliaro@gssi.it}


\begin{document}
\maketitle
\begin{abstract}
    We study the one-dimensional Gross-Pitaevskii equation with a traveling delta potential and non-zero conditions at infinity. This model describes the effect of a moving impurity in a quantum fluid. Firstly, we show that the associated Cauchy problem is globally well-posed in the energy space. This requires the definition of a conserved energy, which involves the notion of renormalized momentum. Secondly, we study the existence and stability of stationary states in a co-moving reference frame. It is known that there exists an impurity-dependent critical velocity above which no stationary state exists. For velocities below the critical one, two different stationary states appear. We show the orbital stability of the one with higher minimal density.
\end{abstract}
\hfill

 

\section{Introduction}
A classical experiment to probe superfluidity in dilute Bose-Einstein condensates (BEC) and superfluid helium consists in studying the response of such fluids to a moving impurity \cite{inouye, kwon,onofrio, neely, raman}. In this experiment, superfluidity manifests in the existence of a critical velocity below which the motion of the impurity is dissipationless, i.e. it happens with no exchange of momentum or energy with the fluid.\\ 
From the mathematical point of view, this phenomenon has been investigated within the Gross-Pitaevskii model, which provides an accurate description of the dynamics of dilute BEC at very low temperatures, and is a qualitative model for superfluid helium \cite{primer_quantum_fluids,pitaevskii_stringari}. A central contribution in this direction, comes from the work of Frisch, Pomeau and Rica \cite{frisch}, who studied numerically the 2D Gross-Pitaevskii equation in the presence of a traveling impenetrable disk (modeling the impurity). In their work, they observe a 'transition to dissipation': below a critical velocity the disk moves without exchanging momentum or energy with the fluid. When this velocity threshold is exceeded, the transfer of momentum happens through the nucleation of vortex-antivortex pairs at the opposite sides of the disk's surface. This transition, in agreement with the experiments, has been observed in subsequent numerical simulations of the Gross-Pitaevskii equation in dimensions two and three \cite{nore, Winiecki,Winiecki_2}.\\
In the present work we focus on a one-dimensional formulation of this problem. We consider a quantum fluid on the line, which is at rest at infinity with constant density. The impurity is modeled by means of a short-range external potential and moves through the fluid at constant speed. More precisely, we study the 1D Gross-Pitaevskii equation with a repulsive delta potential $\gamma \delta$ traveling at speed $v \in \mathbf{R}$:
\begin{equation}
    i\partial_t u + \partial_x^2 u - \gamma \delta(x-vt)u +(1-|u|^2)u  = 0, \qquad x \in \mathbf{R}, \quad t \in \mathbf{R}, \quad  \gamma >0,
    \label{eq:initial}
\end{equation}
for $u$ complex valued field on $\mathbf{R}$. We impose the condition of constant density for the fluid at infinity, which, up to rescaling, reads
\begin{equation}
    |u(x)| \to 1 \qquad \text{as} \qquad |x| \to \infty.
    \label{bozza_condition}
\end{equation}
Equation (\ref{eq:initial}) with condition (\ref{bozza_condition}) has been considered by Hakim in \cite{hakim1997nonlinear} and by several other authors in the physical literature (\cite{pham_2, pitaevskii, Leboeuf, maris2003, pavloff, pham-brachet}). As we will see, in its simplicity, this one-dimensional model captures many of the important features of its higher dimensional analogues, such as the mentioned transition to dissipation, and it provides a simple framework in which to study certain aspects of superfluidity.\\
It is convenient to write equation (\ref{eq:initial}) in a co-moving reference frame, i.e. after the change of coordinates $x \to x-vt$. Thus, we obtain 
\begin{equation}
\tag{GP-$\delta$}
    i\partial_t u + \partial_x^2 u -iv\partial_xu  - \gamma \delta(x)u +(1-|u|^2)u  = 0.
    \label{bozza_NLS}
\end{equation}
This form of the equation, with condition (\ref{bozza_condition}), is the object of our study. As we detail in the next sections, the first of our goals is to formulate the Cauchy problem associated to \eqref{bozza_NLS} and to show its global well-posedness. Moreover, we study the existence and stability of time-independent solutions to \eqref{bozza_NLS} and we discuss their relation to the phenomenon of superfluidity.
\subsection{The Cauchy problem}
In defining the space of finite energy states in which we formulate the Cauchy problem for \eqref{bozza_NLS}, we must take into account the condition \eqref{bozza_condition}. To this end, we draw inspiration from the standard Gross-Pitaevskii equation
\begin{equation}
    i\partial_t u + \partial_x^2 u +(1-|u|^2)u  = 0,
    \label{bozza_intro_gp}
\end{equation}
that corresponds to \eqref{bozza_NLS} by considering $\gamma =v =0$. In the presence of condition \eqref{bozza_condition}, this equation has been studied by many authors (\cite{zhidkov_2,zhidkov,gerard,  Gallo,antonelli}), and it provides a natural setting in which to study nonlinear phenomena in freely evolving BEC. Differently from the case of vanishing conditions at infinity, the presence of condition (\ref{bozza_condition}) permits a rich dynamics, which includes the existence of standing and traveling waves. Equation (\ref{bozza_intro_gp}) can be interpreted as the Hamiltonian system for the energy
\begin{equation}
    E_0(u) = \frac{1}{2}\int_{\mathbf{R}}|\partial_x u|^2dx + \frac{1}{4} \int_{\mathbf{R}}(1-|u|^2)^2dx.
    \label{energy_e_0}
\end{equation}
Therefore, the natural energy space associated to \eqref{energy_e_0} is determined by
\begin{equation*}
    \mathcal{E} = \{u \in H^1_{loc}(\mathbf{R})| \quad \partial_xu \in L^2(\mathbf{R}), \quad 1-|u|^2 \in L^2(\mathbf{R}) \}.
\end{equation*}
Notice that $\mathcal{E}$ is \textit{not} a vector space. On the other hand, by defining the distance
\begin{equation}
    d_{\infty}(u,w) = \bigl|\bigl|\partial_xu-\partial_xw\bigr|\bigr|_{L^2(\mathbf{R})} + \bigl|\bigl|u-w\bigr|\bigr|_{L^{\infty}(\mathbf{R})} + \bigl|\bigl||u|^2-|w|^2\bigr|\bigr|_{L^2(\mathbf{R})}
\end{equation}
the pair $(\mathcal{E},d_{\infty})$ can be shown to be a complete metric space \cite{gerard}. Moreover, one can show that $\mathcal{E}$ consists of uniformly continuous functions which satisfy (\ref{bozza_condition}). 
In \cite{zhidkov_2, zhidkov}, Zhidkov studied the Cauchy problem for \eqref{bozza_intro_gp} and proved that it is globally well-posed in $\mathcal{E}$. In higher dimensions, analogous results have been obtained by Gérard \cite{gerard}, Gallo \cite{Gallo} and Antonelli et al. \cite{antonelli}.\\ 
The singularly perturbed Gross-Pitaevskii equation, obtained from \eqref{bozza_NLS} by setting $v=0$, requires a more delicate analysis, especially in the treatment of the linear propagator. This model has been studied by Ianni et al. in \cite{LeCozIR} and it is used to describe the dynamics of a BEC in the presence of a static impurity. With condition \eqref{bozza_condition}, this equation can be interpreted as the Hamiltonian system for the energy 
\begin{equation*}
    E_\gamma(u) = E_0(u) + \frac{\gamma}{2}|u(0)|^2
\end{equation*}
and, therefore, it admits $\mathcal{E}$ as the natural energy space of the system. As proven in \cite{LeCozIR}, the Cauchy problem for this equation is globally well-posed in $\mathcal{E}$.\\
As we will discuss, the space $\mathcal{E}$ constitutes the  natural energy space also for \eqref{bozza_NLS}, and this fact motivates us to study the Cauchy problem for \eqref{bozza_NLS} in this space.\\
We begin by analyzing the \textit{local} well-posedness of \eqref{bozza_NLS} in $\mathcal{E}$, and by identifying the associated linear evolution. Firstly, we give a rigorous definition of the operator  $H_{\gamma}:=-\partial_x^2+iv\partial_x+\gamma\delta(x)$ as a singular perturbation of $-\partial_x^2 + iv\partial_x$ in $H^2(\mathbf{R})$. This is done thanks to the theory of self-adjoint extensions (\cite{albeverio_solvable, teta2018primer}). The self-adjoint operator $H_{\gamma}$ that we obtain generates a unitary group in $L^2(\mathbf{R})$, which we are able to compute explicitly, following \cite{albeverio1995fundamental}. Using the approach of Ianni et al. \cite{LeCozIR}, we extend the action of the unitary group from $L^2(\mathbf{R})$ to $\dot{H}^1(\mathbf{R})$, and consequently to $\mathcal{E}$. This extension defines the linear propagator $T_{\gamma}(t)$ associated to \eqref{bozza_NLS} in $\mathcal{E}$. \\
Once the linear propagator is identified, we show the local well-posedness in $\mathcal{E}$ by the Duhamel formulation of the equation and a fixed point argument. Moreover, we obtain a blow-up alternative which holds on the functional $E_0$, defined in (\ref{energy_e_0}).\\
The study of the propagator $T_{\gamma}(t)$ reveals an additional property of the dynamics associated to (\ref{bozza_NLS}), which is inherited from the Gross-Pitaevskii equation \cite{gerard}. If $u_0 \in \mathcal{E}$ is the initial datum of the Cauchy problem and $u(t)$ its corresponding solution, then it holds
\begin{equation}
    u(t) -u_0\in  H^1(\mathbf{R}) \qquad \text{for all $t$ of existence.}
    \label{intro_property}
\end{equation}
In other words, the dynamics associated to \eqref{bozza_NLS} preserves the behavior of $u_0$ at infinity.\\
We now turn to the problem of defining the conserved energy associated to \eqref{bozza_NLS}.
As we discuss in Section \ref{section:energy}, its definition requires a proper notion of \textit{linear momentum} for fields in $\mathcal{E}$, as the classical one $P(u) = \frac{1}{2}\Im \int_{\mathbf{R}}\overline{u}\partial_xu\ dx$ is not well-defined in this setting.
This issue has been already addressed in the context of the standard Gross-Pitaevskii equation and has been dealt with a renormalization procedure  (\cite{barashenkov, bethuel,de_laire,maris_mur}). Here we adopt a slightly different approach, which is similar to the previously mentioned, but more oriented towards the study of the Cauchy problem for \eqref{bozza_NLS}. \\
Given $u_0 \in \mathcal{E}$, we consider the affine subspace $u_0+H^1(\mathbf{R}) \subset \mathcal{E}$. In this set, we define the momentum $P_{u_0}:u_0+H^1(\mathbf{R})  \to \mathbf{R}$ by  
\begin{equation}
    P_{u_0}(u) = \frac{1}{2}\Im\int_{\mathbf{R}}(\overline{u}-\overline{u}_0)\partial_x(u+u_0)dx.
    \label{bozza_intro_mom_u0}
\end{equation}
We interpret $P_{u_0}(u)$ as the relative momentum of $u$ with respect to $u_0$. In the following, we will take $u_0$ to be the initial datum of the Cauchy problem. If $u(t)$ is the corresponding solution, the map $t \to P_{u_0}(u(t))$ is well-defined for all times of existence thanks to the property \eqref{intro_property}. Notice that the momentum is not a conserved quantity, due to the presence of the external potential in the equation of motion. \\
The Hamiltonian associated to \eqref{bozza_NLS} can be now defined as follows. Given $u_0 \in \mathcal{E}$ we define the Hamiltonian as the map $K_{u_0}:u_0+H^1(\mathbf{R}) \to \mathbf{R}$ such that
\begin{equation}
    K_{u_0}(u) = E_{\gamma}(u) - vP_{u_0}(u).
    \label{bozza_intro_energy_u0}
\end{equation} 
and we prove that it is conserved by any solution $u(t)$ with initial datum $u(0) \in u_0 + H^1(\mathbf{R})$.\\
We now exploit the conservation of $K_{u_0}$ to study the global well-posedness of \eqref{bozza_NLS}. 
By the blow-up alternative, we need to establish a bound on $E_0(u(t))$. Since $K_{u_0}$ is not sign-definite, its conservation alone is insufficient to this goal. For this reason we study the time evolution of $P_{u_0}(u(t))$, which leads to a Gr\"{o}nwall-type inequality for $E_0(u(t))$ and, consequently, to the global well-posedness of the Cauchy problem (Section \ref{section: nbu}).
\begin{theorem}
    Let $\gamma >0$ and $v \in \mathbf{R}$. For any $u_0 \in \mathcal{E}$, there exists a unique, global, continuous (for $d_{\infty}$) solution $u:\mathbf{R} \to \mathcal{E}$ to equation (\ref{bozza_NLS}) with $u(0) = u_0$. Moreover, the following properties hold.
    \begin{itemize}
        \item Conservation of energy:
    \begin{equation*}
         K_{u_0}(u(t)) = K_{u_0}(u_0) \qquad \forall t \in \mathbf{R}.
    \end{equation*}
    \item Continuity with respect to initial data: Let $R>0$ and $T>0$. There exist two constants $C_R >0$ and $C'_R>0$ such that, if $u_0,\tilde{u}_0 \in \mathcal{E}$ satisfy $E_0(u_0) \leq R$ and $E_0(\tilde{u}_0) \leq R$, the corresponding solutions $u$ and $\tilde{u}$ satisfy
        \begin{equation*}
            d_{\infty}(u(t),\tilde{u}(t)) \leq C_Rd_{\infty}(u_0,\tilde{u}_0), \qquad \forall t \in (-T,T).
        \end{equation*}
        If, in addition, $u_0-\tilde{u}_0 \in H^1(\mathbf{R})$, then the corresponding solutions $u, \tilde{u}$ satisfy $u(t)-\tilde{u}(t) \in C^0(\mathbf{R}, H^1(\mathbf{R}))$ and
        \begin{equation*}
        ||u(t)-\tilde{u}(t)||_{H^1(\mathbf{R})} \leq C'_{R}\Bigl(1+||u_0-\tilde{u}_0||_{H^1(\mathbf{R})}\Bigr)||u_0-\tilde{u}_0||_{H^1(\mathbf{R})}, \qquad \forall t \in (-T,T).
        \end{equation*}
        \end{itemize}
        \label{bozza_intro_prop_cauchy}
\end{theorem}
\subsection{Stationary solutions} The second aim of the present work is the study of time-independent solutions to (\ref{bozza_NLS}) and their stability. In the following, we state our results assuming $v>0$. Analogous results hold for the case $v<0$, thanks to the symmetries of the equation. For the case $v=0$, instead, we refer to the work of Ianni et al. in \cite{LeCozIR}.\\
A time-independent solution to (\ref{bozza_NLS}) is a function $u \in \mathcal{E}$ that satisfies
\begin{equation}
     \partial_x^2 u -iv\partial_x u - \gamma \delta(x)u +(1-|u|^2)u  = 0, \qquad v>0.
     \label{bozza_stationary_sol_intro}
\end{equation}
In the absence of the external potential, i.e. when $\gamma =0$, equation (\ref{bozza_stationary_sol_intro}) reduces to
\begin{equation}
\tag{TW}
     \partial_x^2 u -iv\partial_x u +(1-|u|^2)u  = 0.
     \label{bozza_traveling_w_equation}
\end{equation}
This equation arises in the study of traveling waves for the standard Gross-Pitaevskii equation \eqref{bozza_intro_gp}, and the solutions to \eqref{bozza_traveling_w_equation} are called traveling wave profiles. These have been extensively studied (see \cite{barashenkov, bethuel_existence, chiron_paper, zhiwu}) and they play a central role in solving \eqref{bozza_stationary_sol_intro}. In the following theorem we collect some well-known properties of solutions to \eqref{bozza_traveling_w_equation}, while for a more comprehensive review of these results we refer to \cite{bethuel_existence}. As we will see, a central role is played by the speed $\sqrt{2}$, which corresponds to the speed of sound waves at infinity around the constant solution $u=1$ \cite{chiron_paper}.
\begin{theorem}[\cite{bethuel_existence}]
    Let $v \geq 0$. Assume $u$ is a solution to \eqref{bozza_traveling_w_equation}. 
    \begin{itemize}
        \item[i)] If $v \geq \sqrt{2}$, then $u$ is a constant of modulus one;
        \item[ii)] If $0 \leq v < \sqrt{2}$, up to a multiplication by a constant of modulus one and a translation, $u$ is either identically equal to 1, or $u(x) = (1+r(x))e^{i\theta(x)}$ where
        \begin{equation}
            r(x) = -1+\sqrt{\frac{v^2}{2}+(1-\frac{v^2}{2})\tanh^2\Bigl(\sqrt{\frac{2-v^2}{4}}x\Bigr)}
            \label{bozza_tw_mod}
        \end{equation}
        and $\theta(x)$ is the solution of 
        \begin{equation}
            \partial_x\theta = \frac{v}{2}\Bigl(1-\frac{1}{(1+r)^2}\Bigr), \qquad \text{with} \qquad \theta(0) = 0.
            \label{bozza_tw_phase}
        \end{equation}
    \end{itemize}
    \label{bozza_theorem_tw}
\end{theorem}
In the case $\gamma>0$, equation \eqref{bozza_stationary_sol_intro} differs from \eqref{bozza_traveling_w_equation} due to the presence of the delta potential. The latter imposes a boundary condition at $x=0$, that solutions to \eqref{bozza_stationary_sol_intro} have to satisfy. This condition takes the form of a jump discontinuity on the first derivative and it reads
\begin{equation}
    \partial_xu(0^+)-\partial_xu(0^-) = \gamma u(0).
    \label{bozza_jump_intro}
\end{equation}
Therefore, in order to obtain a solution to \eqref{bozza_stationary_sol_intro} it is sufficient to match two traveling wave profiles at $x=0$ in such a way that continuity and \eqref{bozza_jump_intro} are satisfied. Following this approach, Hakim \cite{hakim1997nonlinear} determined the criterion for the existence of solutions to \eqref{bozza_stationary_sol_intro} and their explicit expression. 
\begin{proposition}[\cite{hakim1997nonlinear}]
    Let $v >0$ and $\gamma >0$. An element $u \in \mathcal{E}$ is a solution to (\ref{bozza_stationary_sol_intro}) if and only if, up to phase shifts, $u(x) = (1+r(x;\xi))e^{i\theta(x;\xi)}$ where
    \begin{equation}
        r(x;\xi) = -1+ \sqrt{\frac{v^2}{2}+(1-\frac{v^2}{2})\tanh^2\Bigl(\sqrt{\frac{2-v^2}{4}}(x\pm \xi)\Bigr)}, \qquad \pm x \geq 0,
        \label{bozza_intro_stat_sol}
    \end{equation}
    and $\theta(x;\xi)$ is the solution of $$\partial_x\theta(x;\xi) = \frac{v}{2}\Bigl(1-\frac{1}{(1+r(x;\xi))^2}\Bigr), \qquad \text{with} \qquad \theta(0;\xi) = 0.$$
    Here, $\xi>0$ is a displacement parameter that satisfies
    \begin{equation}
        \gamma = \sqrt{2}(1-v^2/2)^{3/2}\frac{\tanh(\sqrt{1/2-v^2/4}\ \xi)}{v^2/2+\sinh^2(\sqrt{1/2-v^2/4}\ \xi)}.
        \label{bozza_matching}
    \end{equation}
    \label{bozza_intro_hakim}
\end{proposition}
\begin{figure}
\centering
\subfloat[]
{\includegraphics[width=.6\textwidth]{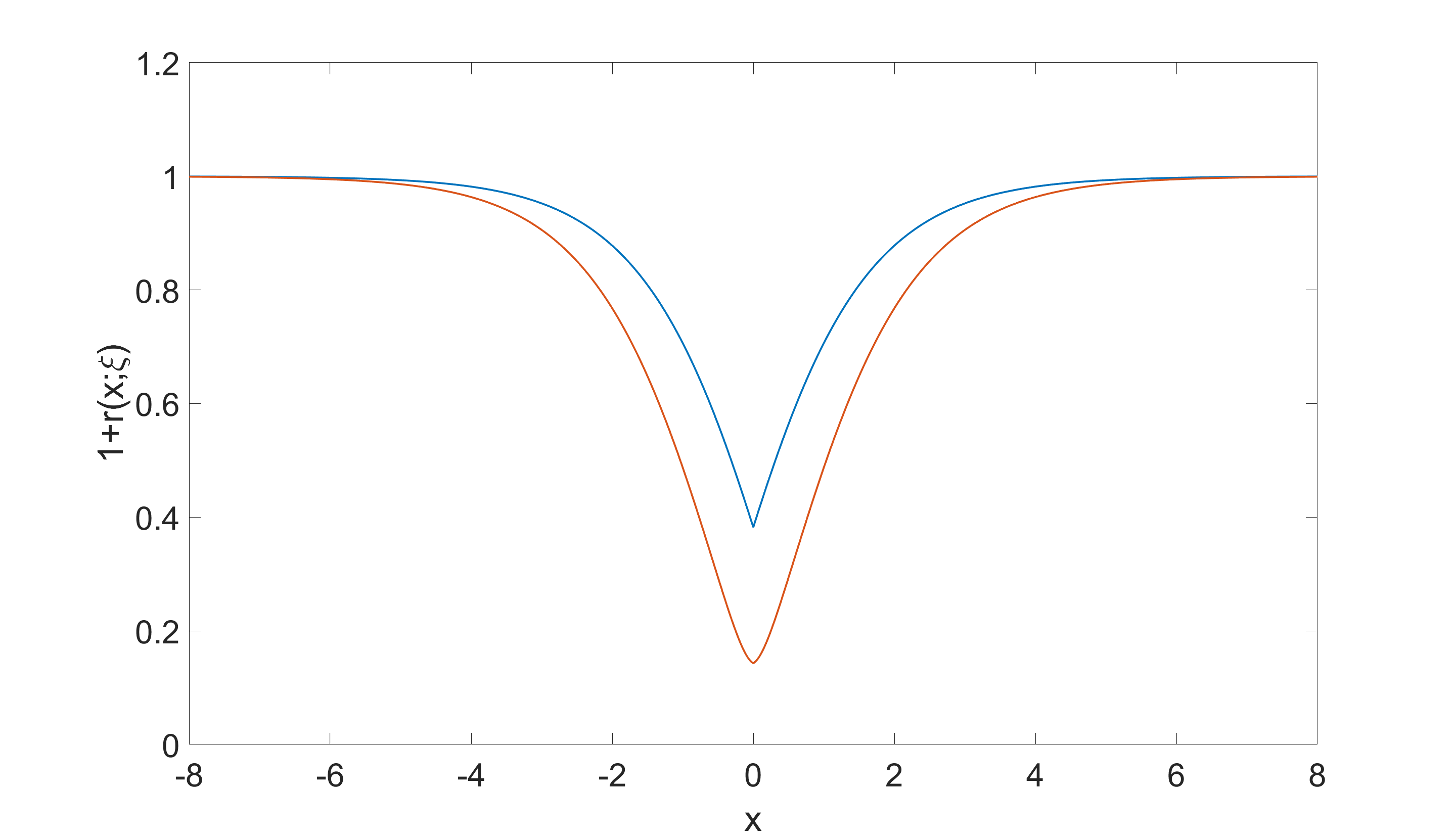}} \quad
\subfloat[]
{\includegraphics[width=.6\textwidth]{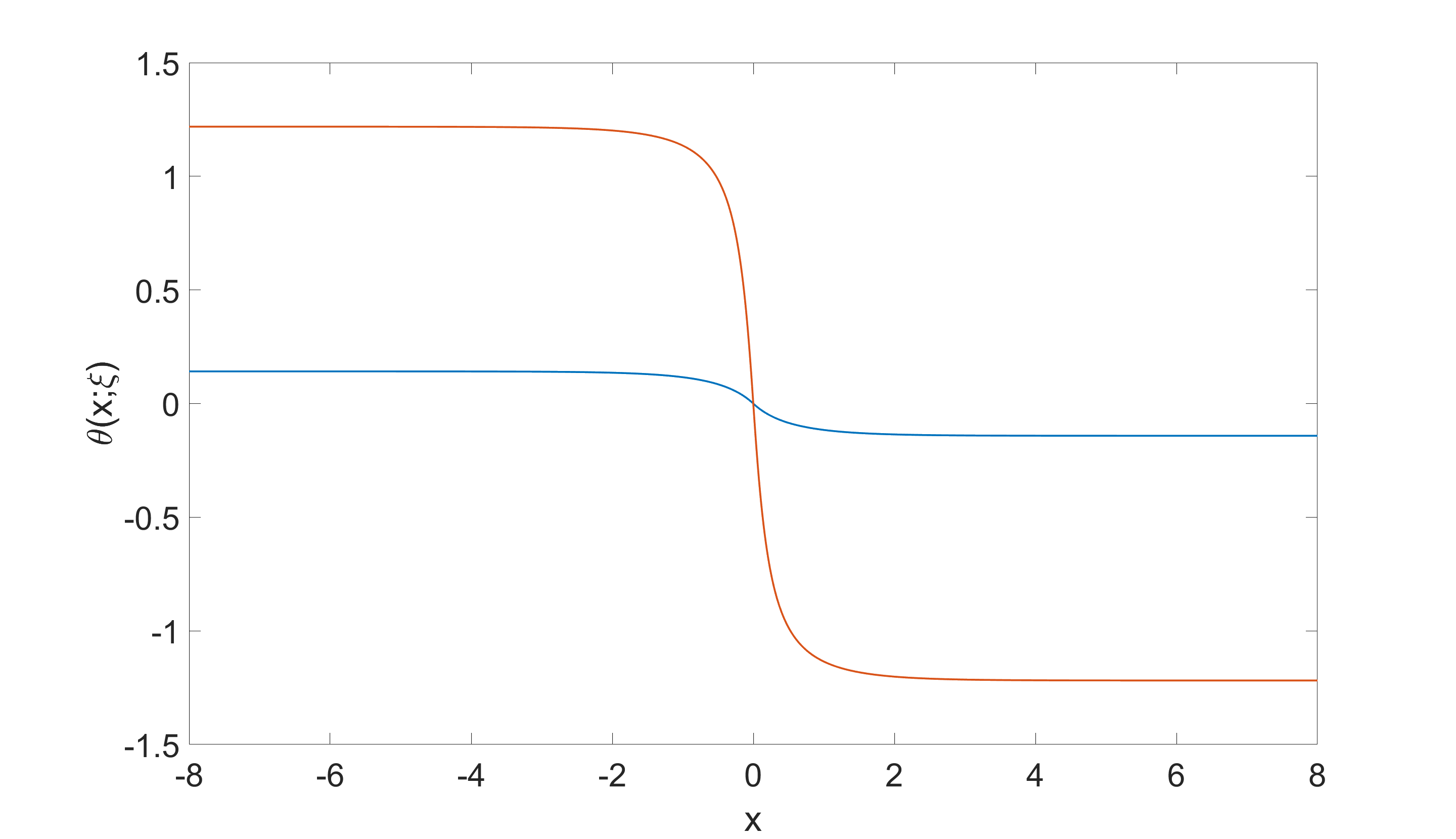}}
\caption{ We set $v=0.2$ and $\gamma = 2$. In this case the critical velocity is $v_{cr} \sim 0.419$ and the two roots of \eqref{bozza_matching} are $\xi_2\sim0.044$ and $\xi_1 \sim 0.766$. In A) we report the function $1+r(x;\xi)$ for these two values of $\xi$. The upper curve corresponds to $1+r(x;\xi_1)$, the bottom curve to $1+r(x;\xi_2)$. In B) we report the associated phases. The curve that has the smallest range corresponds to $\theta(x;\xi_1)$, while the one with the biggest range corresponds to $\theta(x;\xi_2)$.  }
\label{fig:stat_sol}
\end{figure}

By studying the relation (\ref{bozza_matching}), one can deduce the following assertions (see Proposition \ref{bozza_maris} for the complete statement).
\begin{corollary}[\cite{hakim1997nonlinear,maris2003}]
    Given $\gamma >0$, there exists a critical velocity $0<v_{cr}(\gamma)<\sqrt{2}$ such that:
    \begin{itemize}
        \item if $0<v<v_{cr}(\gamma)$, there exist two values $\xi_1,\xi_2 >0$, with $\xi_2<\xi_1$, satisfying (\ref{bozza_matching}) and, up to phase shifts, two solutions of (\ref{bozza_stationary_sol_intro});
        \item if $v=v_{cr}(\gamma)$, there exists a single value $\xi>0$ satisfying (\ref{bozza_matching}) and, up to phase shifts, a single solution to (\ref{bozza_stationary_sol_intro});
        \item if $v>v_{cr}(\gamma)$, there are no values of $\xi>0$ satisfying (\ref{bozza_matching}) and no solutions of (\ref{bozza_stationary_sol_intro}).
    \end{itemize}
    \label{bozza_corollary_intro}
\end{corollary}
    The plot of two stationary solutions for a particular choice of $v$ and $\gamma$ is reported in Figure \ref{fig:stat_sol}.\\
    The existence of stationary solutions in the subcritical regime $0<v\leq v_{cr}(\gamma)$ has interesting physical implications. In contrast to generic time-dependent solutions to (\ref{bozza_NLS}), stationary states naturally conserve the linear momentum. This means that if a quantum fluid is in a stationary configuration, the motion of the impurity happens without transfer of momentum. In the present setting, this corresponds to absence of drag forces \cite{pavloff}, and it is interpreted as superfluid motion. For this reason, the regime $0<v\leq v_{cr}(\gamma)$ is called \textit{superfluid} regime \cite{Leboeuf}.\\
    The supercritical regime $v > v_{cr}(\gamma)$, where no stationary solutions exist, is instead called \textit{dissipative}\footnote{From the mathematical point of view, the word dissipative doesn't refer to dissipation of the total energy, as the governing equations are conservative. It refers to the formation of traveling localized structures (i.e. gray solitons or quantum vortices) which carry a portion of energy away from the impurity location towards infinity. This process has the net effect of reducing the fluid velocity in the proximity of the impurity. In the experimental setting, instead, the word dissipative refers to the heating induced by a moving laser beam on a BEC, as observed in \cite{raman}.} \cite{frisch}. Numerical simulations of (\ref{bozza_NLS}) in this regime, performed by Hakim \cite{hakim1997nonlinear} and Pham \cite{pham-brachet}, display a dynamics characterized by the repeated emission of gray solitons from the location of the impurity. It is by means of this mechanism that the main transfer of momentum between the impurity and the fluid happens. This behavior is analogous to the one observed in the 2D and 3D cases, with gray solitons replaced by quantum vortices \cite{frisch, Winiecki_2}.\\
    The existence of a critical velocity above which stationary solutions disappear is in agreement with Landau's criterion of superfluidity \cite{landau}, which predicts dissipative motion above a velocity threshold. While for a homogeneous quantum fluid the critical velocity is the speed of sound $\sqrt{2}$, the value of $v_{cr}$ from Corollary \ref{bozza_corollary_intro} is strictly below $\sqrt{2}$. Moreover, it depends on the characteristics of the impurity (such as the strength $\gamma$). These two facts are consistent with BEC experiments \cite{onofrio, raman} and 2D and 3D numerical simulations of Gross-Pitaevskii models \cite{frisch, nore}\footnote{We mention that, in certain models (see \cite[Section C]{hakim1997nonlinear}), the value of the critical velocity can be determined with the following criterion: the fluid velocity (relative to the impurity) must be everywhere subsonic, i.e. smaller than the \textit{local} speed of sound. At those velocities $v$ for which this condition cannot be fulfilled, stationary solutions cease to exist and dissipative motion takes place. The occurrence of a critical speed below $\sqrt{2}$ is then explained by the fact that the fluid velocity may become locally supersonic even if $v < \sqrt{2}$. Then, the onset of dissipation corresponds to the emission of solitons (or quantum vortices) from these supersonic regions, which are usually located near impurity (see \cite{hakim1997nonlinear}, \cite{pavloff} and \cite{frisch} for a dedicated discussion in the 1D and 2D cases).}. \\
    The considerations made above seem to indicate that superfluid motion in \eqref{bozza_NLS} is characteristic of stationary states only. These states are, however, isolated configurations in the space $\mathcal{E}$, and it is of interest to investigate whether nearby solutions exhibit similar superfluid properties. To address this question, we study the stability of stationary solutions to \eqref{bozza_NLS}.\\
    From now on we focus on the case $0<v<v_{cr}(\gamma)$. Let $u(x;\xi_1)$ and $u(x;\xi_2)$ be the stationary solutions associated to the parameters $\xi_1,\xi_2$, respectively, with $\xi_2<\xi_1$, see Corollary \ref{bozza_corollary_intro}. 
    It is believed that $u(x;\xi_1)$ is stable, while $u(x;\xi_2)$ is unstable with respect to the (\ref{bozza_NLS}) dynamics \cite{hakim1997nonlinear,pham-brachet}.\\
    Our second main result consists in proving the stability of $u(x;\xi_1)$. The proof is based on a variational argument, which is inspired by the work of Mari\c{s} \cite{maris2003}.\\
    More precisely, we consider the set $U_0 = \{u \in \mathcal{E}\ | \ \inf_{x \in \mathbf{R}}|u(x)|>0\}$ of nowhere vanishing functions. On this set we can define a different notion of linear momentum $\mathcal{P}: U_0 \to \mathbf{R}$ and energy $\mathcal{K}:U_0 \to \mathbf{R}$ as in \cite{bethuel_existence}. Namely, for $u(x) = \rho(x)e^{i\theta(x)} \in U_0$, we define
    \begin{equation}
        \mathcal{P}(u) := \frac{1}{2}\Im\int_{\mathbf{R}}(\overline{u}-e^{-i\theta})\partial_x(u+e^{i\theta})dx = \frac{1}{2}\int_{\mathbf{R}}(\rho^2-1)\partial_x\theta dx
        \label{bozza_intro_mom_theta}
    \end{equation}
    and 
    \begin{equation*}
    \mathcal{K}(u) := E_{\gamma}(u) - v\mathcal{P}(u).   
    \end{equation*}
    In this case, we interpret $\mathcal{P}(u)$ in (\ref{bozza_intro_mom_theta}) as the momentum of $u \in U_0$ relative to its own phase field $e^{i\theta} \in U_0$. In Section \ref{section:energy} and \ref{section: nbu} we study the relation between the two notions of  momentum and energy we have encountered so far. In particular, we show that also the functional $\mathcal{K}$ is conserved by the solutions to (\ref{bozza_NLS}), as long as they belong to $U_0$.\\
    By following the variational argument in \cite{maris2003}, we show that the stationary state $u(x;\xi_1)$  is a local minimum of $\mathcal{K}$ in $U_0$. More precisely, $u(x;\xi_1)$ minimizes the energy $\mathcal{K}$ among all functions whose minimal density is above a certain threshold
    (see Corollary \ref{bozza_corollary_u_a1} for a precise statement).\\
    The proof of stability is then based on a careful analysis of the \eqref{bozza_NLS} dynamics in $U_0$ and on the argument of Cazenave and Lions \cite{cazenave_lions} for the stability of constrained minima. In particular, the latter argument requires us to study the compactness of minimizing sequences, which we do following the work of Bethuel et al. \cite{bethuel_existence}.\\ 
    For the stability result we use the following distance on $\mathcal{E}$
    \begin{equation}
    d_{A}(u,w) = \bigl|\bigl|\partial_xu-\partial_xw\bigr|\bigr|_{L^2(\mathbf{R})} + \bigl|\bigl|u-w\bigr|\bigr|_{L^{\infty}([-A,A])} + \bigl|\bigl||u|^2-|w|^2\bigr|\bigr|_{L^2(\mathbf{R})},
    \label{bozza_d_A}
    \end{equation}
    defined for any $A>0$.
    \begin{theorem}
        For $\gamma >0$ and $v \in \bigl(0,v_{cr}(\gamma)\bigr)$, let $\xi_1,\xi_2 >0$ be the two solutions of equation (\ref{bozza_matching}), with $\xi_2<\xi_1$. If $u(x;\xi_1)$ denotes the stationary solution to \eqref{bozza_NLS} associated with $\xi_1$, then $u(x;\xi_1)$ is orbitally stable. More precisely, for any $\varepsilon >0$ and $A>0$ there exists $\delta >0$ such that, for any $u_0 \in \mathcal{E}$ satisfying
        \begin{equation*}
            d_{A}(u_0,u(\cdot;\xi_1)) < \delta,
        \end{equation*}
        it holds
        \begin{equation*}
            \sup_{t \in \mathbf{R}}\inf_{\phi \in \mathbf{R}}d_A(\psi(t),e^{i\phi}u(\cdot;\xi_1)) \leq \varepsilon.
        \end{equation*}
        Here, $\psi(t)$ is the unique solution to (\ref{bozza_NLS}) with initial datum $u_0$.
        \label{bozza_intro_stability}
    \end{theorem}
     The orbital stability of the stationary state $u(x;\xi_1)$ allows us to estimate, for certain solutions, the total exchange of momentum between the impurity and the quantum fluid. More precisely, we have the following corollary.
     \begin{corollary}
         Under the same assumptions of Theorem \ref{bozza_intro_stability}, consider the stationary state $u(x;\xi_1)$. For any $\varepsilon>0$ and $A>0$, there exists $\delta>0$ such that, if $u_0 \in U_0$ satisfies $d_A(u_0,u(\cdot;\xi_1)) <\delta$, then the solution $\psi(t)$ to \ref{bozza_NLS} with $\psi(0) = u_0$ satisfies $\psi(t) \in U_0$ for all $t \in \mathbf{R}$ and
         \begin{equation*}
             \sup_{t \in \mathbf{R}}|\mathcal{P}(\psi(t))-\mathcal{P}(u_0)| \leq \varepsilon.
         \end{equation*}
         \label{bozza_corollary_intro_momentum_exchange}
     \end{corollary}
     We conclude the introduction by mentioning an interesting model related to \eqref{bozza_NLS}, which is called the Gross-Clark-Schr\"{o}dinger system. This system also describes the motion of an impurity in a quantum fluid. Differently from the case presented here, there the impurity is described by a wavefunction which obeys a linear Schr\"{o}dinger equation, coupled to a Gross-Pitaevskii equation for the fluid. We refer the interested reader to \cite{maris_clarck, alhelou} (see also \cite{giurato,salvador} for related models). On the numerical side, we mention the recent work in \cite{chauleur}.
\section{Preliminaries}
\subsection{Notation}
We begin by briefly introducing the notation. We denote the Fourier transform as 
    \begin{equation*}
     (\mathcal{F}f)(p) = \hat{f}(p) = \frac{1}{\sqrt{2\pi}}\int_{\mathbf{R}} e^{-ipx}f(x)dx,
    \end{equation*}
    and its inverse as
    \begin{equation*}
     (\mathcal{F}^{-1}g)(x) = \check{g}(x) = \frac{1}{\sqrt{2\pi}}\int_{\mathbf{R}} e^{ipx}g(p)dp.
    \end{equation*}
    We endow the space $L^2(\mathbf{R})$ with the following inner product 
    \begin{equation*}
        \langle u,w\rangle = \int_{\mathbf{R}}\overline{u}(x)w(x)dx.
    \end{equation*}
    We indicate with $\dot{H}^1(\mathbf{R})$ the homogeneous Sobolev space defined by 
    \begin{equation}
        \dot{H}^1(\mathbf{R}) = \{u \in L^2(\mathbf{R}) | \  \ \partial_xu \in L^2(\mathbf{R})\}.
    \end{equation}
    Finally, we indicate with $\mathcal{S}(\mathbf{R})$ the Schwartz space of rapidly decaying functions, and with $C^{\infty}_0(\mathbf{R})$ the space of smooth functions from $\mathbf{R}$ to $\mathbf{C}$ with compact support. For $x \in \mathbf{R}$, we use the Japanese bracket to denote $\langle x \rangle = \sqrt{1+x^2}$.\\
    Given a time interval $I \subset\mathbf{R}$, we write $u \in C^0(I, \mathcal{E})$ to indicate a function $u:I \to \mathcal{E}$ which is continuous from $I$ to $(\mathcal{E},d_{\infty})$. In particular, if $u \in C^0(I,\mathcal{E})$, then the map $t \to u(t)$ is continuous from $I$ to $(\mathcal{E},d_A)$ for every $A>0$. In order to stay close to the notation in \cite{LeCozIR}, we denote for $u \in \mathcal{E}$
    \begin{equation*}
        |u|^2_{\mathcal{E}}:=E_0(u) = \frac{1}{2}\int_{\mathbf{R}}|\partial_xu|^2dx + \frac{1}{4} \int_{\mathbf{R}}(1-|u|^2)^2dx.
    \end{equation*}
\subsection{\texorpdfstring{Properties of the energy space $\mathcal{E}$}{Properties of the energy space E}}
The goal of this subsection is to collect some properties of functions in $\mathcal{E}$. Some of these properties have been already proven elsewhere, see for example \cite{gerard,LeCozIR}.\\
We begin with the following Lemma, which characterizes elements of $\mathcal{E}$ as bounded and uniformly continuous functions, whose modulus tends to one at infinity. 
\begin{lemma}[\cite{LeCozIR}]
    Let $u \in \mathcal{E}$. Then $u$ is uniformly continuous and bounded and 
    \begin{equation*}
        \lim_{|x| \to +\infty} |u(x)| = 1.
    \end{equation*}
    Moreover, there exists $C>0$ such that for any $u \in \mathcal{E}$ we have
    \begin{equation}
        ||u||_{L^{\infty}(\mathbf{R})} \leq C (1+|u|_{\mathcal{E}}^{2/3}).
    \end{equation}
    \label{bozza_L_infty}
\end{lemma}
 In the following Lemma we study the effect of adding a function $ w \in H^1(\mathbf{R})$ to an element in $u \in \mathcal{E}$. We always refer to \cite[Lemma 2.3]{LeCozIR} for a proof.
\begin{lemma}[\cite{LeCozIR}]
    The following assertions hold.
    \begin{itemize}
        \item [(i)] If $u \in \mathcal{E}$ and $w \in H^1(\mathbf{R})$, then $u+w \in H^1(\mathcal{E})$;
        \item [(ii)] There exists $C>0$ such that for $u \in \mathcal{E}$ and $w \in H^1(\mathbf{R})$ we have
        \begin{equation*}
            |u+w|_{\mathcal{E}} \leq C (1+|u|_{\mathcal{E}})\bigl(1+||w||^2_{H^1(\mathbf{R})}\bigr);
        \end{equation*}
        \item [iii)] Let $R>0$. There exists $C_{R}>0$ such that for every $u \in \mathcal{E}$ satisfying $|u|_{\mathcal{E}}\leq R$ we have
        \begin{equation*}
            d_{\infty}(u,u+w) \leq C_{R}\bigl(1+||w||_{H^1(\mathbf{R})}\bigr)||w||_{H^1(\mathbf{R})}, \qquad \forall w \in H^1(\mathbf{R});
        \end{equation*}
        \item[(iv)] let $R>0$. There exists $C_R>0$ such that for every $u_1,u_2 \in \mathcal{E}$ and $w_1,w_2 \in H^1(\mathbf{R})$, satisfying $\max \bigl(|u_1|_{\mathcal{E}},|u_2|_{\mathcal{E}} ||w_1||_{H^1(\mathbf{R})},||w_2||_{H^1(\mathbf{R})}\bigr)\leq R$, we have
        \begin{equation*}
            d_{\infty}(u_1+w_1,u_2+w_2) \leq C_R(d_{\infty}(u_1,u_2)+||w_1-w_2||_{H^1(\mathbf{R})}).
        \end{equation*}
    \end{itemize}
    \label{bozza_lemma_E}
\end{lemma}
The next Lemma studies the continuity of the map $u \in \mathcal{E} \to |u|_{\mathcal{E}}^2$ for the distance $d_{\infty}$.
\begin{lemma}
    There exists $C>0$ such that, for any $u_1,u_2 \in \mathcal{E}$,  it holds
    \begin{equation*}
        \Bigl||u_1|_{\mathcal{E}}^2-|u_2|_{\mathcal{E}}^2\Bigr| \leq C d_{\infty}(u_1,u_2)\Bigl(1+|u_2|_{\mathcal{E}}^2 + d_{\infty}(u_1,u_2)\Bigr)
    \end{equation*}
    \label{bozza_cont_norm}
\end{lemma}
    \begin{proof}
        The proof is based on the following inequality: for any $A,B \in \mathbf{C}$ it holds
        \begin{equation*}
            \Bigl||A|^2-|B|^2\Bigr| = \Bigl|\Re\Bigl((A-B)(\overline{A}+\overline{B})\Bigr)\Bigr| \leq |A-B|\Bigl(2|A|+|A-B|\Bigr).
        \end{equation*}
        We have, for any $u_1,u_2 \in \mathcal{E}$,
        \begin{equation*}
           \begin{split}
                \Bigl||u_1|_{\mathcal{E}}^2-|u_2|_{\mathcal{E}}^2\Bigr| & \leq \int_{\mathbf{R}}\Bigl||\partial_xu_1|^2-|\partial_xu_2|^2\Bigr| dx +\int_{\mathbf{R}}\Bigl|(1-|u_1|^2)^2-(1-|u_2|^2)^2\Bigr|dx \\ & \leq  ||\partial_xu_1-\partial_xu_2||_{L^2(\mathbf{R})}\Bigl(2||\partial_xu_2||_{L^2(\mathbf{R})}+||\partial_xu_1-\partial_xu_2||_{L^2(\mathbf{R})}\Bigr ) \\ & +  \bigl|\bigl||u_1|^2-|u_2|^2\bigr|\bigr|_{L^2(\mathbf{R})}\Bigl(2||1-|u_2|^2||_{L^2(\mathbf{R})}+\bigl|\bigl||u_1|^2-|u_2|^2\bigr|\bigr|_{L^2(\mathbf{R})}\Bigr ) \\ & \leq 16d_{\infty}(u_1,u_2)\Bigl(1+|u_2|^2_{\mathcal{E}}+d_{\infty}(u_1,u_2)\Bigr)
           \end{split}
        \end{equation*}
    \end{proof}
    The following Lemma will be useful in the study of the stability of stationary solutions. It allows us to show that if an element $u_1 \in \mathcal{E}$ is nowhere vanishing, also its close neighbors for the distance $d_A$ in \eqref{bozza_d_A} are nowhere vanishing.
\begin{lemma}
    Set $A>0$. Consider $u_1,u_2 \in \mathcal{E}$ such that $\inf_{x \in \mathbf{R}}|u_1(x)|= \varepsilon>0$. Then, there exists $C>0$, depending on $\varepsilon, ||\partial_xu_1||_{L^2(\mathbf{R})}$, such that
    \begin{equation*}
        \Bigl|\Bigl||u_1|-|u_2|\Bigr|\Bigr|^2_{L^{\infty}(\mathbf{R})} \leq C d_A(u_1,u_2)(1+d_A(u_1,u_2))
    \end{equation*}
    \label{bozza_lemma_modulus}
\end{lemma}
\begin{proof}
    By the Gagliardo-Nirenberg inequality, we have\begin{equation*}
            \begin{split}
                \Bigl|\Bigl| |u_1| - |u_2  \Bigr|\Bigr|^{2}_{L^{\infty}(\mathbf{R})} & \lesssim  \Bigl|\Bigl| |u_1| - |u_2|  \Bigr|\Bigr|_{L^{2}(\mathbf{R})} \Bigl|\Bigl| \partial_x|u_1| - \partial_x|u_2| \Bigr|\Bigr|_{L^{2}(\mathbf{R})} \\ & \lesssim \varepsilon^{-1} \Bigl|\Bigl| |u_1|^2 - |u_2|^2  \Bigr|\Bigr|_{L^{2}(\mathbf{R})} \Bigl(|| \partial_xu_1||_{L^2(\mathbf{R})} + ||\partial_xu_2||_{L^{2}(\mathbf{R})}\Bigr) \\ & \leq C d_A(u_1,u_2)(1+d_A(u_1,u_2)).
            \end{split}
        \end{equation*}
        Here we used the fact that $\bigl|\bigl|\partial_x|f|\bigr|\bigr|_{L^2(\mathbf{R})} \leq \bigl|\bigl|\partial_xf\bigr|\bigr|_{L^2(\mathbf{R})}$, for any $f \in \dot{H}^1(\mathbf{R})$.
\end{proof}
The following Theorem has been proved by Gérard in \cite[Theorem 1.8]{farina}, and provides a useful decomposition of the elements of $\mathcal{E}$.
\begin{theorem}[\cite{farina}]
    The energy space $\mathcal{E}$ consists of functions $u$ of the form
    \begin{equation}
        u=e^{i\varphi} + w
        \label{bozza_decomposition_formula}
    \end{equation}
    where $\varphi$ is a real-valued function on $\mathbf{R}$ such that $\partial_x\varphi \in L^2(\mathbf{R})$ and $w \in H^1(\mathbf{R})$. The phase $\varphi$ is determined up to adding a real function $\psi$ on $\mathbf{R}$ such that $\partial_x\psi \in L^2(\mathbf{R})$ and such that there exist $(k_+,k_-) \in \mathbf{Z}^2$ with $\psi-2\pi k_{\pm} \in L^2(\mathbf{R}_{\pm})$. Moreover, one can choose $\varphi$ so that $\partial_{x}^{\alpha}\varphi \in L^2(\mathbf{R})$ for every $\alpha >0$.
    \label{theorem_gerard}
\end{theorem}
We conclude this section with some estimates for the nonlinear term appearing in \eqref{bozza_NLS}, which we denote as
\begin{equation}
    F(u) := (1-|u|^2)u.
    \label{bozza_F_def}
\end{equation}
In particular, the nonlinearity maps element of $\mathcal{E}$ into the space $H^1(\mathbf{R})$.
\begin{lemma}[\cite{LeCozIR}]
    The function $F$ maps $\mathcal{E}$ to $H^1(\mathbf{R})$. Moreover, for any $R>0$, there exists $C_{R}>0$ such that, if $u, u_1,u_2 \in \mathcal{E}$ and $w,w_1,w_2 \in H^1(\mathbf{R})$ satisfy 
    \begin{equation*}
    \max(|u|_{\mathcal{E}},|u_1|_{\mathcal{E}},|u_2|_{\mathcal{E}}, ||w||_{H^1(\mathbf{R})},||w_1||_{H^1(\mathbf{R})},||w_2||_{H^1(\mathbf{R})}) \leq R
    \end{equation*}
    then
    \begin{align*}
        ||F(u+w)||_{H^1(\mathbf{R})} \leq C_{R} \\
        ||F(u_1+w_1)-F(u_2+w_2)||_{H^1(\mathbf{R})} \leq C_R (d_{\infty}(u_1,u_2)+||w_1-w_2||_{H^1(\mathbf{R})}).
    \end{align*}
    \label{bozza_lemma_F}
\end{lemma}
The proof of the lemma above can be found in \cite[Lemma 2.5]{LeCozIR}.
\section{\texorpdfstring{Definition of $-\partial_x^2+iv\partial_x+\gamma\delta(x)$ as a self-adjoint operator}{Definition of -partial x 2+iv partial x+ gamma delta(x) as a self-adjoint operator}}
\label{section:self-adjoint}
This section is devoted to a rigorous definition of the operator $H_{\gamma}:=-\partial_x^2+iv\partial_x+\gamma\delta(x)$ appearing in (\ref{bozza_NLS}). By means of the theory of self-adjoint extensions (\cite{albeverio_solvable,albeverio1995fundamental, teta2018primer}), we define $H_{\gamma}$ as a singular perturbation of $-\partial_x^2+iv\partial_x$ in $H^2(\mathbf{R})$, and we study some of its properties. The results we obtain are analogous to those of  Posilicano \cite{Posilicano}, who studied a similar problem in dimension three.\\
We begin by recalling some properties of the unperturbed operator \begin{equation}
    H_0:=-\partial_x^2+iv\partial_x
    \label{bozza_def_H0}
\end{equation}
defined in $H^2(\mathbf{R})$. 
\begin{lemma}
    For any $v \in \mathbf{R}$, the operator $H_0$ with domain $D(H_0) = H^2(\mathbf{R})$ is a self-adjoint operator in $L^2(\mathbf{R})$. Its spectrum consists only of the continuous part and $\sigma(H_0) =\sigma_{c}(H_0) = [-\frac{v^2}{4},+\infty)$. For $z \in \mathbf{C}\backslash\sigma(H_0)$, the integral kernel of the resolvent operator $(H_0-z)^{-1}$, i.e. the Green's function, reads
    \begin{equation}
        G_v^{-z}(x-y):= (H_0-z)^{-1}(x-y) = \frac{e^{-k|x-y|}}{2k} e^{i\frac{v}{2}(x-y)}.
        \label{bozza_green}
    \end{equation}
    where $k^2 = -z-\frac{v^2}{4}$, with $\Re k >0$.
    \label{bozza_lemma_H0}
\end{lemma}
\begin{proof}
    We consider $iv\partial_x$ as a perturbation of the self-adjoint operator $-\partial_x^2$ in $H^2(\mathbf{R})$. Since for any $\varepsilon>0$ and any $\psi \in H^2(\mathbf{R})$ we have
    \begin{equation*}
        ||iv\partial_x\psi||^2_{2}  \leq |v|^2 ||\partial_x^2\psi||_2||\psi||_2 \leq \frac{|v|^2}{2}\Bigl(\varepsilon||\partial_x^2\psi||^2_2+\frac{1}{\varepsilon}||\psi||^2_2\Bigr),
    \end{equation*}
    the Kato-Rellich's theorem \cite[Proposition 4.8]{teta2018primer} implies that $(H_0, D(H_0))$ is self-adjoint. In Fourier space, the operator $(H_0,D(H_0))$ reads $(\hat{H}_0,D(\hat{H}_0))$ where
    \begin{equation}
        \hat{H}_0 := \mathcal{F}H_0\mathcal{F}^{-1}=p^2-vp \qquad \text{and} \qquad D(\hat{H}_0) = \{\hat{f} \in L^2(\mathbf{R}) \  |  \ \int_{\mathbf{R}}|p^2\hat{f}(p)|^2dp < \infty\} 
        \label{bozza_H_0_free}
    \end{equation}
    By the unitarity of the Fourier transform, the operator $(\hat{H}_0,D(\hat{H}_0))$ is still self-adjoint and has the same spectrum of $(H_0,D(H_0))$ (see \cite[Proposition 4.17]{teta2018primer}). The resolvent operator of $\hat{H}_0$ reads
    \begin{equation}
        (\hat{H}_0-z)^{-1}(p) = \frac{1}{p^2-vp-z} = \frac{1}{(p-\frac{v}{2})^2-z-\frac{v^2}{4}}. 
        \label{bozza_resolvent_fourier}
    \end{equation}
    and we conclude that $\sigma(\hat{H}_0) = \sigma(H_0) = [-\frac{v^2}{4}, +\infty)$.  Since $H_0$ admits no eigenvalues, we conclude $\sigma(H_0) = \sigma_c(H_0)$. We denote the resolvent set as $\rho(H_0) = \mathbf{C}\backslash\sigma(H_0)$ and, for $z \in \rho(H_0)$, we set $k^2=-z-\frac{v^2}{4}$, with $\Re k>0$. By the unitarity of the Fourier transform, we can compute the integral kernel of the resolvent $(H_0-z)^{-1}$ as (see \cite[Proposition 4.24]{teta2018primer})
    \begin{equation}
        \begin{split}
            \Bigl[(H_0-z)^{-1}f\Bigr](x) &= \Bigl[\mathcal{F}^{-1}(\hat{H}_0-z)^{-1}\mathcal{F}f\Bigr](x) = \frac{1}{\sqrt{2\pi}}\int_{\mathbf{R}}e^{ipx}\frac{1}{p^2-vp-z}\hat{f}(p)dp \\ &= \int_{\mathbf{R}}(H_0-z)^{-1}(x-y)f(y)dy
        \end{split}
    \end{equation}
    where
    \begin{equation}
       \begin{split}
           (H_0-z)^{-1}(x-y) &= \frac{1}{2\pi}\int_{\mathbf{R}}e^{ip(x-y)} (\hat{H}_0 - z)^{-1}(p)dp = e^{i\frac{v}{2}(x-y)}\frac{1}{2\pi}\int_{\mathbf{R}}e^{iq(x-y)} \frac{1}{q^2+k^2}dq  \\ & = e^{i\frac{v}{2}(x-y)} \frac{e^{-k|x-y|}}{2k}.  
       \end{split}
    \end{equation}
\end{proof}
Of particular importance in the following will be the Green's function (\ref{bozza_green}) evaluated at $-z = \lambda > \frac{v^2}{4}$, which reads
\begin{equation}
    G_v^{\lambda}(x) = \frac{e^{-\sqrt{\lambda-\frac{v^2}{4}}|x|}}{2\sqrt{\lambda-\frac{v^2}{4}}}e^{i\frac{v}{2}x}.
    \label{bozza_green_function}
\end{equation}
If we take the Fourier transform of \eqref{bozza_green_function}, we obtain
\begin{equation}
    \hat{G}^{\lambda}_v(p) = \frac{1}{\sqrt{2\pi}}\frac{1}{p^2-vp+\lambda}.
    \label{bozza_green_function_Fourier}
\end{equation}
We notice that $G_v^{\lambda} \in H^1(\mathbf{R})$, but $G_v^{\lambda} \notin H^2(\mathbf{R})$.\\ 
In order to rigorously define the operator $-\partial_x^2+iv\partial_x+\gamma\delta(x)$, we first introduce the symmetric operator $(A,D(A))$ defined as 
\begin{equation}
    A = -\partial_x^2+iv\partial_x, \qquad   \qquad D(A) = \{f \in H^2(\mathbf{R}) \ | \ f(0) =0\}.
\end{equation}
The operator $(A,D(A))$ is symmetric, but not self-adjoint in $L^2(\mathbf{R})$, since $D(A)$ does not contain the domain of the adjoint operator. The operator $-\partial_x^2+iv\partial_x+\gamma\delta(x)$ is then defined as the self-adjoint extension of $(A,D(A))$ in $L^2(\mathbf{R})$, associated with the parameter $\gamma \in \mathbf{R}\backslash\{0\}$.
\begin{proposition}
    The self-adjoint extensions of $(A,D(A))$,  different from $(H_0,D(H_0))$ itself, are labeled in a bijective way by a real parameter $\gamma \in \mathbf{R}\backslash\{0\}$. They are given by operators
    \begin{equation*}
        H_{\gamma}: D(H_{\gamma}) \to L^2(\mathbf{R}),
    \end{equation*}
    where
    \begin{equation}
            D(H_{\gamma}) = \{u \in L^2(\mathbf{R}) \ | \ u = w^{\lambda} + qG_v^{\lambda}, \ w^{\lambda} \in H^2(\mathbf{R}), \ q \in \mathbf{C}, \  \ \text{with} \ \  w^{\lambda}(0) = -  (\gamma^{-1}+G_v^{\lambda}(0))q  \}
            \label{bozza_domain_1}
    \end{equation}  
    and where the action is defined by
    \begin{equation*}
        (H_{\gamma}+\lambda)u = (H_0+\lambda)w^{\lambda},
    \end{equation*}
    the definition being $\lambda$-independent. For $\gamma >0$ the spectrum is $\sigma(H_{\gamma}) = \sigma_c(H_{\gamma}) = [-\frac{v^2}{4}, +\infty)$. If we denote $\rho(H_{\gamma}) =  \mathbf{C}\backslash\sigma(H_{\gamma})$, then for $z \in \rho(H_{\gamma})$ the integral kernel of the resolvent of $H_{\gamma}$ is
    \begin{equation}
        (H_{\gamma}-z)^{-1}(x,y) = (H_0-z)^{-1}(x-y) - \frac{\gamma}{1+\gamma/2k}G_{-v}^{-z}(y)G_v^{-z}(x)
        \label{bozza_resolvent}
    \end{equation}
    where $k^2 = -z-\frac{v^2}{4}$ with $\Re k >0$.
    \label{bozza_prop_sa_extension}
\end{proposition}
In order to prove Proposition \ref{bozza_prop_sa_extension} we follow the same steps as for the construction of the singular perturbations of the Laplace operator $-\partial_x^2$ in \cite[Chapter 8]{teta2018primer}. We postpone the proof to Appendix \ref{bozza_appendix} and we just collect some remarks below. Recall $H_0$ has been defined in \eqref{bozza_def_H0}.
\begin{remark}
\hfill
    \begin{itemize}
    \item[i)] Each function $u$ in $D(H_{\gamma})$ is made of a regular component $w^{\lambda} \in H^2(\mathbf{R})$ plus a singular component $G^{\lambda}_v \in H^1(\mathbf{R})$. The constraint $w^{\lambda}(0)=(-\gamma+G^{\lambda}(0))q$, relates the variable $w^{\lambda}(0)$ and $q$, and represents a boundary condition for $u$ in zero. This condition can be formulated as
    \begin{equation*}
        \partial_xu(0^+)-\partial_xu(0^-) = \gamma u(0),
    \end{equation*}
    and the domain can be equivalently written as
    \begin{equation}
        D(H_{\gamma}) = \{u \in H^1(\mathbf{R}) \cap H^2(\mathbf{R}\backslash\{0\})\ \ | \ \partial_xu(0^+)-\partial_xu(0^-) = \gamma u(0)\};
        \label{bozza_domain_2}
    \end{equation}
    \item [ii)]if $u(0)=0$, then $q=0$. This implies $u \in H^2(\mathbf{R})$ and $H_{\gamma}u=H_0u$;
    \\
    \item [iii)]since $(H_0+\lambda)G_v^{\lambda}=0$ in $\mathbf{R}\backslash\{0\}$, we have    $$\int_{\mathbf{R}}\overline{\phi}(x)(H_{\gamma}u)(x)dx = \int_{\mathbf{R}}(\overline{-\partial_x^2\phi(x)+iv\partial_x\phi(x)})u(x)dx \qquad \forall\phi \in C_0^{\infty}(\mathbf{R}\backslash\{0\}),$$
    or, equivalently, $H_{\gamma}u = -\partial_x^2u+iv\partial_xu$ a.e. in $\mathbf{R}\backslash\{0\}$;
    \\
    \item [iv)]using the point above we have, for any $g \in H^1(\mathbf{R})$,
    \begin{equation*}
        \begin{split}
            \int_{\mathbf{R}}\overline{g}(x)(H_{\gamma}u)(x)dx &= \lim_{\varepsilon \to 0}\int_{|x|\geq \varepsilon}\overline{g}(x)\bigl(-\partial_x^2u(x)+iv\partial_xu(x)\bigr)dx \\ & = \int_{\mathbf{R} }\bigl[\partial_x\overline{g}(x)\partial_xu(x)+iv\overline{g}(x)\partial_xu(x)\bigr] \ dx + \lim_{\varepsilon \to 0} \ \bigl(\overline{g}(\varepsilon)\partial_xu(\varepsilon)-\overline{g}(-\varepsilon)\partial_xu(-\varepsilon)\bigr) \\ & = \int_{\mathbf{R}}\bigl[\partial_x\overline{g}(x)\partial_xu(x) +iv\overline{g}(x)\partial_xu(x)\bigr] \ dx + \gamma \overline{g}(0)u(0).
        \end{split}
    \end{equation*} 
\end{itemize}
\label{bozza_remark_sa}
\end{remark}
We conclude with the following observation. In the case $v=0$, the operator $(H_{\gamma},D(H_{\gamma}))$ reduces to the singularly perturbed Laplace operator, which we denote as $(H^0_{\gamma},D(H^0_{\gamma}))$ (see \cite{albeverio_solvable, teta2018primer}). The latter is defined by its domain 
\begin{equation*}
    D(H^0_{\gamma}) := D(H_{\gamma}) = \{u \in H^1(\mathbf{R}) \cap H^2(\mathbf{R}\backslash\{0\})\ \ | \ \partial_xu(0^+)-\partial_xu(0^-) = \gamma u(0)\},
\end{equation*}
and its action on the elements $u \in D(H_{\gamma}^0)$: $$H^0_{\gamma}u := -\partial_x^2u  \quad \text{a.e. in} \quad   \mathbf{R}\backslash\{0\}.$$ Using point iv) of Remark \ref{bozza_remark_sa}, we can study the relation between the operator $H^{0}_{\gamma}$ and $H_{\gamma}$ as follows. Given $u \in D(H_{\gamma})$ we have, for any $g \in H^1(\mathbf{R})$, 
    \begin{equation*}
    \int_{\mathbf{R}}\overline{g}(x)(H_{\gamma}u)(x)dx =  \int_{\mathbf{R}}\Bigl[\partial_x\overline{g}(x)\partial_xu(x) +iv\overline{g}(x)\partial_xu(x)\Bigr]dx + \gamma \overline{g}(0)u(0) =  \int_{\mathbf{R}}\overline{g}(x)(H^0_{\gamma}u+iv\partial_xu)(x)\ dx.
\end{equation*}
We deduce that the identity $H_{\gamma}u = H^0_{\gamma}u + iv\partial_x u$ in $L^2(\mathbf{R})$ holds for any $u \in D(H_{\gamma})$.
\section{\texorpdfstring{The unitary group generated by $H_{\gamma}$}{The unitary group generated by H gamma}}
\label{section:propagator}
In this section we study the unitary group generated by the self-adjoint operator $H_{\gamma}$ in $L^2(\mathbf{R})$. In particular, we determine its explicit expression and we study some of its properties. These results will be useful in Section \ref{section:linear}, where we extend the action of the unitary group to elements of $\mathcal{E}$, and in Section \ref{section:cauchy_problem}, where we study the Duhamel formulation of \eqref{bozza_NLS}.\\

The self-adjoint operator $H_{\gamma}$, defined in Proposition \ref{bozza_prop_sa_extension}, generates a unitary group in $L^2(\mathbf{R})$, which we denote $e^{-itH_{\gamma}}$. If $u_0 \in D(H_{\gamma})$, then the function $u: t \to e^{-itH_{\gamma}}u_0$ belongs to $C^0(\mathbf{R}, D(H_{\gamma}))\cap C^1(\mathbf{R}, L^2(\mathbf{R}))$ and it is the unique solution to
\begin{equation*}
    \begin{cases}
        i\partial_tu=H_{\gamma}u, \quad \forall t \in \mathbf{R}\\
        u(0) = u_0.
    \end{cases}
\end{equation*}
Following \cite{albeverio1995fundamental}, we can compute the integral kernel of the operator $e^{-itH_{\gamma}}$ from the resolvent kernel in \eqref{bozza_resolvent}. Firstly, we compute the inverse Laplace transform of the map $\omega \to(H_{\gamma}+\omega)^{-1}(x,y)$, for $\Re \omega > \frac{v^2}{4}$. In this way we obtain $P^{\gamma}(\xi;x,y)$, defined for $\xi=t > 0$, which corresponds to the integral kernel of the semigroup $e^{-tH_{\gamma}}$, $t\geq 0$. Then, we analytically continue $P^{\gamma}(\xi;x,y)$ to $\xi=it$, $t>0$ and we obtain in this way the integral kernel of the unitary operator $e^{-itH_{\gamma}}$, $t\in\mathbf{R}$. \\\\
We begin by noticing that the integral kernel of the resolvent operator $(H_{\gamma}+\omega)^{-1}$, defined in (\ref{bozza_resolvent}) is the sum of two components. The first component is given by the kernel of free resolvent $(H_{0}+\omega)^{-1}$, while the second component is a rank one perturbation. By linearity of the inverse Laplace transform, we compute the two contributions separately. For the first component we proceed as follows. We consider the \textit{heat kernel}, which we denote as $$K(t;x) = \frac{1}{\sqrt{4\pi t}}e^{-\frac{|x|^2}{4t}}.$$ The Laplace transform $\mathcal{L}$ of the map $(0,+\infty) \ni t \to K(t;x)$ is the integral kernel of the resolvent $(-\partial_x^2+\omega)^{-1}$, i.e. 
\begin{equation*}
    \mathcal{L}\Bigl(\frac{1}{\sqrt{4\pi t}}e^{-\frac{|x|^2}{4t}}\Bigr)(\omega) = \int_{0}^{\infty}e^{-\omega t} \frac{1}{\sqrt{4\pi t}}e^{-\frac{|x|^2}{4t}} dt = \frac{e^{-\sqrt{\omega}|x|}}{2\sqrt{\omega}} = (-\partial_x^2+\omega)^{-1}(x) \qquad \text{for} \quad \Re \omega >0.
\end{equation*}
By the linearity and frequency-shift properties of the Laplace transform, we have
\begin{equation*}
    \mathcal{L}\Bigl(e^{\frac{v^2}{4}t}e^{i\frac{v}{2}x}\frac{1}{\sqrt{4\pi t}}e^{-\frac{|x|^2}{4t}}\Bigr)(\omega) = \frac{e^{-\sqrt{\omega-\frac{v^2}{4}}|x|}}{2\sqrt{\omega-\frac{v^2}{4}}} e^{i\frac{v}{2}x} = (H_0+\omega)^{-1}(x) \qquad \text{for} \quad \Re \omega >\frac{v^2}{4}, 
\end{equation*}
where $(H_0+\omega)^{-1}(x)$ is defined in \eqref{bozza_green}. We conclude that the inverse Laplace transform of the map $\omega \to (H_0+\omega)^{-1}(x)$ is
\begin{equation*}
    (0,+\infty) \ni t \to  e^{\frac{v^2}{4}t}e^{i\frac{v}{2}x} K(t;x).
\end{equation*}
We are now left with the second component in \eqref{bozza_resolvent}.
We use following identity
\begin{equation*}
    -\frac{\gamma}{2k(2k+\gamma)}e^{i\frac{v}{2}(x-y)}e^{-k|x|}e^{-k|y|} = -\frac{\gamma}{2} e^{i\frac{v}{2}(x-y)}\int_{0}^{\infty}e^{-\gamma s/2}\frac{1}{2k}e^{-k(s+|x|+|y|)} ds,
\end{equation*}
where $k = \sqrt{\omega-\frac{v^2}{4}}$, $\Re k >0$. Using the expression above, we obtain as the inverse Laplace transform of resolvent kernel in (\ref{bozza_resolvent}) the map
\begin{equation}
    P^{\gamma}(t;x,y) = e^{\frac{v^2}{4}t}e^{i\frac{v}{2}(x-y)} K(t;x-y) - \frac{\gamma}{2} e^{\frac{v^2}{4}t} e^{i\frac{v}{2}(x-y)}\int_{0}^{\infty}e^{-\gamma s/2} K(t;s+|x|+|y|)ds, \qquad \text{for} \quad t \in (0,+\infty).
    \label{bozza_real_time_prop}
\end{equation}
When $\gamma >0$, the function $P^{\gamma}(\xi;x,y)$ with $\xi=t>0$, can be analytically continued to $\xi=it$, for $t>0$, since the integral on the right-hand side of (\ref{bozza_real_time_prop}) is absolutely convergent. This is due to the term $e^{-\gamma s/2}$, which dominates the oscillatory term in $K(it;s+|x|+|y|)$). This continuation provides the integral kernel of $e^{-itH_{\gamma}}$, which reads
\begin{equation*}
    P^{\gamma}(it;x,y) = e^{i\frac{v^2}{4}t}e^{i\frac{v}{2}(x-y)} K(it;x-y) - \frac{\gamma}{2} e^{i\frac{v^2}{4}t} e^{i\frac{v}{2}(x-y)}\int_{0}^{\infty}e^{-\gamma s/2} K(it;s+|x|+|y|)ds, \qquad \text{for} \quad t \in \mathbf{R}.
\end{equation*}
\begin{proposition}
Let $\gamma \geq 0$. The unitary group $e^{-itH_{\gamma}} $ in $ L^2(\mathbf{R})$, generated by the self-adjoint operator $(H_{\gamma},D(H_{\gamma}))$, can be written in the following way
    \begin{equation}
        e^{-itH_{\gamma}} = e^{-itH_{0}} + \Gamma_v(t).
        \label{bozza_propagator_decomposition}
    \end{equation}
    Here $e^{-itH_{0}}$ is the unitary group generated by $H_0 = -\partial_x^2+iv\partial_x$, and its integral kernel is
    \begin{equation}
        P^0(it;x,y) = e^{i\frac{v^2}{4}t}e^{i\frac{v}{2}(x-y)} K(it;x-y).
    \end{equation}
    The operator $\Gamma_v(t)$, instead, is defined by the integral kernel
    \begin{equation}
        \Gamma_v(t,x,y)=-\frac{\gamma}{2} e^{i\frac{v^2}{4}t} e^{i\frac{v}{2}(x-y)}\int_{0}^{\infty}e^{-\gamma s/2} K(it;s+|x|+|y|)ds.
    \end{equation}
    \label{bozza_prop_propagator}
\end{proposition}
In the following, we will write $\Gamma_v(t,x,y) = e^{i\frac{v^2}{4}t} e^{i\frac{v}{2}(x-y)}\Gamma(t,x,y),$ where
\begin{equation}
    \Gamma(t,x,y) := -\frac{\gamma}{2}\int_{0}^{\infty}e^{-\gamma s/2} K(it;s+|x|+|y|)ds.
    \label{bozza_gamma_def}
\end{equation}
\subsection{Some properties of the unitary group} 
In this subsection we study some properties of the unitary group $e^{-itH_{\gamma}}$ and, in particular, its action on elements of $H^1(\mathbf{R})$. By using the decomposition in Proposition \ref{bozza_prop_propagator}, we study the contributions of $e^{-itH_{0}}$ and $\Gamma_v(t)$ separately.\\
We begin by introducing the following operators:
\begin{equation*}
   \begin{split}
        &L_x(t,s,x,y)= -\frac{2it\, \sign(x)}{(s+|x|+|y|)} \frac{\partial}{\partial x} \\ &
         L_y(t,s,x,y)= -\frac{2it\, \sign(x)}{(s+|x|+|y|)} \frac{\partial}{\partial y} \\ & 
          L_s(t,s,x,y)= -\frac{2it}{(s+|x|+|y|)} \frac{\partial}{\partial s}.
   \end{split}
\end{equation*}
defined for $ x,y \neq 0$ and $s \geq 0$. For these operators the following property holds 
\begin{equation}
    L_xK(it;s+|x|+|y|) = K(it;s+|x|+|y|),
    \label{bozza_L_x}
\end{equation}
and similarly for $L_y$ and $L_s$. They will be useful in obtaining powers of $t$ and negative powers of $x,y$ in the estimates on the unitary group. Another useful identity is the following
\begin{equation}
    \partial_xK(it;s+|x|+|y|) = \sign(x) \sign(y) \partial_yK(it;s+|x|+|y|)  \qquad \text{for} \quad x,y \neq 0.
    \label{bozza_sign}
\end{equation}

The following Lemma studies the action of the operator $\Gamma_v(t) = e^{-itH_{\gamma}}-e^{-itH_0}$ in $H^1(\mathbf{R})$.
\begin{lemma}
    Let $T>0$. There exists $C>0$ such that for $t \in [-T,T]$ and $w \in H^1(\mathbf{R})$ we have $\Gamma_v(t)w \in H^1(\mathbf{R})$ and
    \begin{equation}
        ||\Gamma_v(t)w||_{H^1(\mathbf{R})} \leq C ||w||_{H^1(\mathbf{R})}.
    \end{equation}
    Moreover, the map $t \to \Gamma_v(t)w$ is continuous from $\mathbf{R}$ to $H^1(\mathbf{R})$.
    \label{bozza_H1_Gamma}
\end{lemma}
\begin{proof}
    By a density argument, consider $\phi \in C^{\infty}_0(\mathbf{R})$. Using the decomposition in \eqref{bozza_propagator_decomposition}, we see that the map $t \to \Gamma_v(t)\phi$ is continuous from $\mathbf{R}$ to $L^2(\mathbf{R})$ and we have $||\Gamma_v(t)\phi||_{L^2(\mathbf{R})}\leq 2 ||\phi||_{L^2(\mathbf{R})}$.\\ For $x \in \mathbf{R} \backslash\{0\}$, we write
    \begin{equation}
        \partial_x(\Gamma_v(t)\phi)(x) = \int_{\mathbf{R}}\partial_x\Gamma_v(t,x,y)\phi(y)dy =   \int_{\mathbf{R}}\Bigl[e^{i\frac{v^2}{4}t}e^{i\frac{v}{2}(x-y)}\partial_x\Gamma(t,x,y)\phi(y)+i\frac{v}{2}\Gamma_v(t,x,y)\phi(y) \Bigr]dy. 
        \label{bozza_proof_prop_H1}
    \end{equation}
   The map $t \to i\frac{v}{2}\Gamma_v(t)\phi$ is again a continuous function from $\mathbf{R}$ to $L^2(\mathbf{R})$, bounded by $|v|\bigl|\bigl|\phi\bigl|\bigl|_{L^2(\mathbf{R})}$. Then we have, using \eqref{bozza_gamma_def} and \eqref{bozza_sign},
   \begin{equation*}
       \begin{split}
           &\int_{\mathbf{R}}e^{i\frac{v^2}{4}t}e^{i\frac{v}{2}(x-y)}\partial_x\Gamma(t,x,y)\phi(y)dy  =  \sign(x)\int_{\mathbf{R}}e^{i\frac{v^2}{4}t}e^{i\frac{v}{2}(x-y)}\sign(y)\partial_y\Gamma(t,x,y)\phi(y)dy \\ & = -\sign(x)[\Gamma_v(t)(\sign(y)\partial_y\phi)](x) + i\frac{v}{2}\sign(x)[\Gamma_v(t)\sign(y)\phi](x) - 2\sign(x)\Gamma_v(t,x,0)\phi(0)
       \end{split} 
   \end{equation*}
    The first two terms in the line above are continuous functions from $t \in \mathbf{R}$ to $L^2(\mathbf{R})$. Their $L^2(\mathbf{R})$-norm is bounded by $2||\partial_y\phi||_{L^2(\mathbf{R})}$ and $|v|||\phi||_{L^2(\mathbf{R})}$, respectively. For the last term we use the bound  $|\phi(0)|\lesssim ||\phi||_{H^1(\mathbf{R})}$. To conclude the proof, we show that the map $t \to \Gamma_v(t,\cdot,0)$ is continuous from $\mathbf{R}$ to $L^2(\mathbf{R})$. This has been essentially proven in \cite[Proposition 3.5]{LeCozIR} , and we report the proof for the sake of completeness.\\ Using the identity $\Gamma_v(-t,\cdot,0) = \overline{\Gamma_{-v}(t,\cdot,0)}$, it's sufficient to show the continuity of $t \to \Gamma_v(t,\cdot,0)$ in $[0,+\infty)$. Set $t>0$ and consider $\beta \geq 0$. Applying the operator $L_s(t,s,x,0)$ and integrating by parts, we have the identity
    \begin{equation*}
        \frac{\gamma}{2}\int_{\beta}^{+\infty}e^{-\frac{\gamma s}{2}}K(it;s+|x|) ds = \frac{i\gamma t}{|x|+\beta}e^{-\frac{\gamma \beta}{2}}K(it;\beta+|x|) + i \gamma t \int_{\beta}^{+\infty}K(it;s+|x|)\partial_s\Bigl(\frac{e^{-\frac{\gamma s}{2}}}{s+|x|}\Bigr)ds.
    \end{equation*}
    We conclude 
    \begin{equation}
        |\Gamma_v(t,x,0)| \lesssim \frac{\beta}{\sqrt{t}} + \frac{\sqrt{t}}{\beta + |x|}
        \label{bozza_bound_Gamma}
    \end{equation}
    If we apply the inequality above with $\beta =1$ for $|x| <1$ and with $\beta = 0$ for $|x|\geq 1$, we obtain $\Gamma_v(t,\cdot,0) \in L^2(\mathbf{R})$ for any $t>0$. Moreover, for any $x \in \mathbf{R}$, the map $t \to \Gamma_v(t,x,0)$ is continuous on $(0,+\infty)$, and the bound (\ref{bozza_bound_Gamma}) is uniform for $t$ in a compact subset of $(0,+\infty)$. By the dominated convergence theorem, $t \to \Gamma_v(t,\cdot, 0)$ is continuous from $(0,+\infty)$ to $L^2(\mathbf{R})$. Finally, we write $$||\Gamma_v(t, \cdot, 0)||^2_{L^2(\mathbf{R})} = \int_{|x|\leq \sqrt{t}}|\Gamma_v(t,x,0)|^2dx+\int_{|x|\geq \sqrt{t}}|\Gamma_v(t,x,0)|^2dx.$$
    If we use the bound (\ref{bozza_bound_Gamma}), with $\beta = \sqrt{t}$ in the first integral, and with $\beta =0 $ in the second, we obtain
    \begin{equation*}
        ||\Gamma_v(t, \cdot, 0)||^2_{L^2(\mathbf{R})} \lesssim \sqrt{t} + t \int_{|x|\geq \sqrt{t}}\frac{1}{|x|^2}dx \lesssim \sqrt{t} \to 0, \qquad \text{as} \quad t\to 0, 
    \end{equation*}
    which implies the continuity of $t \to \Gamma_v(t,\cdot,0)$ from $[0,+\infty)$ to $L^2(\mathbf{R})$.
\end{proof}
Using the fact that the free propagator $e^{-itH_0}$ is an isometry on $H^1(\mathbf{R})$, we can conclude the following Corollary.
\begin{corollary} The following assertions hold.
\begin{itemize}
    \item [i)] Let $w \in H^1(\mathbf{R})$. Then the map $t \to e^{-itH_{\gamma}}w$ is continuous from $\mathbf{R}$ to $H^1(\mathbf{R})$;
    \item [ii)] Let $T>0$. There exists $C >0$ such that for any $w \in H^1(\mathbf{R})$ and any $t \in [-T,T]$ it holds
    \begin{equation*}
        ||e^{-itH_{\gamma}}w||_{H^1(\mathbf{R})} \leq C ||w||_{H^1(\mathbf{R})}.
    \end{equation*}
\end{itemize}
\label{bozza_corollary_H1_propag}
\end{corollary}
We conclude this section with the following pointwise estimate on $\Gamma_v(t)$, which will be useful in extending the action of the unitary group $e^{-itH_{\gamma}}$ to elements of $\mathcal{E}$.
\begin{lemma}
    Let $\phi \in \mathcal{S}$ and $t \in \mathbf{R}$. There exists $C>0$, which only depends on $t$ and on some seminorm of $\phi$ in $\mathcal{S}$, such that
    \begin{equation*}
        |(\Gamma_v(t)\phi)(x)| \leq C \langle x\rangle^{-2}
    \end{equation*}
    \label{bozza_lemma_decay}
\end{lemma}
\begin{proof}
It is a direct consequence of Lemma 3.3 in \cite{LeCozIR}. There it is proven that, given $\phi \in \mathcal{S}$ and $t \in \mathbf{R}$, it holds $$|(\Gamma(t)\phi)(x)| \leq C \langle x \rangle^{-2},$$ where $C>0$ depends only on $t$ and on $||\phi||_{L^1(\mathbf{R})}, |\phi(0)|$ and $||\phi||_{W^{1,1}(\mathbf{R})}$. Since it holds
\begin{equation*}
    |(\Gamma_v(t)\phi)(x)| = |(\Gamma(t)(e^{-i\frac{v}{2}y}\phi))(x)|, 
\end{equation*} 
we obtain the desired result.
\end{proof}
\section{\texorpdfstring{The linear evolution of \eqref{bozza_NLS} in the space $\mathcal{E}$}{The linear evolution of (bozza\_NLS) in the space mathcal E}}
\label{section:linear}
In this section we extend the action of $e^{-itH_{\gamma}}$ to elements of $\dot{H}^1(\mathbf{R})$, hence to elements of $\mathcal{E}$, following the approach in \cite{gerard,LeCozIR}. As we will discuss, this extension defines the linear evolution of \eqref{bozza_NLS} in $\mathcal{E}$.\\

We begin by recalling that for any $u \in \dot{H}^1(\mathbf{R})$ there exists $C>0$ such that
$$|u(x)| \leq C \langle x \rangle^{1/2}.$$
For $u \in \dot{H}^1(\mathbf{R})$ and $t \in \mathbf{R}$, we define the temperate distribution $\Gamma_v(t)u \in \mathcal{S}'$ by 
\begin{equation*}
    \langle\Gamma_v(t)u, \phi \rangle_{\mathcal{S}',\mathcal{S}} = \langle u, \Gamma_v(-t)\phi\rangle \qquad \forall \phi \in \mathcal{S}.
\end{equation*}
Notice that, thanks to Lemma \ref{bozza_lemma_decay}, the pairing on the right-hand side of the expression above is well defined. Similarly, we define the linear propagator $T_{\gamma}(t)u \in \mathcal{S}'$ by\footnote{We notice the use of the complex conjugation in the definition of the distribution, as in \cite[Section 1.13]{tao}. In particular, to every function $w \in \dot{H}^1(\mathbf{R})$ we associate the tempered distribution defined by $\langle w,\phi\rangle_{\mathcal{S',\mathcal{S}}}= \int_{\mathbf{R}}\overline{w}(x)\phi(x)dx$, for all $\phi \in \mathcal{S}$. For consistency with the Plancherel's theorem, the Fourier transform  $\mathcal{F}$ of the tempered distribution $w$ is defined by $\langle\mathcal{F}w, \phi\rangle_{\mathcal{S'},\mathcal{S}} = \langle w, \mathcal{F}^{-1}\phi\rangle$, for all $ \phi \in \mathcal{S}$.}
\begin{equation*}
    \langle T_{\gamma}(t)u, \phi\rangle_{\mathcal{S}',\mathcal{S}} =  \langle u, e^{itH_{\gamma}}\phi\rangle \qquad \forall \phi \in \mathcal{S}.
\end{equation*}
If $u \in H^1(\mathbf{R})$, then $T_{\gamma}(t)u = e^{-itH_{\gamma}}u \in L^2(\mathbf{R})$ for all $t \in \mathbf{R}$. For clarity, we use a different notation for the unitary group $e^{-itH_{\gamma}}$ defined in $L^2(\mathbf{R})$ and the map $T_{\gamma}(t)$, as in \cite{LeCozIR}.\\
An immediate property that follows from these definitions is that, by duality, the following composition rule holds for $\gamma \geq 0$:
\begin{equation*}
    T_{\gamma}(t)\circ T_{\gamma}(s) = T_{\gamma}(t+s) \qquad \forall s,t \in \mathbf{R}.
\end{equation*}
Moreover, it holds the identity
\begin{equation*}
    T_{\gamma}(t) = T_0(t) + \Gamma_v(t), \qquad \forall t \in \mathbf{R},
\end{equation*}
where we denote by $T_0(t)$ the distribution which is dual to the free propagator $e^{-itH_0}$, i.e. the distribution $T_{\gamma}(t)$ evaluated at $\gamma=0$.\\
Our goal now is to show that, for any $u \in \mathcal{E}$, the map $t \to T_{\gamma}(t)u$ is continuous from $\mathbf{R}$ to $\mathcal{E}$. In order to do that, we first consider $\gamma =0$ and we study the action of $T_0(t)$ on $\dot{H}^1(\mathbf{R})$.
\begin{lemma}
        Let $T>0$. There exists $C>0$ such that for any $u \in \dot{H}^1(\mathbf{R})$ we have 
        \begin{equation}
            ||T_{0}(t)u-u||_{H^1(\mathbf{R})} \leq C ||\partial_xu||_{L^2(\mathbf{R})}, \qquad \forall t \in [ -T,T]. 
            \label{bozza_lemma_T_0_H1}
        \end{equation}
       Moreover, the map $t \to T_{0}(t)u - u$ is continuous from $\mathbf{R}$ to $H^1(\mathbf{R})$.\\
        If, in particular, $u \in \mathcal{E}$, the following properties hold.
        \begin{itemize}
        \item [i)] For all $t \in \mathbf{R}$, the distribution $T_{0}(t)u$ belongs to $\mathcal{E}$;
        \item [ii)] the map $t \to T_{0}(t)u$ is continuous from $\mathbf{R}$ to $\mathcal{E}$;
         \item [iii)] let $R\geq 0$ and $T \geq 0$. There exists $C_R>0$ such that for any $u, \tilde{u} \in \mathcal{E}$ with $|u|_{\mathcal{E}}<R$ and $|\tilde{u}|_{\mathcal{E}} < R$ we have
         \begin{equation*}
            d_{\infty}(T_{0}(t)u,T_{0}(t)\tilde{u}) \leq C_R d_{\infty}(u,\tilde{u}), \qquad  \forall t \in [-T,T].
         \end{equation*}
    \end{itemize}
    \label{bozza_lemma_T0}
\end{lemma}
\begin{proof}
    Consider $u \in \dot{H}^1(\mathbf{R})$. We have $\partial_x(T_0(t)u-u) = (T_0(t)-1)\partial_xu \in L^2(\mathbf{R})$, with
    \begin{equation}
        ||\partial_x(T_0(t)u-u)||_{L^2(\mathbf{R})} \leq C ||\partial_xu||_{L^2(\mathbf{R})} \quad \text{and} \quad ||\partial_x(T_0(t)u-u)||_{L^2(\mathbf{R})} \to 0
        \label{bozza_der_T_0}
    \end{equation}
    as $t \to 0$. We conclude that the map $t \to \partial_x(T_0(t)u-u)$ is continuous from $\mathbf{R}$ to $L^2(\mathbf{R})$. We now want to show that $T_0(t)u-u \in L^2(\mathbf{R})$. We have the following identity
    \begin{equation*}
        \begin{split}
            T_0(t)u-u  =  \mathcal{F}^{-1}(e^{-it(p^2-vp)}\hat{u}-\hat{u}) \qquad \text{in} \quad \mathcal{S}',
        \end{split}
    \end{equation*}
    which can be verified by testing against any $\phi \in \mathcal{S}$. We write 
    \begin{equation}
        T_{0}(t)u-u = \mathcal{F}^{-1}\Bigl(\frac{e^{-it(p^2-vp)}-1}{ip}\widehat{\partial_xu}\Bigr).
        \label{bozza_Tu0-u0}
    \end{equation}
    and we want to show that the map  $ p \to g_t(p):= \frac{1}{ip} (e^{-it(p^2-vp)}-1)$ belongs to $L^{\infty}(\mathbf{R})$ for any $t \in \mathbf{R}$. Following \cite{gerard}, we consider a smooth cut-off $\chi \in C_0^{\infty}(\mathbf{R})$ such that $\chi=1$ near the origin. If we write
    \begin{equation*}
        g_t(p) = - it\frac{(p^2-vp)}{ip}\chi\bigl(t(p^2-vp)\bigr) \int_{0}^1e^{-it(p^2-vp)s}ds + \frac{1-\chi(t(p^2-vp))}{ip}\bigl(e^{-it(p^2-vp)}-1\bigr)
    \end{equation*}
    we can conclude that, for any $t \in \mathbf{R}$, $g_t \in L^{\infty}(\mathbf{R})$. Moreover, one can show that $g_t(p) = \mathcal{O}(|t|^{1/2})$ for $t$ small enough. This implies the continuity of the map $t \to T_0(t)u-u$  in $L^2(\mathbf{R})$ at $t=0$. Given $t,s \in \mathbf{R}$ we have
    \begin{equation*}
        ||T_{0}(t)u-T_{0}(s)u||_{L^2(\mathbf{R})} = ||T_{0}(t)(u-T_{0}(-s)u) + T_{0}(t-s)u - T_{0}(s)u||_{L^2(\mathbf{R})} \to 0
    \end{equation*}
    as $t \to s$, where we used the continuity at zero and the continuity of the map $t \to T_{0}(t)f \in L^2(\mathbf{R})$ for $f \in L^2(\mathbf{R})$. For inequality (\ref{bozza_lemma_T_0_H1}) we use (\ref{bozza_der_T_0}) and (\ref{bozza_Tu0-u0}), and the fact that $p \to g_t(p)$ is uniformly bounded for $t$ in compact sets. Then, point i) and ii) follow from Lemma \ref{bozza_lemma_E}. Point iii) is a consequence of Lemma \ref{bozza_lemma_E} and inequality (\ref{bozza_lemma_T_0_H1}).
\end{proof}
We turn now to the study of the operator $\Gamma_v(t)$ acting on functions in $\dot{H}^1(\mathbf{R})$.
\begin{proposition}
    Let $T>0$. There exists $C>0$ such that for any $t \in [-T,T]$ and for any $u \in \dot{H}^1(\mathbf{R})$ we have $\Gamma_v(t)u \in H^1(\mathbf{R})$ with
    \begin{equation*}
        ||\Gamma_v(t)u||_{H^1(\mathbf{R})} \leq C(||\partial_xu||_{L^2(\mathbf{R})}+|u(0)|).
    \end{equation*}
    Moreover, the map $t \to \Gamma_v(t)u$ is continuous from $\mathbf{R}$ to $H^1(\mathbf{R})$.
    \label{bozza_Gamma_dot_H1}
\end{proposition}
\begin{proof}
    First we consider $u \in \dot{H}^1(\mathbf{R})$ such that $u$ vanishes on $[-1,1]$. Take $t \in \mathbf{R} \backslash \{0\}$ and $\phi \in C^{\infty}_0(\mathbf{R}\backslash\{0\})$. We compute
    \begin{equation}
       \begin{split}
            \langle u, \Gamma_v(-t)\phi \rangle &= \int_{|x| \geq 1} \overline{u}(x) \int_{\mathbf{R}}\Gamma_v(-t,x,y)\phi(y)dy dx =\int_{|x| \geq 1} \overline{u}(x) \int_{\mathbf{R}}e^{-i\frac{v^2}{4}t}e^{i\frac{v}{2}(x-y)}\Gamma(-t,x,y)\phi(y)dy dx \\ & = -\frac{\gamma}{2}\int_{|x| \geq 1} \overline{u}(x) \int_{\mathbf{R}} \int_{0}^{+\infty} e^{-i\frac{v^2}{4}t}e^{i\frac{v}{2}(x-y)} e^{-\frac{\gamma}{2}s} K(-it;s+|x|+|y|)\phi(y)dsdydx
       \end{split}
       \label{bozza_A1+A2+A3}
    \end{equation}
    We apply the operator $-L_x$ to $K(-it;s+|x|+|y|)$ using \eqref{bozza_L_x} and we integrate by parts. We obtain the sum of three terms $A_1(t)+A_2(t)+A_3(t)$. The first one is obtained when $\partial_x$ acts on $u$ and it reads
    \begin{equation*}
        A_1(t) = (2it)\frac{\gamma}{2}\int_{|x| \geq 1} \int_{\mathbf{R}} \int_{0}^{+\infty} \frac{\partial_x\overline{u}(x)\sign(x)}{s+|x|+|y|}e^{-i\frac{v^2}{4}t}e^{i\frac{v}{2}(x-y)} e^{-\frac{\gamma}{2}s} K(-it;s+|x|+|y|)\phi(y)dsdydx.
    \end{equation*}
    If we apply $-L_s$ to $K(-it;s+|x|+|y|)$ and we integrate by parts, we obtain
    \begin{equation*}
       \begin{split}
            A_1(t) = &+2(2it)^2\frac{\gamma}{2}\int_{|x| \geq 1} \int_{\mathbf{R}} \int_{0}^{+\infty} \frac{\partial_x\overline{u}(x)\sign(x)}{(s+|x|+|y|)^3}e^{-i\frac{v^2}{4}t}e^{i\frac{v}{2}(x-y)} e^{-\frac{\gamma}{2}s} K(-it;s+|x|+|y|)\phi(y)dsdydx \\ & + (2it)^2\frac{\gamma^2}{4}\int_{|x| \geq 1} \int_{\mathbf{R}} \int_{0}^{+\infty} \frac{\partial_x\overline{u}(x)\sign(x)}{(s+|x|+|y|)^2}e^{-i\frac{v^2}{4}t}e^{i\frac{v}{2}(x-y)} e^{-\frac{\gamma}{2}s} K(-it;s+|x|+|y|)\phi(y)dsdydx \\ & -(2it)^2\frac{\gamma}{2}\int_{|x| \geq 1} \int_{\mathbf{R}} \frac{\partial_x\overline{u}(x)\sign(x)}{(|x|+|y|)^2}e^{-i\frac{v^2}{4}t}e^{i\frac{v}{2}(x-y)}K(-it;|x|+|y|)\phi(y)dydx.
       \end{split}
    \end{equation*}
    We have $$|A_1(t)| \lesssim t^{3/2}\int_{|x|\geq 1} \int_{\mathbf{R}}\frac{|\partial_xu(x)||\phi(y)|}{(|x|+|y|)^2}dydx \lesssim t^{3/2}||\partial_xu||_{L^2(\mathbf{R})}||\phi||_{L^2(\mathbf{R})}.$$  
    The term $A_2(t)$, instead, arises from (\ref{bozza_A1+A2+A3}) when $\partial_x$ acts on $(s+|x|+|y|)^{-1}$. It reads
    \begin{equation*}
        A_2(t)= -(2it)\frac{\gamma}{2}\int_{|x| \geq 1} \int_{\mathbf{R}} \int_{0}^{+\infty} \frac{\overline{u}(x)\sign(x)}{(s+|x|+|y|)^2}e^{-i\frac{v^2}{4}t}e^{i\frac{v}{2}(x-y)} e^{-\frac{\gamma}{2}s} K(-it;s+|x|+|y|)\phi(y)dsdydx.
    \end{equation*}
    If we apply $-L_s$ to $K(-it;s+|x|+|y|)$ and we integrate by parts, we obtain
    \begin{equation*}
        \begin{split}
            A_2(t) = &-3(2it)^2\frac{\gamma}{2}\int_{|x| \geq 1} \int_{\mathbf{R}} \int_{0}^{+\infty} \frac{\overline{u}(x)\sign(x)}{(s+|x|+|y|)^4}e^{-i\frac{v^2}{4}t}e^{i\frac{v}{2}(x-y)} e^{-\frac{\gamma}{2}s} K(-it;s+|x|+|y|)\phi(y)dsdydx \\ & - (2it)^2\frac{\gamma^2}{4}\int_{|x| \geq 1} \int_{\mathbf{R}} \int_{0}^{+\infty} \frac{\overline{u}(x)\sign(x)}{(s+|x|+|y|)^3}e^{-i\frac{v^2}{4}t}e^{i\frac{v}{2}(x-y)} e^{-\frac{\gamma}{2}s} K(-it;s+|x|+|y|)\phi(y)dsdydx \\ & + (2it)^2\frac{\gamma}{2}\int_{|x| \geq 1} \int_{\mathbf{R}} \frac{\overline{u}(x)\sign(x)}{(|x|+|y|)^3}e^{-i\frac{v^2}{4}t}e^{i\frac{v}{2}(x-y)} K(-it;|x|+|y|)\phi(y)dydx.
        \end{split}
    \end{equation*}
    We conclude \begin{equation*}
        \begin{split}
            |A_2(t)| &\lesssim t^{3/2} \int_{|x| \geq 1} \int_{\mathbf{R}} \frac{|u(x)||\phi(y)|}{(|x|+|y|)^3} dy dx\\ & \lesssim t^{3/2} \Bigl(\int_{|x| \geq 1} \frac{|u(x)|^2}{|x|^2}dx\int_{\mathbf{R}}\frac{1}{(1+|y|)^2}dy\Bigr)^{1/2}\Bigl(\int_{|x| \geq 1} \frac{1}{|x|^2}dx\int_{\mathbf{R}}|\phi|^2dy\Bigr)^{1/2}  \\ & \lesssim t^{3/2} ||\partial_xu||_{L^2(\mathbf{R})} ||\phi||_{L^2(\mathbf{R})},
        \end{split}
    \end{equation*}
    where in the last step we used Hardy's inequality 
    $$\int_{\mathbf{R}} \frac{|u(x)|^2}{|x|^2}dx  \lesssim ||\partial_xu||^2_{L^2(\mathbf{R})} $$
    which holds for $u$ absolutely continuous and vanishing in $x =0$ (see \cite{laptev}, Theorem 2.64).\\
    The term $A_3(t)$ arises from (\ref{bozza_A1+A2+A3}) when $\partial_x$ acts on $e^{i\frac{v}{2}(y-x)}$. It reads
    \begin{equation*}
        A_3(t) = \frac{iv}{2}(2it)\frac{\gamma}{2}\int_{|x|\geq 1} \int_{\mathbf{R}} \int_{0}^{+\infty} \frac{\overline{u}(x) \sign(x)}{s+|x|+|y|}e^{i\frac{v}{2}(x-y)}e^{-i\frac{v^2}{4}t}e^{-\frac{\gamma}{2}s}K(-it;s+|x|+|y|)\phi(y)dsdydx.
    \end{equation*}
    In this case, we have only gained one power of $x,y$ at the denominator. Thus, we apply again the operator $-L_x$ to $K(-it;s+|x|+|y|)$ and we integrate by parts. We obtain the sum of three terms $A_3(t) = A_{3a}(t)+A_{3b}(t)+A_{3c}(t)$, which we treat as before. We have
    \begin{equation*}
        A_{3a}(t) = -\frac{iv}{2}(2it)^2\frac{\gamma}{2} \int_{|x|\geq 1} \int_{\mathbf{R}} \int_{0}^{+\infty} \frac{\partial_x\overline{u}(x)}{(s+|x|+|y|)^2} e^{i\frac{v}{2}(x-y)}e^{-i\frac{v^2}{4}t}e^{-\frac{\gamma}{2}s}K(-it;s+|x|+|y|)\phi(y)dsdydx,
    \end{equation*}
    and
    \begin{equation*}
        A_{3b}(t) = 2\frac{iv}{2}(2it)^2\frac{\gamma}{2} \int_{|x|\geq 1} \int_{\mathbf{R}} \int_{0}^{+\infty} \frac{\overline{u}(x)\sign(x)}{(s+|x|+|y|)^3} e^{i\frac{v}{2}(x-y)}e^{-i\frac{v^2}{4}t}e^{-\frac{\gamma}{2}s}K(-it;s+|x|+|y|)\phi(y)dsdydx,
    \end{equation*}
    and
    \begin{equation*}
        A_{3c}(t) = -(vt)^2\frac{\gamma}{2} \int_{|x|\geq 1} \int_{\mathbf{R}} \int_{0}^{+\infty} \frac{\overline{u}(x)}{(s+|x|+|y|)^2} e^{i\frac{v}{2}(x-y)}e^{-i\frac{v^2}{4}t}e^{-\frac{\gamma}{2}s}K(-it;s+|x|+|y|)\phi(y)dsdydx.
    \end{equation*}
    We estimate $A_{3a}(t)$ as we did for $A_1(t)$, and $A_{3b}(t)$ as for $A_2(t)$. For the term $A_{3c}(t)$ we apply $-L_s$ to $K(-it;s+|x|+|y|)$ and we proceed as for $A_2(t)$. We obtain
    \begin{equation*}
    \begin{split}
        |A_{3a}(t)| \lesssim t^{3/2}||\partial_xu||&_{L^2(\mathbf{R})} ||\phi||_{L^2(\mathbf{R})}, \qquad |A_{3b}(t)| \lesssim t^{3/2}||\partial_xu||_{L^2(\mathbf{R})} ||\phi||_{L^2(\mathbf{R})},\\ & |A_{3c}(t)| \lesssim t^{5/2}||\partial_xu||_{L^2(\mathbf{R})} ||\phi||_{L^2(\mathbf{R})}.
    \end{split}
    \end{equation*}
    We conclude that $\langle u, \Gamma_v(-t)\phi \rangle \lesssim ||\partial_xu||_{L^2(\mathbf{R})} ||\phi||_{L^2(\mathbf{R})}$. By means of Riesz theorem and the fact that $ C^{\infty}_0(\mathbf{R}\backslash\{0\})$ is dense in $L^2(\mathbf{R})$, the distribution $\Gamma_v(t)u$ defines an element in $L^2(\mathbf{R})$, whose norm satisfies $||\Gamma_v(t)u||_{L^2(\mathbf{R})}\lesssim||\partial_xu||_{L^2(\mathbf{R})}$. In particular, we have $||\Gamma_v(t)u||_{L^2(\mathbf{R})} \to 0$ as $t \to 0$, which implies the continuity of the map $t \to \Gamma_v(t)u$ in $L^2(\mathbf{R})$ at $t=0$.\\
    Now we want to show the continuity of $t \to \Gamma_v(t)u$ in $L^2(\mathbf{R})$ for $\forall t \in \mathbf{R}$. Recall that the group composition property doesn't hold for $\Gamma_v(t)$, as it is the difference of two propagators, i.e. $\Gamma_v(t) = T_{\gamma}(t)-T_{0}(t)$. We begin with the following observation. Using Lemma \ref{bozza_lemma_T0} and the calculations above, we have that $T_{\gamma}(t)u -u = T_{0}(t)u - u +\Gamma_v(t)u \in L^2(\mathbf{R})$ for any $t \in \mathbf{R}$. Moreover, for $f \in L^2(\mathbf{R})$ and $\gamma \geq 0$, the map $t \to T_{\gamma}(t)f \in L^2(\mathbf{R})$ is continuous $\forall t \in \mathbf{R}$. Consider now $t,s \in \mathbf{R}$. We have
    \begin{equation*}
        \begin{split}
            ||\Gamma_v(t)u-\Gamma_v(s)u||_{L^2(\mathbf{R})}& = ||T_{\gamma}(t)u-T_{0}(t)u -\Gamma_v(s)u||_{L^2(\mathbf{R})} \\ & = ||T_{\gamma}(t)\bigl(u-T_{\gamma}(-s)u\bigr) - T_{0}(t)\bigl(u-T_{0}(-s)u\bigr) + \Gamma_v(t-s)u -\Gamma_v(s)u ||_{L^2(\mathbf{R})}\\ & 
            \to ||T_{\gamma}(s)u - T_0(s)u - \Gamma_v(s)u||_{L^2(\mathbf{R})} = 0, \qquad \text{as} \quad t \to s 
        \end{split} 
    \end{equation*}
    where in the limit above we used the continuity of $t \to \Gamma_v(t)u$ at $t=0$.\\
    Finally, we want to show that the distribution $\partial_x(\Gamma_v(t)u)$ belongs to $L^2(\mathbf{R})$. We consider again $t \in \mathbf{R}\backslash\{0\}$ and $\phi \in C^{\infty}_0(\mathbf{R}\backslash\{0\})$, and we compute
    \begin{equation*}
        \begin{split}
           & \langle u, \Gamma_v(-t)\partial_y\phi \rangle = \int_{|x| \geq 1} \overline{u}(x) \int_{\mathbf{R}}e^{-i\frac{v^2}{4}t}e^{i\frac{v}{2}(x-y)}\Gamma(-t,x,y)\partial_y\phi(y)dy dx \\ &  = - \int_{|x| \geq 1} \overline{u}(x) \int_{\mathbf{R}} \Bigl[e^{-i\frac{v^2}{4}t}e^{i\frac{v}{2}(x-y)}\partial_y\Gamma(-t,x,y)\phi(y) - i\frac{v}{2} e^{-i\frac{v^2}{4}t}e^{i\frac{v}{2}(x-y)}\Gamma(-t,x,y)\phi(y) \Bigr]dy dx \\ & = -\int_{|x| \geq 1} \overline{u}(x) \int_{\mathbf{R}} \Bigl[e^{-i\frac{v^2}{4}t}e^{i\frac{v}{2}(x-y)}\sign(x)\sign(y)\partial_x\Gamma(-t,x,y)\phi(y)dydx + i\frac{v}{2} \langle u, \Gamma_v(-t)\phi \rangle \\ & = \langle\Gamma_v(t)(\sign(x)\partial_xu),\sign(y)\phi \rangle + i\frac{v}{2} \langle \Gamma_v(t)(\sign(x)u), \sign(y)\phi \rangle + i\frac{v}{2} \langle\Gamma_v(t) u,\phi \rangle.
        \end{split}
    \end{equation*}
    Since $\sign(x)u \in \dot{H}^1(\mathbf{R})$ and vanishes on $[-1,1]$, we conclude that the map $t\to \partial_x(\Gamma_v(t)u)$ is continuous from $\mathbf{R}$ to $L^2(\mathbf{R})$ and that 
    \begin{equation*}
        ||\partial_x(\Gamma_v(t)u)||_{L^2(\mathbf{R})} \lesssim ||\partial_xu||_{L^2(\mathbf{R})}.
    \end{equation*}
    This proves the proposition for $u$ vanishing in $[-1,1]$. For $u$ generic in $\dot{H}^1(\mathbf{R})$, we consider $\chi$ a smooth cut-off, supported in $(-2,2)$ and such that $\chi(x)=1$ for $x \in [-1,1]$. Then $\chi u \in H^1(\mathbf{R})$, and $(1-\chi)u \in \dot{H}^1(\mathbf{R})$ vanishes on [-1,1]. Using Lemma \ref{bozza_H1_Gamma}, we have that the map $t \to \Gamma_v(t)\chi u+ \Gamma_v(t)(1-\chi)u$ is continuous from $\mathbf{R}$ to $H^1(\mathbf{R})$. Then, we have
    \begin{equation*}
       \begin{split}
            ||\Gamma_v(t)u||_{H^1(\mathbf{R})} & \leq ||\Gamma_v(t)\chi u||_{H^1(\mathbf{R})} + ||\Gamma_v(t)(1-\chi)u||_{H^1(\mathbf{R})} \lesssim ||\chi u||_{H^1(\mathbf{R})} + ||\partial_x((1-\chi)u)||_{L^2(\mathbf{R})} \\ & \lesssim  ||\partial_xu||_{L^2(\mathbf{R})} + ||u||_{L^{\infty}([-2,2])}\lesssim ||\partial_xu||_{L^2(\mathbf{R})} + |u(0)|
       \end{split}
    \end{equation*}
\end{proof}
We conclude with the following Corollary, which lists some properties of the propagator $T_{\gamma}(t)$ in $\mathcal{E}$.
\begin{corollary}
    Given $u \in \mathcal{E}$, the following properties hold.
    \begin{itemize}
        \item [i)] the map $t \to T_{\gamma}(t)u - u$ is continuous from $\mathbf{R}$ to $H^1(\mathbf{R})$;
        \item [ii)] let $R>0$ and $T>0$. There exists $C_R>0$ such that for any $u_0 \in \mathcal{E}$ with $|u|_{\mathcal{E}} < R$ we have
        \begin{equation*}
            ||T_{\gamma}(t)u-u||_{H^1(\mathbf{R})} \leq C_R, \qquad \forall t \in [ -T,T]; 
        \end{equation*}
        \item [iii)] For all $t \in \mathbf{R}$, the distribution $T_{\gamma}(t)u$ belongs to $\mathcal{E}$;
        \item [iv)] the map $t \to T_{\gamma}(t)u$ is continuous from $\mathbf{R}$ to $\mathcal{E}$;
         \item [v)] let $R\geq 0$ and $T \geq 0$. There exists $C_R$ such that for any $u, \tilde{u} \in \mathcal{E}$ with $|u|_{\mathcal{E}}<R$ and $|\tilde{u}|_{\mathcal{E}} < R$ we have
         \begin{equation*}
             d_{\infty}(T_{\gamma}(t)u,T_{\gamma}(t)\tilde{u}) \leq C_R d_{\infty}(u,\tilde{u}), \qquad \forall t \in [-T,T].
         \end{equation*}
    \end{itemize}
    \label{bozza_prop_Tgamma}
\end{corollary}
\begin{proof} 
    Writing $T_{\gamma}(t)u-u = T_{0}(t)u-u + \Gamma_v(t)u$, and using Lemma \ref{bozza_lemma_T0} and Proposition \ref{bozza_Gamma_dot_H1} we prove point i) and ii). Using Lemma \ref{bozza_lemma_E}, point iii) and iv) follow. Point v) also follows from Lemma \ref{bozza_lemma_E}, combined with Lemma \ref{bozza_lemma_T0} and Proposition \ref{bozza_Gamma_dot_H1}.
\end{proof}

\subsection{Linear evolution in higher regularity spaces.} In this subsection we introduce the space $X^2_{\gamma}$ of functions with higher regularity. This space is locally similar to the domain $D(H_{\gamma})$, defined in \eqref{bozza_domain_2}, since its elements have the same regularity and satisfy the same jump condition as the elements of $D(H_{\gamma})$. On the other hand, functions in $X^2_{\gamma}$ have a non-vanishing behavior at infinity; in particular, the intersection $X^2_{\gamma}\cap \mathcal{E}$ is non-empty, and it is in fact dense in $\mathcal{E}$. Following \cite{LeCozIR}, we study the linear dynamics of \eqref{bozza_NLS} in $X^2_{\gamma}$, and we show that, for $u \in X^2_{\gamma}$, the map $t \to T_{\gamma}(t)u$ is differentiable for every $t \in \mathbf{R}$. In Section \ref{section:cauchy_problem} we will use these results to prove the persistence of regularity for \eqref{bozza_NLS} in $X^2_{\gamma}\cap \mathcal{E}$. Eventually, this will lead to the proof of the conservation of energy in Section \ref{section:energy}. \\

We begin by introducing the following Zhidkov space \cite{zhidkov} $$X^2_0 = L^{\infty}(\mathbf{R}) \cap \dot{H}^1(\mathbf{R})\cap\dot{H}^2(\mathbf{R}).$$ In this space we can define the operator $\tilde{H}_0$, which consists in the natural extension of the operator $(H_0,D(H_0))$ defined in Lemma \ref{bozza_lemma_H0}. In particular, for $u \in X^2_0$ we define $$\tilde{H}_0 u := -\partial_x^2u+iv\partial_xu \in L^2(\mathbf{R}).$$
As in \cite{LeCozIR}, we define the space $X^2_{\gamma}$ as 
\begin{equation*}
    X^2_{\gamma} = \{ u \in L^{\infty}(\mathbf{R}) \cap \dot{H}^1(\mathbf{R})\cap\dot{H}^2(\mathbf{R}\backslash\{0\}), \quad  \text{s.t.} \quad \partial_xu(0^+)-\partial_xu(0^-) = \gamma u(0)\}.
\end{equation*}
In this space we define the extension $\tilde{H}_{\gamma}$ of the operator $(H_{\gamma}, D(H_{\gamma}))$ as follows. Given $u \in X^2_{\gamma}$, we define $\tilde{H}_{\gamma}u$ as the unique element in $L^2(\mathbf{R})$ such that $$\langle \tilde{H}_{\gamma}u, \phi \rangle = \langle u, -\partial_x^2\phi+iv\partial_x\phi\rangle, \qquad \forall\phi \in \mathbf{C}^{\infty}_{0}(\mathbf{R}\backslash\{0\}).$$
The definition of $X^2_{\gamma}$ can be equivalently given by the approach in Section \ref{section:self-adjoint}. More precisely, $X^2_{\gamma}$ can be defined as  
\begin{equation*}
    X^2_{\gamma} = \{ u \in L^{\infty}(\mathbf{R})\quad \text{s.t.} \quad u = w^{\lambda} + qG_v^{\lambda}, \quad \text{with} \ \  w^{\lambda} \in X^2_0, \ q \in \mathbf{C}, \ w^{\lambda}(0) = -(\gamma^{-1}+G_v^{\lambda}(0))q\},
\end{equation*}
where $G^{\lambda}_v$ is the Green's function defined in (\ref{bozza_green_function}), which satisfies  $(H_0+\lambda)G_v^{\lambda} = 0$ in $\mathbf{R} \backslash\{0\}$. With this definition, we see that each element $u \in X^2_{\gamma}$ consists in a regular part $w^{\lambda} \in X^2_0$ and a singular part $G^{\lambda}_v \in H^1(\mathbf{R})$. The constraint $w^{\lambda}(0) = -(\gamma^{-1}+G_v^{\lambda}(0))q$ imposes the usual jump condition at $x=0$ to $u$. Similarly, we can equivalently define the action of the operator $\tilde{H}_{\gamma}$ on $u = w^{\lambda}+q G_v^{\lambda} \in  X^2_{\gamma}$ by the identity
\begin{equation*}
    (\tilde{H}_{\gamma}+\lambda)u = (\tilde{H}_0+\lambda)w^{\lambda}.
\end{equation*}
\begin{remark}
   In the case $v=0$, the operator $\tilde{H}_{\gamma}$ acting on elements in $X^2_{\gamma}$, reduces to the operator $\tilde{H}^0_{\gamma}$ studied in \cite{LeCozIR}. The latter is defined as follows: for any $u \in X^2_{\gamma}$, the function $\tilde{H}^0_{\gamma}u$ is the unique element in $L^2(\mathbf{R})$ such that $\langle\tilde{H}^0_{\gamma}u, \phi \rangle = \langle u,-\partial_x^2\phi \rangle$ for any $\phi \in C^{\infty}_{0}(\mathbf{R} \backslash\{0\})$. It follows that, for every $u \in X^2_{\gamma}$, it holds
\begin{equation}
    \tilde{H}_{\gamma}u = \tilde{H}^0_{\gamma}u + iv\partial_xu \quad \text{in} \quad L^2(\mathbf{R})
    \label{bozza_H0+der}
\end{equation}
\end{remark}
\begin{remark}
    If $\tilde{H}_{\gamma}^0$ is the operator defined in the remark above and $u \in X^2_{\gamma}$, then the pairing $\langle u, \tilde{H}^0_{\gamma}u\rangle$ is well-defined. Indeed, for any $\varepsilon>0$, the following identity holds
    \begin{equation*}
        \overline{u}(x)(\tilde{H}_{\gamma}^0u)(x) = -\overline{u}(x)\partial_x^2u(x) = \partial_x\overline{u}(x)\partial_xu(x) - \partial_x\bigl(\overline{u}(x)\partial_xu(x)\bigr), \qquad \ a.e. \quad \text{in} \quad \mathbf{R}\backslash(-\varepsilon,\varepsilon).
    \end{equation*}
    We notice that $\partial_xu \in L^{\infty}(\mathbf{R})$, since $\partial_xu \in H^1(\mathbf{R}\backslash\{0\})$ with $\partial_xu(0^+)-\partial_xu(0^-) = \gamma u(0)$. In particular, $\overline{u}\partial_xu \in H^1(\mathbf{R}\backslash\{0\})$ and, by the Fundamental Theorem of Calculus (\cite{brezis}, Th. 8.2), we have
\begin{equation*}
    \int_{\mathbf{R}}\overline{u}(x)(\tilde{H}^0_{\gamma}u)(x)dx = -\lim_{\varepsilon \to 0}\int_{|x|\geq \varepsilon}\overline{u}(x)\partial_x^2u(x) dx = \int_{\mathbf{R}}\partial_x\overline{u}(x)\partial_xu(x)dx +\gamma |u(0)|^2.
\end{equation*}
    On the other hand, the the pairing $\langle u,\tilde{H}_{\gamma}u\rangle$ is not in general well defined for $u \in X^2_{\gamma}$, nor for $u \in \mathcal{E}$. This is due to the identity in \eqref{bozza_H0+der} and the fact that the function $\Im(\overline{u}\partial_xu)$ may not be an element of $L^1(\mathbf{R})$. Naturally, this issue is related to problem of defining the linear momentum of a field with non-vanishing behavior at infinity, and will be further discussed in Section \ref{section:energy}.
\label{bozza_remark_energy}
\end{remark}
In the next lemma we study a density property of the space $X^2_{\gamma}\cap \mathcal{E}$ in $\mathcal{E}$.
\begin{lemma}
    For every $u \in \mathcal{E}$, there exists a sequence $\{u_n\}_{n \in \mathbf{N}} \subset X^2_{\gamma}\cap\mathcal{E}$ such that $u-u_n \in H^1(\mathbf{R})$ for all $n \in \mathbf{N}$ and $$||u-u_n||_{H^1(\mathbf{R})} \to 0, \qquad \text{as} \quad n \to \infty.$$
    Moreover, the space $X^2_{\gamma}\cap\mathcal{E}$ is dense in $\mathcal{E}$ for the distance $d_{\infty}$.
    \label{bozza_lemma_X^2}
\end{lemma}
\begin{proof}
    Consider $u \in \mathcal{E}$ and consider $R>0$ with $R>|u|_{\mathcal{E}}$. By Theorem \ref{theorem_gerard}, there exists a real valued function $\varphi$, satisfying $\partial^{\alpha}\varphi \in L^2(\mathbf{R})$ for $\alpha=1,2$, and $w \in H^1(\mathbf{R})$ such that
    $$u = e^{i\varphi}+w.$$
    For $n \in \mathbf{N}\backslash\{0\}$, consider the function $\xi_n(x) = 1+\frac{\gamma}{2}|x|e^{-n x^2}$. We have
    \begin{equation*}
        \partial_x(\xi_ne^{i\varphi})(0^+)-\partial_x(\xi_ne^{i\varphi})(0^-) = e^{i\varphi(0)}\bigl(\partial_x\xi_n(0^+)-\partial_x\xi_n(0^-)\bigr) = \gamma (\xi_ne^{i\varphi})(0),
    \end{equation*}
    which implies that $\xi_n e^{i\varphi} \in X^2_{\gamma}$. Moreover, using the fact that $$||\partial_x\xi_n||_{L^2(\mathbf{R})} \to 0, \qquad ||\xi_n-1||_{L^{\infty}}\to 0, \qquad ||\xi_n-1||_{L^2(\mathbf{R})} \to 0, \qquad \text{as} \quad n \to \infty,$$
    we can show that $||\xi_ne^{i\varphi}-e^{i\varphi}||_{H^1(\mathbf{R})} \to 0$ as $n \to \infty$. By density of $D(H_{\gamma})$ in $H^1(\mathbf{R})$, we can find a sequence $\{w_n\}\subset D(H_{\gamma})$ such that $||w-w_n||_{H^1(\mathbf{R})} \to 0$ as $n \to \infty$. If we define $u_n = \xi_n e^{i\varphi}+w_n$ for every $n>0$, then we have $u_n \in X^2_{\gamma}\cap\mathcal{E}$ and $u-u_n \in H^1(\mathbf{R})$ for every $n >0$. Moreover, we have $||u-u_n||_{H^1(\mathbf{R})} \to 0$ as $n \to \infty$.\\ Finally, by means of Lemma \ref{bozza_lemma_E} point iii), there exists $C_{R}>0$ such that $$d_{\infty}(u,u_n) \leq C_{R}\Bigl(1+||u-u_n||_{H^1(\mathbf{R})}\Bigr)||u-u_n||_{H^1(\mathbf{R})} \to 0, \qquad \text{as} \quad n \to \infty.$$ 
\end{proof}
In the following proposition we investigate the linear dynamics generated by $T_{\gamma}(t)$ in the space $X^2_{\gamma}$.
\begin{proposition}
    Let $t \in \mathbf{R}$ and $u \in X^2_{\gamma}$. Then the following properties hold.
    \begin{itemize}
        \item [i)] $T_{\gamma}(t)u \in X^2_{\gamma}$;
        \item[ii)] $\tilde{H}_{\gamma}T_{\gamma}(t)u = e^{-itH_{\gamma}}\tilde{H}_{\gamma}u$;
        \item [iii)] we have $$ T_{\gamma}(t)u = u - i\int_{0}^{t} e^{-isH_{\gamma}}\tilde{H}_{\gamma}uds.$$
        In particular, the map $t \to T_{\gamma}(t)u$ is differentiable on $\mathbf{R}$ and for every $t \in \mathbf{R}$ we have
        $$\frac{d}{dt}T_{\gamma}(t)u = -ie^{-itH_{\gamma}}\tilde{H}_{\gamma}u \in L^2(\mathbf{R}).$$
    \end{itemize}
    \label{bozza_prop_X2}
\end{proposition}
\begin{proof}
    We follow the same strategy as in \cite[Proposition 3.9]{LeCozIR}. Consider $t \in \mathbf{R}$ and $u \in X^2_{\gamma}$, and denote $z(t) := T_{\gamma}(t)u = T_0(t)u + \Gamma_v(t)u$. Proving point i) means showing that $z(t) \in X^2_{\gamma}$. By means of Corollary \ref{bozza_prop_Tgamma}, we know that $z(t)-u \in H^1(\mathbf{R})$, so that $z(t) \in L^{\infty}(\mathbf{R}) \cap \dot{H}^1(\mathbf{R}).$ Thus, we need to prove that $z(t) \in \dot{H}^2(\mathbf{R} \backslash\{0\})$, which we do by a density argument. More precisely, we prove that the following identity holds
    \begin{equation}
        \langle z(t), \partial_y^2\phi-iv\partial_y\phi\rangle = -\langle e^{-itH_{\gamma}}\tilde{H}_{\gamma}u, \phi \rangle, \qquad \forall \phi \in C^{\infty}_0(\mathbf{R}\backslash\{0\}).
        \label{bozza_z(t)}
    \end{equation} Using the fact that $e^{-itH_{\gamma}}\tilde{H}_{\gamma}u$ is an element of $L^2(\mathbf{R})$ and the fact that $z(t) \in \dot{H}^1(\mathbf{R})$, the identity above implies $z(t) \in \dot{H}^2(\mathbf{R} \backslash\{0\})$. Subsequently, we show that $z(t)$ satisfies the correct jump condition at $x=0$. This will complete the proof of point i). By means of the identity in \eqref{bozza_z(t)}, point ii) will follow as well.\\
    We begin by showing \eqref{bozza_z(t)}. We consider $\phi \in C^{\infty}_0(\mathbf{R}\backslash\{0\})$ and we define $z_0(t) := T_0(t)u$ and $z_{\Gamma}(t) := \Gamma_v(t)u$, so that $z(t)= z_0(t) +z_{\Gamma}(t)$. The contribution of $z_0(t)$ to \eqref{bozza_z(t)} gives
    \begin{equation}
        \begin{split}
            \langle z_0(t), \partial_y^2\phi-iv\partial_y\phi \rangle & = \langle u, e^{itH_0}(\partial_y^2\phi-iv\partial_y\phi) \rangle = \langle u, \partial_x^2(e^{itH_0}\phi)-iv\partial_x(e^{itH_0}\phi) \rangle \\ &  = -\langle \partial_xu, \partial_x(e^{itH_0}\phi)\rangle -\langle iv \partial_xu , e^{itH_0}\phi \rangle = -\langle \tilde{H}_{\gamma}u, e^{itH_0}\phi\rangle  +\gamma \overline{u}(0)(e^{itH_0}\phi)(0) 
        \end{split}
        \label{bozza_z0}
    \end{equation}
    On the other hand, the contribution of $z_{\Gamma}(t)$ gives
    \begin{equation*}
        \begin{split}
            \langle z_{\Gamma}(t), \partial_y^2\phi-iv\partial_y\phi  \rangle  &= \int_{\mathbf{R}} \int_{\mathbf{R}}\overline{u}(x) \Gamma(-t,x,y)e^{-i\frac{v^2}{4}t}e^{i\frac{v}{2}(x-y)}\bigl[\partial_y^2\phi(y)-iv\partial_y\phi(y)\bigr]dydx \\ & =\int_{\mathbf{R}} \int_{\mathbf{R}}\overline{u}(x) \Gamma(-t,x,y)e^{-i\frac{v^2}{4}t}e^{i\frac{v}{2}x}\bigl[ \partial_y^2(e^{-i\frac{v}{2}y}\phi(y))+\frac{v^2}{4}e^{-i\frac{v}{2}y}\phi(y)  \bigr]dydx
        \end{split}
    \end{equation*}
    We now integrate by parts the first term in the square parenthesis. Its contribution can be written as
    \begin{equation}
        \begin{split}
            &-\int_{\mathbf{R}} \int_{\mathbf{R}} \overline{u}(x)\partial_y\Gamma(-t,x,y)e^{-i\frac{v^2}{4}t}e^{i\frac{v}{2}x}\partial_y(e^{-i\frac{v}{2}y}\phi(y))dydx  \\ & 
            = \int_{\mathbf{R}} \int_{\mathbf{R}} \partial_x(e^{i\frac{v}{2}x}\overline{u}(x))\sign(x)\sign(y)\Gamma(-t,x,y)e^{-i\frac{v^2}{4}t} \partial_y(e^{-i\frac{v}{2}y}\phi(y))dydx \\ & + 2\overline{u}(0) \int_{\mathbf{R}}\sign(y)\Gamma(-t,0,y)e^{-i\frac{v^2}{4}t} \partial_y(e^{-i\frac{v}{2}y}\phi(y))dy =: I_1+I_2,
        \end{split}
        \label{bozza_interm_comp}
    \end{equation}
    where we used the property in \eqref{bozza_sign}.
    If we integrate by parts in $I_1$ we obtain
\begin{equation*}
    \begin{split}  I_1&=-\int_{\mathbf{R}}\int_{\mathbf{R}}\partial_x(e^{i\frac{v}{2}x}\overline{u}(x))\partial_x\Gamma(-t,x,y)e^{-i\frac{v^2}{4}t}e^{-i\frac{v}{2}y}\phi(y)dydx \\ & =\lim_{\varepsilon \to 0}\int_{|x|\geq \varepsilon}\int_{\mathbf{R}}\partial^2_x(e^{i\frac{v}{2}x}\overline{u}(x))\Gamma(-t,x,y)e^{-i\frac{v^2}{4}t}e^{-i\frac{v}{2}y}\phi(y)dydx  \\& + \gamma \overline{u}(0) \int_{\mathbf{R}}\Gamma(-t,0,y)e^{-i\frac{v^2}{4}t}e^{-i\frac{v}{2}y}\phi(y)dy.
    \end{split}
\end{equation*}
The term $I_2$ can be written as
\begin{equation*}
    \begin{split}
        I_2&=-2\overline{u}(0)\int_{\mathbf{R}}\sign(y)\partial_y\Gamma(-t,0,y)e^{-i\frac{v^2}{4}t}e^{-i\frac{v}{2}y}\phi(y)dy  \\ &= \gamma \overline{u}(0) \int_{\mathbf{R}} \int_{0}^{+\infty}e^{-\frac{\gamma}{2}s}\partial_sK(-it;s+|y|)e^{-i\frac{v^2}{4}t} e^{-i\frac{v}{2}y}\phi(y) dsdy \\ & = -\gamma \overline{u}(0) \int_{\mathbf{R}}\Gamma(-t,0,y)e^{-i\frac{v^2}{4}t}e^{-i\frac{v}{2}y}\phi(y)dy - \gamma \overline{u}(0)(e^{itH_0}\phi)(0).
    \end{split}
\end{equation*}
Overall, writing $\partial_x^2(e^{i\frac{v}{2}x}\overline{u}) = e^{i\frac{v}{2}x}(\partial_x^2\overline{u}+iv\partial_x\overline{u}-\frac{v^2}{4}\overline{u})$ for $x \neq 0$, we obtain
\begin{equation*}
    \langle z_{\Gamma}(t), \partial_y^2\phi-iv\partial_y\phi\rangle = -\langle \Gamma_v(t)\tilde{H}_{\gamma}u, \phi \rangle - \gamma \overline{u}(0)(e^{itH_0}\phi)(0). 
\end{equation*}
Together with (\ref{bozza_z0}), this implies that $\langle z(t), \partial_y^2\phi-iv\partial_y\phi\rangle = -\langle e^{-itH_{\gamma}}\tilde{H}_{\gamma}u, \phi \rangle$, and that $z(t) \in \dot{H}^2(\mathbf{R}\backslash\{0\})$.\\ We now show the jump condition for $z(t)$. Consider $\phi_{\pm} \in \mathcal{S}$ supported $\mathbf{R}^*_{\pm}$. We have, 
\begin{equation*}
    \begin{split}
        -\langle z_{\Gamma}(t), \partial_y\phi_{\pm} \rangle  &= -\int_{\mathbf{R}} \int_{\mathbf{R}}\overline{u}(x)e^{i\frac{v}{2}(x-y)}e^{-i\frac{v^2}{4}t}\Gamma(-t,x,y)\partial_y\phi_{\pm}(y)dydx  \\ & = \int_{\mathbf{R}} \int_{\mathbf{R}}\overline{u}(x)e^{i\frac{v}{2}(x-y)}e^{-i\frac{v^2}{4}t}(\partial_y-i\frac{v}{2})\Gamma(-t,x,y)\phi_{\pm}(y)dydx
    \end{split}
\end{equation*}
In the line above, the term containing $\partial_y$ can be written as
\begin{equation*}
    \begin{split}
         &-\frac{\gamma}{2}\int_{\mathbf{R}}\int_{\mathbf{R}}\int_0^{\infty} \overline{u}(x) e^{-\frac{\gamma}{2}s}e^{i\frac{v}{2}(x-y)}e^{-i\frac{v^2}{4}t} \sign(y)\partial_s K(-it;s+|x|+|y|)\phi_{\pm}(y)dsdydx \\ & = \pm \frac{\gamma}{2} \int_{\mathbf{R}} \int_{\mathbf{R}}\overline{u}(x)e^{i\frac{v}{2}(x-y)}e^{-i\frac{v^2}{4}t}\Gamma(-t,x,y)\phi_{\pm}(y)dydx \\&\pm \frac{\gamma}{2} \int_{\mathbf{R}}\int_{\mathbf{R}}\overline{u}(x)K(-it;|x|+|y|)e^{i\frac{v}{2}(x-y)}e^{-i\frac{v^2}{4}t}\phi_{\pm}(y)dydx.
    \end{split}
\end{equation*}
If we take a sequence $\{\phi^n_{\pm}\}_{n \in \mathbf{N}} \subset \mathcal{S}$ of functions supported in $\mathbf{R}^*_{\pm}$ and that approximate a Dirac distribution, we obtain
\begin{equation*}
    \begin{split}
        \partial_yz_{\Gamma}(t, 0^{\pm}) &=-\lim_{n \to \infty}\langle z_{\Gamma}(t), \partial_y(\phi^n_{\pm}) \rangle = \pm \frac{\gamma}{2} z_{\Gamma}(t,0) \pm \frac{\gamma}{2}z_0(t,0) -i\frac{v}{2}z_{\Gamma}(t,0).
    \end{split}
\end{equation*}
We conclude that $\partial_yz(t,0^+)-\partial_yz(t,0^-) = \gamma z(t,0)$, and hence $z(t) \in X^2_{\gamma}$. With \eqref{bozza_z(t)}, this also shows point ii).\\
We now prove point iii). For $t \in \mathbf{R}$, we set
\begin{equation*}
    w(t) = T_{\gamma}(t)u - u + i \int_{0}^{t}e^{-isH_{\gamma}}\tilde{H}_{\gamma}u\ ds.
\end{equation*}
Thanks to Corollary \ref{bozza_prop_Tgamma}, $w(t)$ belongs to $L^2(\mathbf{R})$ for any $t \in \mathbf{R}$. For any $\psi_0 \in C^{\infty}_0(\mathbf{R}\backslash\{0\})$ we have
\begin{equation*}
    \langle w(t), \psi_0 \rangle = \Bigl\langle  u, e^{itH_{\gamma}}\psi_0 - \psi_0 -i \int_{0}^te^{isH_{\gamma}}H_{\gamma}\psi_0  \Bigr\rangle  =0.
\end{equation*}
By density of $C^{\infty}_0(\mathbf{R}\backslash\{0\})$ in $L^2(\mathbf{R})$, we have $w(t) = 0$ for $\forall t \in \mathbf{R}$.
\end{proof}
\section{The Cauchy problem}
\label{section:cauchy_problem}
In this section we formulate the Cauchy problem associated to equation (\ref{bozza_NLS}). We first introduce the notion of solution. Recall $F(u) = (1-|u|^2)u$, as defined in \eqref{bozza_F_def}.
\begin{definition}
    Let $u_0 \in \mathcal{E}$ and $T \in (0,+\infty]$. We say that $u: (-T,T) \to \mathcal{E}$ is a solution to (\ref{bozza_NLS}) if the following properties are satisfied.
    \begin{itemize}
        \item [i)] the map $t \to u(t)$ is continuous from $(-T,T)$ to $(\mathcal{E},d_{\infty})$;
        \item [ii)] we have $u(0) = u_0$;
        \item [iii)] for $\phi \in \mathcal{S}(\mathbf{R})$ the following identity holds in the sense of distribution in $(-T,T)$
        \begin{equation*}
            \frac{d}{dt}\langle iu(t), \phi \rangle -\langle \partial_xu, \partial_x\phi\rangle -\langle iv\partial_x u, \phi \rangle  - \gamma \overline{u}(t,0)\phi(0) + \langle F(u), \phi\rangle =0.  
        \end{equation*}
    \end{itemize}
    \label{bozza_def_solution}
\end{definition}
We have the following proposition from \cite[Proposition 4.2]{LeCozIR}.
\begin{proposition}[\cite{LeCozIR}, Duhamel formula]
    Let $u_0 \in \mathcal{E}$ and $u \in C^0((-T,T), \mathcal{E})$ for some $T \in (0,+\infty]$. Then $u$ is a solution of (\ref{bozza_NLS}) with $u(0) = u_0$ if and only if 
    \begin{equation*}
        u(t) = T_{\gamma}(t)u_0 +i \int_{0}^te^{-i(t-s)H_{\gamma}}F(u(s))ds.
    \end{equation*}
    \label{bozza_duhamel}
\end{proposition}
We can now prove the local well-posedness of (\ref{bozza_NLS}) in the space $\mathcal{E}$. The proof makes use of the standard fixed point argument as in \cite{LeCozIR}, see also \cite{Gallo, gerard, zhidkov}. 
\begin{proposition}
    Let $R >0$. Then there exists $T>0$ such that for all $u_0 \in \mathcal{E}$ with $|u_0|_{\mathcal{E}} \leq R$, equation (\ref{bozza_NLS}) has a unique solution $u \in C^0((-T,T),\mathcal{E})$ with initial datum $u(0) =u_0$. Moreover, there exist two constants $C_R >0$ and $C'_R>0$ such that, if $u_0,\tilde{u}_0 \in \mathcal{E}$ satisfy $|u_0|_{\mathcal{E}} \leq R$ and $|\tilde{u}_0|_{\mathcal{E}} \leq R$, the corresponding solutions $u$ and $\tilde{u}$ satisfy
    \begin{equation*}
        d_{\infty}(u(t),\tilde{u}(t)) \leq C_Rd_{\infty}(u_0,\tilde{u}_0), \qquad \forall t \in (-T,T).
    \end{equation*}
    If, in addition, $u_0$ and $\tilde{u}_0$ satisfy $u_0-\tilde{u}_0 \in H^1(\mathbf{R})$, then the corresponding solutions $u, \tilde{u}$ satisfy $u(t)-\tilde{u}(t) \in C^0((-T,T), H^1(\mathbf{R}))$ and
    \begin{equation}
        ||u(t)-\tilde{u}(t)||_{H^1(\mathbf{R})} \leq C'_{R}\Bigl(1+||u_0-\tilde{u}_0||_{H^1(\mathbf{R})}\Bigr)||u_0-\tilde{u}_0||_{H^1(\mathbf{R})}, \qquad \forall t \in (-T,T).
        \label{bozza_cont_in_data}
    \end{equation}
    \label{bozza_lwp}
\end{proposition}
\begin{proof}
    Let $u_0 \in \mathcal{E}$ and consider $T >0$. For $w \in C^0((-T,T), H^1(\mathbf{R}))$, we consider the map 
    \begin{equation}
        w \to \Phi_{T,u_0}(w):= i\int_{0}^{t}e^{-i(t-s)H_{\gamma}}F(w(s)+T_{\gamma}(s)u_0)ds.
        \label{bozza_Phi}
    \end{equation}
    Notice that $\Phi_{T,u_0}(w) \in C^0((-T,T), H^1(\mathbf{R}))$. If we can find a particular $w$ such that $w = \Phi_{T,u_0}(w)$, then the function $u(t) := w(t) + T_{\gamma}(t)u_0$ is a solution of the Duhamel formula, and hence of (\ref{bozza_NLS}).\\
    The strategy to solve \eqref{bozza_Phi} is to use a contraction argument on the following space. Consider $R >0$ such that $|u_0|_{\mathcal{E}} < R$; we set
    \begin{equation*}
        W_R(T) = \{ w \in C^0((-T,T), H^1(\mathbf{R})): \ ||w(t)||_{H^1(\mathbf{R})} \leq R, \quad \forall t \in (-T,T)\}.
    \end{equation*}
    Using Lemma \ref{bozza_lemma_E} point ii), Corollary \ref{bozza_prop_Tgamma} point v) and Lemma \ref{bozza_lemma_F}, we can show that there exists a constant $C'_R>0$, depending only on $R$, such that for any $w \in W_R(T)$ it holds
    \begin{equation*}
        ||F(w(s)+T_{\gamma}(s)u_0)||_{H^1(\mathbf{R})} \leq C'_R, \qquad \forall s \in (-T,T).
    \end{equation*}
    Using Corollary \ref{bozza_corollary_H1_propag}, we have
    \begin{equation}
        ||\Phi_{T,u_0}(w)(t)||_{H^1(\mathbf{R})} \lesssim  T \sup_{s \in (-T,T)}||F(w(s)+T_{\gamma}(s)u_0)||_{H^1(\mathbf{R})} \lesssim T C'_R, \qquad \forall t \in (-T,T).
        \label{bozza_contraction_0}
    \end{equation}
    By choosing $T$ small enough (depending only on $R$), we have $\Phi_{T,u_0}(w) \in W_R(T)$. We now show that the map $\Phi_{T,u_0}(w)$ from $W_R(T)$ to itself is a contraction, for  $T$ possibly smaller. Using Lemma \ref{bozza_lemma_F}, we can similarly prove that, for $w, \tilde{w} \in W_R(T)$, 
    \begin{equation}
        \begin{split}
            ||\Phi_{T,u_0}(w)(t)-\Phi_{T,u_0}(\tilde{w})(t)||_{H^1(\mathbf{R})} &\lesssim T \sup_{s \in (-T,T)}||F(w(s)+T_{\gamma}(s)u_0)-F(\tilde{w}(s)+T_{\gamma}(s)u_0)||_{H^1(\mathbf{R})} \\ &  \lesssim T C'_R \sup_{s \in (-T,T)}||w(s)-\tilde{w}(s)||_{H^1(\mathbf{R})}, \qquad \forall t \in (-T,T)
        \end{split}
        \label{bozza_contraction}
    \end{equation}
    By choosing $T$ small enough (depending only on $R$), we have that $w \to\Phi_{T,u_0}(w)$ is a contraction on $W_R(T)$. Using the fixed point theorem, we conclude the existence of a unique $w \in W_R(T)$ such that $w = \Phi_{T,u_0}(w)$, which gives a solution $u(t)$ to \eqref{bozza_NLS} with $u(0) = u_0$. Conversely, for the same choice of $R,T>0$, suppose $u_1$ is a solution of \eqref{bozza_NLS} for $t\in I \subset \mathbf{R}$, with $u_1(0)=u_0$. Then $w_1(t) := u_1(t) - T_{\gamma}(t)u_0$ belongs $C^0(I,H^1(\mathbf{R}))$ with $w_1(0) = 0$. By uniqueness of the fixed point, we have that $w_1 = w$ for a small time interval, which can be then extended to $(-T,T)$. This shows existence and uniqueness. \\
    To prove the continuity with respect to the initial data, we consider $u_0, \tilde{u}_0 \in \mathcal{E}$ and $R>0$ such that $|u_0|_{\mathcal{E}}\leq R$ and $|\tilde{u}_0|_{\mathcal{E}}\leq R$. Let $w, \tilde{w} \in W_R(T)$ be the fixed points of the maps $\Phi_{T,u_0}$ and $\Phi_{T,\tilde{u}_0}$ for $T>0$ small enough. Proceeding as for \eqref{bozza_contraction_0} and \eqref{bozza_contraction}, we have for $T$ small enough,
    \begin{equation*}
        \begin{split}
            ||w-\tilde{w}||_{(L^{\infty}(-T,T),H^1(\mathbf{R}))} &= ||\Phi_{T,u_0}(w)(t)-\Phi_{T,\tilde{u}_0}(\tilde{w})(t)||_{(L^{\infty}(-T,T),H^1(\mathbf{R}))} \\ & \leq \frac{1}{2}(d_{\infty}(u_0,\tilde{u}_0)+||w-\tilde{w}||_{(L^{\infty}(-T,T),H^1(\mathbf{R}))}).
        \end{split} 
    \end{equation*}
    We deduce 
    \begin{equation}
        ||w-\tilde{w}||_{(L^{\infty}(-T,T),H^1(\mathbf{R}))} \leq d_{\infty}(u_0,\tilde{u}_0)
        \label{bozza_eq_w_lwp}
    \end{equation}
    Now call $u, \tilde{u}$ the solutions to (\ref{bozza_NLS}) with initial data $u_0,\tilde{u}_0$, respectively. Using \eqref{bozza_eq_w_lwp} and Corollary \ref{bozza_prop_Tgamma}, we conclude that there exists $C_R>0$ such that  
    \begin{equation*}
        d_{\infty}(u(t), \tilde{u}(t)) \leq C_R d_{\infty}(u_0, \tilde{u}_0), \qquad \forall t \in (-T,T).
    \end{equation*}
     We now want to prove \eqref{bozza_cont_in_data}. Suppose the initial data satisfy $u_0-\tilde{u}_0 \in {H^1(\mathbf{R})}$. By Corollary \ref{bozza_corollary_H1_propag}, there exists a constant $C>0$ such that
        \begin{equation*}
            \begin{split}
                ||u(t)-\tilde{u}(t)||_{H^1(\mathbf{R})} &\leq ||w(t)-\tilde{w}(t)||_{H^1(\mathbf{R})} + ||T_{\gamma}(t)u_0-T_{\gamma}(t)\tilde{u}_0||_{H^1(\mathbf{R})} \\ & \leq d_{\infty}(u_0,\tilde{u}_0) + C ||u_0-\tilde{u}_0||_{H^1(\mathbf{R})}, \qquad \forall t \in (-T,T),
            \end{split}
        \end{equation*}
        where we use the identity of $T_{\gamma}(t) = e^{-itH_{\gamma}}$ in $H^1(\mathbf{R})$. By Lemma \ref{bozza_lemma_E} point iii), there exists $C'_R>0$ such that
        \begin{equation*}
            ||u(t)-\tilde{u}(t)||_{H^1(\mathbf{R})} \leq C'_{R}\Bigl(1+||u_0-\tilde{u}_0||_{H^1(\mathbf{R})}\Bigr)||u_0-\tilde{u}_0||_{H^1(\mathbf{R})}, \qquad \forall t \in (-T,T).
        \end{equation*}
\end{proof}
\begin{corollary}[Blow-up alternative]
    Let $u_0 \in \mathcal{E}$. The problem (\ref{bozza_NLS}) admits a unique maximal solution $u$ with $u(0) =u_0$, defined on $(-T_-,T_+)$, with $T_{\pm} \in (0,+\infty]$. Moreover, if $T_{\pm} < +\infty$, then 
    \begin{equation*}
        \lim_{t \to \pm T_{\pm}}|u(t)|_{\mathcal{E}} = +\infty.
    \end{equation*}
    \label{bozza_corollary}
\end{corollary}
We now state the following proposition, which will central to our analysis.
\begin{proposition}
    Let $u_0 \in \mathcal{E}$, and consider $u \in C^0((-T_-,T_+), \mathcal{E})$ be the corresponding maximal solution, defined for $T_{\pm}  \in (0,+\infty]$. Then, it holds
    \begin{equation*}
        u - u_0 \in C^0((-T_-,T_+), H^1(\mathbf{R})). 
    \end{equation*}
    \label{bozza_u-u0_prop}
\end{proposition}
\begin{proof}
    The proof follows by considering the Duhamel formula. We have
    \begin{equation*}
        u(t) - u_0 = T_{\gamma}(t)u_0 - u_0 +i \int_{0}^te^{-i(t-s)H_{\gamma}}F(u(s))ds.
    \end{equation*}
    From the property i) in Corollary \ref{bozza_prop_Tgamma} we know that $T_{\gamma}(t)u_0 - u_0$ belongs to $C^0(\mathbf{R}, H^1(\mathbf{R}))$. Using the properties of the propagator $e^{-itH_{\gamma}}$ in $H^1(\mathbf{R})$ from Corollary \ref{bozza_corollary_H1_propag}, the proposition follows.
\end{proof}

In the following proposition we state a persistence of regularity result. More precisely, we show that if $u_0 \in X^2_{\gamma} \cap \mathcal{E}$, then the corresponding solution $u(t)$ belongs to $X^2_{\gamma}\cap\mathcal{E}$ for all times of existence. Moreover, the solution $u(t)$ is differentiable in time.
\begin{proposition}
    Let $u_0 \in X^2_{\gamma}\cap \mathcal{E}$, and let $u \in C^0((-T_-,T_+), \mathcal{E})$ be the maximal solution to (\ref{bozza_NLS}) with $u(0) = u_0$ given by Proposition \ref{bozza_corollary}, defined for  $T_{\pm} \in (0,+\infty]$. We have $u(t) \in X^2_{\gamma}$ for all $t \in (-T_-,T_+)$. Moreover, $\partial_t u \in C^0((-T_-,T_+),L^2(\mathbf{R}))$, and
    \begin{equation}
        \partial_t u(t) = -i\tilde{H}_{\gamma}u(t) +iF(u(t)), \qquad \forall t \in (-T_-,T_+).
        \label{bozza_eq_regular}
    \end{equation}
    \label{bozza_persistence_regularity}
\end{proposition}
The proof of Proposition \ref{bozza_persistence_regularity} closely follows the steps of the proof of Proposition 4.5 in the work of Ianni et al. \cite{LeCozIR}, and it is postponed to Appendix \ref{section: appendix_II}.

\section{Renormalized momentum and conserved energy}
\label{section:energy}
In order to prove the global well-posedness of \eqref{bozza_NLS}, we need an a priori bound on the functional $|u|^2_{\mathcal{E}}$, which we derive from the conserved energy. In order to define the latter, we first need a well-defined notion of linear momentum for fields in $\mathcal{E}$. It is the goal of this section to introduce the notion of renormalized momentum and of conserved energy associated to \eqref{bozza_NLS}.\\ 

We begin by considering the case $v=0$, i.e. the problem with a static impurity. In this case, \eqref{bozza_NLS} conserves the energy
\begin{equation*}
    E_{\gamma}(u) = \frac{1}{2}\int_{\mathbf{R}}|\partial_xu|^2dx + \frac{\gamma}{2}|u(0)|^2 + \frac{1}{4} \int_{\mathbf{R}}(1-|u|^2)^2dx,
\end{equation*}
as shown by Ianni et al. in \cite{LeCozIR}.\\
If $v \neq 0$, instead, the functional $E_{\gamma}$ is not conserved by the dynamics. The definition of conserved energy, in this case, can be obtained by adding to $E_{\gamma}$ a linear momentum term. More precisely, the conserved energy for \eqref{bozza_NLS} takes the form
\begin{equation*}
    K(u) = E_{\gamma}(u) - vP(u).
\end{equation*}
where $P(u)$ is a linear momentum, i.e. it is a functional that, given $u \in \mathcal{E}$, satisfies $$\lim_{t \to 0}\frac{P(u+tw)-P(u)}{t}= \Im\int_{\mathbf{R}} \partial_x u \  \overline{w} \ dx,$$
for any $w \in H^1(\mathbf{R})$ with compact support.\\
A natural choice would be to define the momentum as $P(u) = \frac{1}{2}\Im\int_{\mathbf{R}}\partial_x u \  \overline{u} \ dx$. In fact, this notion of momentum fails to be well defined on $\mathcal{E}$, due to the behavior of the fields at infinity (in particular, the function  $\Im(\partial_x u \overline{u})$ may not be an element of $ L^1(\mathbf{R})$).\\
On the other hand, a rigorous definition of linear momentum for fields in $\mathcal{E}$ has been given by Bethuel et al. in \cite[Lemma 3]{bethuel} and by Mari\c{s} and Mur in \cite[Section 3]{maris_mur}. However, employing it in the definition of $K$ may lead to complications, since this notion of momentum is defined modulo\footnote{The definition of linear momentum in \cite{bethuel,maris_mur} relies on the non-unique decomposition (\ref{bozza_decomposition_formula}), which, a priori, entails a non-unique notion of momentum for each element $u \in \mathcal{E}$, up to multiples of $\pi$. Therefore, this momentum has to be defined modulo $\pi\mathbf{Z}$ in order to be single valued in $\mathcal{E}$.} $\pi \mathbf{Z}$.\\
In the following we introduce a notion of linear momentum which relies on the same approach of \cite{bethuel, maris_mur}, but it is more oriented towards the definition of the conserved energy $K$.\\
Given a reference function $u_0 \in \mathcal{E}$, we consider the affine subspace $u_0+H^1(\mathbf{R}) \subset \mathcal{E}$. In this set, we define the linear momentum to be the map $P_{u_0}: u_0+H^1(\mathbf{R}) \to \mathbf{R}$ such that
\begin{equation}
    P_{u_0}(u) := \frac{1}{2} \Im\int_{\mathbf{R}}(\overline{u}-\overline{u}_0)\partial_x(u+u_0)dx.
    \label{bozza_def_momentum}
\end{equation}
and we interpret it as the relative momentum of $u$ with respect to $u_0$. Clearly, for a given $u \in \mathcal{E}$, we have different notions of momentum, depending on the choice of the reference function. The following lemma states the relation between these notions.
\begin{lemma}
    Consider two reference functions $u_0,\tilde{u}_0 \in \mathcal{E}$, satisfying $u_0-\tilde{u}_0 \in H^1(\mathbf{R})$. Then the sets $u_0 +H^1(\mathbf{R})$ and $\tilde{u}_0+H^1(\mathbf{R})$ coincide, and, for any $u \in u_0 + H^1(\mathbf{R})$, it holds
    \begin{equation}
        P_{u_0}(u)-P_{\tilde{u}_0}(u) = P_{u_0}(\tilde{u}_0).
    \end{equation}
    \label{bozza_mom_difference_subscript}
\end{lemma} 
\begin{proof}
We write $u_0=\tilde{u}_0+w$, where $w \in H^1(\mathbf{R})$. We have
    \begin{equation*}
        \begin{split}
            P_{u_0}(u) &= \frac{1}{2}\Im\int_{\mathbf{R}}(\overline{u}-\overline{u}_0)\partial_x(u+u_0)dx = \frac{1}{2}\Im\int_{\mathbf{R}}(\overline{u}-\overline{\tilde{u}}_0-\overline{w})\partial_x(u+\tilde{u}_0+w)dx \\ & = P_{\tilde{u}_0}(u) + 
            \frac{1}{2}\Im\int_{\mathbf{R}}\Bigl((\overline{u}-\overline{\tilde{u}}_0)\partial_xw-\overline{w}\partial_x(u+\tilde{u}_0)-\overline{w}\partial_xw \Bigr)dx. 
        \end{split}
    \end{equation*}
    If we take the complex conjugate and we integrate by parts the first integrand above, we obtain
    \begin{equation*}
        P_{u_0}(u) = P_{\tilde{u}_0}(u) - \frac{1}{2}\Im\int_{\mathbf{R}}\overline{w}\partial_x(2\tilde{u}_0+w) = P_{\tilde{u}_0}(u)+P_{u_0}(\tilde{u}_0).
    \end{equation*}
\end{proof}
We can now define the conserved energy associated to \eqref{bozza_NLS}. For $u_0 \in \mathcal{E}$, we define the conserved energy $K_{u_0}: u_0+H^1(\mathbf{R}) \to \mathbf{R}$ as
\begin{equation}
        K_{u_0}(u) := E_{\gamma}(u) - vP_{u_0}(u).
        \label{bozza_K_u0}
    \end{equation}
In the following, we will take $u_0$ to be the initial datum of the Cauchy problem. In this case, the corresponding solution $u(t)$ will belong to the set $u_0+H^1(\mathbf{R})$ for all times of existence, thanks to Proposition \ref{bozza_u-u0_prop}. In particular, the values $P_{u_0}(u(t))$ and $K_{u_0}(u(t))$ will be well defined for all times of existence of the solution $u(t)$.\\
We now formulate two lemmas. Recall the definition of the distance $d_A$ in \eqref{bozza_d_A}.
\begin{lemma}[\cite{LeCozIR}]
    Set $\gamma >0$. The functional $ E_{\gamma}$ is continuous on $(\mathcal{E}, d_A)$, for any $A>0$, and hence on $(\mathcal{E},d_{\infty})$. In particular, for $R>0$ the functional $E_{\gamma}$ is Lipschitz continuous on $\{u \in \mathcal{E}, \ |u|_{\mathcal{E}} < R\}$.
    \label{bozza_lemma_energy0}
\end{lemma}
The proof of the lemma above can be found in \cite[Lemma 2.2]{LeCozIR}
\begin{lemma}
    Set $\gamma>0$ and $u_0 \in \mathcal{E}$. The energy $K_{u_0}$ defined in \eqref{bozza_K_u0} is continuous on $u_0+H^1(\mathbf{R})$. In particular, it is Lipschitz continuous on $\{u \in u_0+H^1(\mathbf{R}), \ ||u-u_0||_{H^1(\mathbf{R})} < R\}$.
    \label{bozza_energy_cont}
\end{lemma}
\begin{proof}
    Set $R>0$. Consider $u_0 \in \mathcal{E}$ and  $u_1,u_2 \in u_0+H^1(\mathbf{R})$ such that $||u_1-u_0||_{H^1(\mathbf{R})}<R$ and $||u_2-u_0||_{H^1(\mathbf{R})} < R$. 
    By point iv) of Lemma \ref{bozza_lemma_E}, Lemma \ref{bozza_cont_norm} and the Lipschitz continuity of $E_{\gamma}$ in $d_{\infty}$, there exists $C>0$ depending only on $R$ and $|u_0|_{\mathcal{E}}$ such that 
    \begin{equation*}
        |E_{\gamma}(u_1)-E_{\gamma}(u_2)| \leq C  ||u_1-u_2||_{H^1(\mathbf{R})}.
    \end{equation*}
    Similarly, we have
    \begin{equation*}
        |P_{u_0}(u_1)-P_{u_0}(u_2)| \leq \frac{1}{2}\Bigl|\int_{\mathbf{R}}(\overline{u}_1-\overline{u}_2)\partial_x(u_1+u_0)  - (\overline{u}_2-\overline{u}_0)\partial_x(u_2-u_1)dx \Bigr| \leq C' ||u_1-u_2||_{H^1(\mathbf{R})},
        \end{equation*}
        where $C'>0$ only depends on $R$ and $|u_0|_{\mathcal{E}}$.
\end{proof}
We can now prove the conservation of the energy $K_{u_0}$ for the solution with initial datum $u_0 \in X^2_{\gamma} \cap \mathcal{E}$.
\begin{proposition}
    Let $u_0 \in X^2_{\gamma}\cap\mathcal{E}$ and let $u \in C^0((-T_-,T_+), \mathcal{E})$ with $\partial_t u \in C^0 ((-T_-,T_+), L^2(\mathbf{R}))$ be the corresponding maximal solution to (\ref{bozza_NLS}), with satisfying $u(0) = u_0$ and defined for $T_{\pm} \in [0,+\infty)$. It holds
    \begin{equation*}
        K_{u_0}(u(t)) = K_{u_0}(u_0), \qquad \forall t \in (-T_-,T_+).
    \end{equation*}
    \label{bozza_cons_energy_X^2}
\end{proposition}
\begin{proof}
    By means of Remark \ref{bozza_remark_energy}, we can write
    \begin{equation*}
        K_{u_0}(u) = \frac{1}{2}\int_{\mathbf{R}}\overline{u}(\tilde{H}^0_{\gamma}u) dx + \frac{1}{4}\int_{\mathbf{R}}(1-|u|^2)^2 dx -\frac{v}{2}\Im \int_{\mathbf{R}}(\overline{u}-\overline{u}_0)\partial_x(u+u_0)dx
    \end{equation*}
    Our goal is to show that the map $t \to K_{u_0}(u(t))$ is differentiable on $(-T_-,T_+)$. Then, the identity $\frac{d}{dt}K_{u_0}(u(t)) = 0$ follows by direct computation. As in \cite{adami_noja}, we compute $\frac{1}{\tau}\Bigl(K_{u_0}(u(t+\tau))-K_{u_0}(u(t))\Bigr)$. The first term we obtain reads
    \begin{equation}
        \begin{split}
            &\frac{1}{2\tau}\int_{\mathbf{R}}\overline{u}(t+\tau)\tilde{H}^0_{\gamma}u(t+\tau)-\overline{u}(t)\tilde{H}^0_{\gamma}u(t) \ dx \\ &= \frac{1}{2}\int_{\mathbf{R}}  \Bigl[\frac{\overline{u}(t+\tau)-\overline{u}(t)}{\tau}\Bigr] \tilde{H}^0_{\gamma}u(t+\tau) -(\tilde{H}^0_{\gamma}\overline{u}(t))\Bigl[\frac{u(t)-u(t+\tau)}{\tau}\Bigr] dx.
        \end{split}
        \label{bozza_der_2}
    \end{equation}
    The momentum term in $K_{u_0}$ contributes with
    \begin{equation}
        \begin{split}
            &-\frac{v}{2\tau}\Im \int_{\mathbf{R}}(\overline{u}-\overline{u}_0)\partial_x(u+u_0)\Bigr|_{t+\tau}dx + \frac{v}{2\tau}\Im \int_{\mathbf{R}}(\overline{u}-\overline{u}_0)\partial_x(u+u_0)\Bigr|_{t}dx \\ &  = - \frac{v}{2} \Im\int_{\mathbf{R}}  \Bigl[\frac{\overline{u}(t+\tau)-\overline{u}(t)}{\tau}\Bigr] \partial_x u(t+\tau) + (\partial_x\overline{u}(t))\Bigl[\frac{u(t)-u(t+\tau)}{\tau}\Bigr] dx.
        \end{split}
        \label{bozza_der_1}
    \end{equation}
    Using (\ref{bozza_H0+der}), the sum of (\ref{bozza_der_2}) and (\ref{bozza_der_1}) gives
    \begin{equation*}
        \begin{split}
            \frac{1}{2}\Re\int_{\mathbf{R}}  \Bigl[\frac{\overline{u}(t+\tau)-\overline{u}(t)}{\tau}\Bigr] \tilde{H}_{\gamma}u(t+\tau) &-(\overline{\tilde{H}_{\gamma}u(t)})\Bigl[\frac{u(t)-u(t+\tau)}{\tau}\Bigr] dx \to \Re\int_{\mathbf{R}}  \partial_t \overline{u} \tilde{H}_{\gamma}u(t) \  dx.
        \end{split}
    \end{equation*}
    as $\tau \to 0$. A similar computation with the quartic contribution in $K_{u_0}$ shows the conservation of energy.
\end{proof}
\begin{corollary}
    Let $u_0 \in X^2_{\gamma}\cap\mathcal{E}$ and let $u \in C^0((-T_-,T_+), \mathcal{E})$ with $\partial_t u \in C^0 ((-T_-,T_+), L^2(\mathbf{R}))$ be the corresponding maximal solution to (\ref{bozza_NLS}), with satisfying $u(0) = u_0$ and defined for $T_{\pm} \in [0,+\infty)$. Consider a reference function $\tilde{u} \in u_0+H^1(\mathbf{R})$. It holds the conservation law
    \begin{equation*}
        K_{\tilde{u}}(u(t)) = K_{\tilde{u}}(u_0), \qquad \forall t \in (-T_-,T_+).
    \end{equation*}
    \label{bozza_corollary_reference_field_K}
\end{corollary}
\begin{proof}
    By Lemma \ref{bozza_mom_difference_subscript}, the identity $K_{u_0}(u(t)) = K_{\tilde{u}}(u(t)) - vP_{u_0}(\tilde{u})$ holds for all $t \in (-T_-,T_+)$. By the conservation of the energy $K_{u_0}$ in \ref{bozza_cons_energy_X^2}, we have $K_{\tilde{u}}(u(t)) - K_{\tilde{u}}(u_0) = 0$, for all $t \in (-T_-,T_+)$.
\end{proof}
\subsection{The case of nowhere vanishing fields} In this section we introduce another notion of linear momentum, which is defined for functions $u \in \mathcal{E}$ which are nowhere vanishing. This leads to an alternative definition of the total energy for such fields. The latter is globally defined in the set of nowhere vanishing fields, and will be useful later on, in the study of the orbital stability of stationary states.\\
Consider a field $u \in \mathcal{E}$ that satisfies $\inf_{x \in \mathbf{R}}|u(x)|>0$. Then, there exist two functions $\rho: \mathbf{R} \to \mathbf{R}$ and $\theta: \mathbf{R} \to \mathbf{R}$, with $\rho^2-1 \in L^2(\mathbf{R})$ and $\partial_x\rho, \partial_x\theta \in L^2(\mathbf{R})$, such that $u(x) = \rho(x)^{i\theta(x)}$ for all $x \in \mathbf{R}$. We define the linear momentum of $u$, as
\begin{equation}
    \mathcal{P}(u) = \frac{1}{2}\Im\int_{\mathbf{R}}(\overline{u}-e^{-i\theta})\partial_x(u+e^{i\theta})dx = \frac{1}{2}\int_{\mathbf{R}}(\rho^2-1)\partial_x\theta dx.
    \label{bozza_def_mom_nonvanishing}
\end{equation}
Similarly, we define the energy of $u$ as
\begin{equation}
    \mathcal{K}(u) = E_{\gamma}(u) - v\mathcal{P}(u).
    \label{bozza_energy_non_vanishing}
\end{equation}
We can interpret $\mathcal{P}(u)$ as the relative momentum of $u \in U_0$ with respect to its own phase $e^{i\theta} \in U_0$.\\
The following lemma from \cite[Lemma A.3]{de_laire} provides the continuity of the momentum $\mathcal{P}(u)$ in the set of nowhere vanishing fields, with respect to the distance $d_A$ (defined in \eqref{bozza_d_A}).
\begin{lemma}
    Set $U_0 = \{u \in \mathcal{E}| \ \inf_{x \in \mathbf{R}}|u(x)| >0\}$. The functional $\mathcal{P}:U_0 \to \mathbf{R}$ is continuous for the distance $d_A$ for any $A >0$.
    \label{bozza_mom_continuity}
\end{lemma}
\begin{proof}
    Set $A>0$. In order to show the continuity of the momentum $\mathcal{P}$, it is sufficient to show the continuity, with respect to the distance $d_A$, of the maps $u \to \rho^2-1$ and $u \to \partial_x\theta$ from $U_0$ to $L^2(\mathbf{R})$, where $u(x) = \rho(x) e^{i\theta(x)}$ for all $x \in \mathbf{R}$. The continuity of the first map is a consequence of the definition of $d_A$; in the following we focus on the second map.
    Given $u(x) = \rho(x) e^{i\theta}(x) \in U_0$, we write
    \begin{equation*}
        \partial_x\theta = \frac{\Re(iu\partial_x\overline{u})}{\rho^2}.
    \end{equation*}
    Let's fix $u_1 = \rho_1e^{i\theta_1} \in U_0$ with $b_1 := \inf_{x \in \mathbf{R}}|u_1(x)|>0$, and choose any $\varepsilon >0$. We want to show that there exists $\delta >0$ such that if $u_2 = \rho_2 e^{i\theta_2} \in U_0$ satisfies $d_A(u_1,u_2) \leq \delta$, then $||\partial_x\theta_1-\partial_x\theta_2||_{L^2(\mathbf{R})} \leq \varepsilon$. Set $b_2 := \inf_{x \in \mathbf{R}} |u_2(x)|>0$. By choosing $\delta>0$ small enough, we have $b_2 \geq b_1/2$, by Lemma \ref{bozza_lemma_modulus}. Then, we compute
    \begin{equation}
        \partial_x\theta_1-\partial_x\theta_2 = \Re(iu_2\partial_x\overline{u}_2) \frac{\rho_2^2-\rho_1^2}{\rho_1^2\rho_2^2} -  \frac{\Re\bigl(iu_2\partial_x(\overline{u}_2-\overline{u}_1)\bigr)}{\rho_1^2}  +\frac{\Re\bigl(i(u_1-u_2)\partial_x\overline{u}_1\bigr)}{\rho_1^2}
    \end{equation}
    Hence, we obtain
    \begin{equation}
    \begin{split}
        ||\partial_x\theta_1-\partial_x\theta_2||_{L^2(\mathbf{R})} &\leq \frac{2}{b_1^3} ||\partial_x u_2||_{L^2(\mathbf{R})}||\rho^2_1-\rho^2_2||_{L^{\infty}(\mathbf{R})} + \frac{||\rho_2||_{L^{\infty}(\mathbf{R})}}{b_1^2}||\partial_xu_1-\partial_xu_2||_{L^2(\mathbf{R})} \\ &+ \frac{1}{b_1^2}||(u_1-u_2)\partial_xu_1||_{L^2(\mathbf{R})}.
    \end{split}
        \label{bozza_ineq_theta}
    \end{equation}
    By means of Lemma \ref{bozza_lemma_modulus} and the definition of $d_A$, each of the first two terms on the right hand side of the inequality above can be made smaller than $\varepsilon/3$ by choosing $\delta>0$ small enough. For the last term, we proceed as follows. We write, for $R >0$,
    \begin{equation*}
        \int_{\mathbf{R}}|u_1-u_2|^2|\partial_xu_1|^2 dx =  \int_{\mathbf{R}\backslash[-R,R]}|u_1-u_2|^2|\partial_xu_1|^2 dx + \int_{[-R,R]}|u_1-u_2|^2|\partial_xu_1|^2 dx.
    \end{equation*}
    We can choose $R_*>0$ big enough, so that, for $\delta>0$ small enough, we have
    \begin{equation*}
         \int_{\mathbf{R}\backslash[-R_*,R_*]}|u_1-u_2|^2|\partial_xu_1|^2 dx \leq (3||u_1||_{L^{\infty}(\mathbf{R})})^2\int_{\mathbf{R}\backslash[-R_*,R_*]}|\partial_xu_1|^2 dx \leq \frac{\varepsilon^2}{8}b_1^4.
    \end{equation*}
    Finally, for any $x \in [-R_*,R_*]$, we have
    \begin{equation*}
        |u_1(x)-u_2(x)| \leq \int_{0}^x|\partial_xu_1(s)-\partial_xu_2(s)|ds + |u_1(0)-u_2(0)| \leq (R_*^{\frac{1}{2}}+1) d_A(u_1,u_2).
    \end{equation*}
    By choosing $\delta>0$ even smaller we have
    \begin{equation*}
         \int_{[-R,R]}|u_1-u_2|^2|\partial_xu_1|^2 dx \leq ||\partial_xu_1||^2_{L^2(\mathbf{R})}||u_1-u_2||_{L^{\infty}[-R_*,R_*]} \leq \frac{\varepsilon^2}{8}b_1^4.
    \end{equation*}
    We conclude that $||(u_1-u_2)\partial_xu_1||_{L^2(\mathbf{R})} \leq \frac{1}{2}b_1^2\varepsilon$ for $\delta >0$ small enough. By (\ref{bozza_ineq_theta}), we conclude that $||\partial_x\theta_1-\partial_x\theta_2||_{L^2(\mathbf{R})} \leq \varepsilon$ for $\delta> 0$ small enough.
\end{proof}
From the continuity of the energy $E_{\gamma}$ with respect to the distance $d_A$ (Lemma \ref{bozza_lemma_energy0}), we deduce the following corollary.
\begin{corollary}
    Set $U_0 = \{u \in \mathcal{E}| \ \inf_{x \in \mathbf{R}}|u(x)| >0\}$. Then, the functional $\mathcal{K}:U_0 \to \mathbf{R}$ is continuous for the distance $d_A$, for any $A >0$.
    \label{bozza_en_K_continuity}
\end{corollary}
\subsection{A relation between the two notions of energy}
We now establish a useful relation between the two notions of energy we have encountered so far. In particular, this relation will allow us to prove the conservation law for the energy $\mathcal{K}$.
\begin{lemma}
        Let $u_0 \in \mathcal{E}$. Consider two fields $u_1,u_2 \in u_0+H^1(\mathbf{R})$ which are nowhere vanishing. Then, there exists $k \in \mathbf{Z}$ such that
        $$\mathcal{K}(u_1) - \mathcal{K}(u_2)  =  K_{u_0}(u_1) - K_{u_0}(u_2) + v\pi k$$
        \label{bozza_rel_energies_1}
    \end{lemma}
In order to prove Lemma \ref{bozza_rel_energies_1} it is useful to determine first a relation between the two notions of momentum in \eqref{bozza_def_momentum} and \eqref{bozza_def_mom_nonvanishing}. To this purpose, we state the following lemma.
\begin{lemma}
    Let $u_0 \in \mathcal{E}$, and consider the decomposition $u_0 = e^{i\varphi} + w_0$, where $\varphi$ is real valued and satisfying $\partial_x\varphi \in L^2(\mathbf{R})$ and $w_0 \in H^1(\mathbf{R})$, as given by Theorem \ref{theorem_gerard}. For any $u \in u_0 + H^1(\mathbf{R})$ nowhere vanishing, there exists $k \in \mathbf{Z}$ such that
    \begin{equation}
        \mathcal{P}(u) - P_{u_0}(u) = \pi k + P_{e^{i\varphi}}(u_0)
    \end{equation}
    \label{bozza_lemma_rel_momenta}
\end{lemma}
\begin{proof} We begin with the following observation: since $u \in u_0 + H^1(\mathbf{R})$ and $u_0 = e^{i\varphi} + w_0$, there exists $w \in H^1(\mathbf{R})$ such that $u = e^{i\varphi}+w$. This implies that the momentum $P_{e^{i\varphi}}(u)$ is well defined. Now the proof consists of two steps. First we evaluate the difference $P_{e^{i\varphi}}(u) - P_{u_0}(u)$; second, we evaluate the difference $\mathcal{P}(u) - P_{e^{i\varphi}}(u)$. The sum of these two terms will give us the desired identity. \\
For the first step, it is sufficient to use Lemma \ref{bozza_mom_difference_subscript} to obtain
\begin{equation}
    P_{e^{i\varphi}}(u) - P_{u_0}(u) = P_{e^{i\varphi}}(u_0).
    \label{bozza_rel_mom_2}
\end{equation}
For the second step, we compute
\begin{equation*}
        \begin{split}
            \mathcal{P}(u)-P_{e^{i\varphi}}(u) &= \frac{1}{2}\Im\int_{\mathbf{R}}(\overline{u}-e^{-i\theta})\partial_x(u+e^{i\theta})dx - \frac{1}{2}\Im \int_{\mathbf{R}}(\overline{u}-e^{-i\varphi})\partial_x(u+e^{i\varphi})dx  \\ & = \frac{1}{2}\Im\int_{\mathbf{R}} (\overline{u}-e^{-i\theta})\partial_x(e^{i\theta}-e^{i\varphi})-(e^{-i\theta}-e^{-i\varphi})\partial_x(u+e^{i\varphi})dx \\ &  =\frac{1}{2} \Im\int_{\mathbf{R}} \partial_x(u-e^{i\theta})(e^{-i\theta}-e^{-i\varphi})-(e^{-i\theta}-e^{-i\varphi})\partial_x(u+e^{i\varphi})dx,
        \end{split}
    \end{equation*}
    where in the integration by parts we used the fact that $\overline{u}-e^{-i\theta} = (\rho-1)e^{-i\theta} \in H^1(\mathbf{R})$, and that $e^{i\theta}-e^{i\varphi} = w-(1-\rho)e^{i\theta} \in H^1(\mathbf{R})$. We obtain
    \begin{equation*}
        \mathcal{P}(u)-P_{e^{i\varphi}}(u) = -\frac{1}{2}\Im\int_{\mathbf{R}} (e^{-i\theta}-e^{-i\varphi})\partial_x(e^{i\theta}+e^{i\varphi})dx = -\frac{1}{2}\Im\int_{\mathbf{R}}i\partial_x(\theta-\varphi) - \partial_x(e^{i(\theta-\varphi)})dx.
    \end{equation*}
    Since $e^{i(\theta -\varphi)} -1 = (\rho-1) - \overline{w}e^{i\theta} \in H^1(\mathbf{R})$, the second integrand gives no contribution. Then, writing $u= e^{i\theta}+(\rho-1)e^{i\theta}$, we conclude by Theorem \ref{theorem_gerard} that there exist $k_{\pm}\in \mathbf{Z}$ such that $\theta-\varphi-2\pi k_{\pm} \in L^2(\mathbf{R_{\pm}})$, hence $\theta-\varphi \to 2k_{\pm}\pi$ as $x \to \pm \infty$. We conclude 
    \begin{equation}
        \mathcal{P}(u) - P_{e^{i\varphi}}(u) = -(k_+-k_-)\pi
        \label{bozza_rel_mom_1}
    \end{equation}
    Taking the sum of the expressions in (\ref{bozza_rel_mom_2}) and (\ref{bozza_rel_mom_1}) we can conclude the proof.
    \end{proof}
    We can now prove Lemma \ref{bozza_rel_energies_1}.
    \begin{proof}[Proof of Lemma \ref{bozza_rel_energies_1}]
        Given $u_0 \in \mathcal{E}$, consider again the decomposition $u_0 = e^{i\varphi} + w_0$, where $\varphi$ is real valued and satisfying $\partial_x\varphi \in L^2(\mathbf{R})$ and $w_0 \in H^1(\mathbf{R})$, as given by Theorem \ref{theorem_gerard}. We now apply Lemma \ref{bozza_lemma_rel_momenta} to the nowhere vanishing fields $u_1,u_2 \in u_0+H^1(\mathbf{R})$. We have that there exist $k_1,k_2 \in \mathbf{Z}$ such that
        \begin{equation*}
            \mathcal{P}(u_j) - P_{u_0}(u_j) = \pi k_j + P_{e^{i\varphi}}(u_0), \qquad j=1,2.
        \end{equation*}
    We deduce that 
        $$\mathcal{P}(u_1)- \mathcal{P}(u_2) = P_{u_0}(u_1) - P_{u_0}(u_2) - \pi (k_1-k_2),$$
        which allows us to the conclude the proof.
    \end{proof}
    \begin{remark}
        Notice that the parameter $k \in \mathbf{Z}$ in Lemma \ref{bozza_lemma_rel_momenta} depends on the (non unique) decomposition $u_0=e^{i\varphi}+w_0$. On the other hand, the parameter $k \in \mathbf{Z} $ in Lemma \ref{bozza_rel_energies_1} does not change if we change the decomposition of $u_0$.
    \end{remark}

\section{Global well-posedness}
\label{section: nbu}
In this section we prove Proposition \ref{bozza_intro_prop_cauchy}. More precisely, we prove the global well-posedness of the Cauchy problem associated to \eqref{bozza_NLS} and the conservation of energy in the space $\mathcal{E}$. In order to do that, we first consider the Cauchy problem in the space $X^2_{\gamma}\cap \mathcal{E}$ and we prove that it is globally well-posed. Then, by means of a density argument, we extend the global well-posedness to $\mathcal{E}$.\\

By means of Proposition \ref{bozza_persistence_regularity}, we know that the Cauchy problem associated to \eqref{bozza_NLS} is locally well-posed in $X^2_{\gamma}\cap\mathcal{E}$. Using the conservation of the energy $K_{u_0}$ given in Proposition \ref{bozza_cons_energy_X^2} we can extend this local result to a global one.
\begin{proposition}
    Let $u_0 \in X^2_{\gamma}\cap\mathcal{E}$. The unique maximal solution $u$ to (\ref{bozza_NLS}) with $u(0)=u_0$ given by Corollary \ref{bozza_corollary} is global, i.e. it satisfies $T_{\pm} = +\infty$.
    \label{bozza_nbu_X^2}
\end{proposition}
\begin{proof}
    Consider $u_0 \in X^2_{\gamma}\cap\mathcal{E}$, and $u$ the corresponding maximal solution, defined for $t \in (-T_-,T_+)$, where $T_{\pm} \in (0,+\infty]$.  By the blow-up alternative in Corollary \ref{bozza_corollary}, the solution $u$ is global if the limits $\lim_{t \to \pm T_{\pm}} |u(t)|^2$ are finite. The strategy of the proof consists in providing a bound on the map $t \to |u(t)|_{\mathcal{E}}^2$ by means of the conservation on energy and a Gr\"{o}nwall inequality. We begin with a preliminary estimate. For any $t \in (-T_-,T_+)$, we have
    \begin{equation*}
        \begin{split}
            |u(t)|^2_{\mathcal{E}} \leq E_{\gamma}(u(t)) &\leq K_{u_0}(u(t))+ \frac{|v|}{2}\int_{\mathbf{R}}|u(t)-u_0||\partial_xu(t)+\partial_xu_0|dx \\ & \leq E_{\gamma}(u_0) + \frac{|v|}{2}||u(t)-u_0||_{L^2(\mathbf{R})}\Bigl(||\partial_xu(t)||_{L^2(\mathbf{R})}+||\partial_xu_0||_{L^2(\mathbf{R})}\Bigr) \\ & 
            \leq E_{\gamma}(u_0) + \frac{|v|^2}{4\delta^2}||u(t)-u_0||^2_{L^2(\mathbf{R})}+\frac{\delta^2}{2}\Bigl(||\partial_xu(t)||^2_{L^2(\mathbf{R})}+||\partial_xu_0||^2_{L^2(\mathbf{R})}\Bigr)
        \end{split}
    \end{equation*}
    for any $\delta >0$, where we used the Cauchy-Schwarz and Young's inequalities. By choosing $\delta^2 <1$, we can find $C>0$ (depending only on $\delta$ and $|v|$) such that
    \begin{equation}
        |u(t)|^2_{\mathcal{E}} \leq C\Bigl(E_{\gamma}(u_0) +||u(t)-u_0||^2_{L^2(\mathbf{R})}\Bigr).
        \label{bozza_est_u-u_0}
    \end{equation}
    We now want to estimate the growth of $||u(t)-u_0||^2_{L^2(\mathbf{R})}$. Similarly to \cite{Gallo}, we compute 
    \begin{equation*}
        \begin{split}
            \frac{d}{dt}||u(t)-u_0||^2_{L^2(\mathbf{R})} &= 2\Re\int_{\mathbf{R}}(\overline{u}(t)-\overline{u}_0)\partial_tu(t)dx = 2\Re\int_{\mathbf{R}}(\overline{u}(t)-\overline{u}_0)\bigl(-i\tilde{H}_{\gamma}u(t)+iF(u(t))\bigr)dx \\ & = -2\Re\int_{\mathbf{R}}i (\partial_x\overline{u}(t)-\partial_x\overline{u}_0)\partial_xu(t)dx - 2\gamma\Re\Bigl(i(\overline{u}(t,0)-\overline{u}_0(0))u(t,0)\Bigr) \\ & +2v\Re\int_{\mathbf{R}}(\overline{u}(t)-\overline{u}_0)\partial_xu(t)+ 2\Re\int_{\mathbf{R}}i(\overline{u}-\overline{u}_0)u_0(1-|u(t)|^2)dx.
        \end{split}
    \end{equation*}
    We can write
    \begin{equation*}
       \begin{split}
           \frac{d}{dt}||u(t)-u_0||^2_{L^2(\mathbf{R})} &\leq ||\partial_xu_0||^2_{L^2(\mathbf{R})}+2|u(t)|^2_{\mathcal{E}} + 2C'\gamma |u_0(0)| (2+|u(t)|^2_{\mathcal{E}}) \\ & + |v| ||u(t)-u_0||^2_{L^2(\mathbf{R})}+2|v||u(t)|^2_{\mathcal{E}} +4 ||u_0||_{L^{\infty}(\mathbf{R})}|u(t)|^2_{\mathcal{E}} + ||u_0||_{L^{\infty}(\mathbf{R})}||u(t)-u_0||^2_{L^2(\mathbf{R})},
       \end{split}
    \end{equation*}
    where the constant $C'>0$ is given by Lemma \ref{bozza_L_infty}. By means of (\ref{bozza_est_u-u_0}), there exists a constant $C_1(u_0)>0$ such that
    \begin{equation*}
        \frac{d}{dt}||u(t)-u_0||^2_{L^2(\mathbf{R})} \leq C_1(u_0)(1+E_{\gamma}(u_0)+||u(t)-u_0||^2).
    \end{equation*}
    \begin{remark}
        The constant $C_1(u_0)$ is positive and it is a linear combination of $||\partial_xu_0||^2_{L^2(\mathbf{R})}$, $||u_0||_{L^{\infty}(\mathbf{R})}$ and $1+|v|$.
        \label{bozza_remark_C_1}
    \end{remark}
    By the Gr\"{o}nwall inequality we have
    \begin{equation}
        ||u(t)-u_0||^2_{L^2(\mathbf{R})} \leq (1+E_{\gamma}(u_0))(e^{C_1(u_0)t}-1), \qquad \forall t \in [0, T_+).
    \end{equation}
    Using (\ref{bozza_est_u-u_0}), we obtain
    \begin{equation}
        |u(t)|^2_{\mathcal{E}} \leq C (1+E_{\gamma}(u_0))e^{C_1(u_0)t}, \qquad \forall t \in [0, T_+),
        \label{bozza_estimate_growth}
    \end{equation}
    By the blow-up alternative, we conclude $T_{+} = +\infty$. By the time reversal symmetry of equation (\ref{bozza_NLS}), we conclude that $T_-=+\infty$.
\end{proof}
Using the fact that regular solutions are global in time, we can use a density argument to prove Proposition \ref{bozza_intro_prop_cauchy}. 
\begin{proof}[Proof of Proposition \ref{bozza_intro_prop_cauchy}]
    By Proposition \ref{bozza_lwp}, we know that the Cauchy problem associated to (\ref{bozza_NLS}) is locally well-posed in $\mathcal{E}$. Here we prove that every solution is global in time. The proof of the conservation of energy in $\mathcal{E}$ is postponed to the dedicated subsection \ref{section:en_cons}.\\ Consider $u_0 \in \mathcal{E}$ and let $u \in C^0((-T_-,T_+),\mathcal{E})$ the maximal solution to (\ref{bozza_NLS}) with $u(0)=u_0$ given by Corollary \ref{bozza_corollary}.  We want to prove that $T_{\pm} = +\infty$. Assume by contradiction $T_+ < +\infty$. Consider a sequence $\{u_{0,n}\}_{n \in \mathbf{N}} \subset X^2_{\gamma}\cap\mathcal{E}$ such that $d_{\infty}(u_0,u_{0,n}) \to 0$ as $n \to \infty$, as given by Lemma \ref{bozza_lemma_X^2}. For each $n \in \mathbf{N}$, denote by $u_n$ the unique global solution to (\ref{bozza_NLS}) with $u_n(0)=u_{0,n}$ given by Proposition \ref{bozza_nbu_X^2}.\\    
    By inequality (\ref{bozza_estimate_growth}) in Proposition \ref{bozza_nbu_X^2}, there exists $C>0$ such that $\forall n \in \mathbf{N}$ it holds
    \begin{equation*}
            |u_n(t)|^2_{\mathcal{E}}  \leq C(1+E_{\gamma}(u_{0,n}))e^{C_1(u_{0,n})t}, \qquad \forall t \geq 0
    \end{equation*}
    (recall that $C>0$ is independent of $u_{0,n}$, see \eqref{bozza_est_u-u_0}). By Remark \ref{bozza_remark_C_1}, the constant $C_1(\cdot)$ si continuous with respect to $d_{\infty}$. Together with the continuity of $E_{\gamma}$, we conclude the existence of a constant $K>0$ such that, $\forall n \in \mathbf{N}$, it holds
    \begin{equation*}
        |u_n(t)|^2_{\mathcal{E}} \leq Ke^{KT_+}, \qquad \text{for} \quad t \in [0,T_+].
    \end{equation*}
    Consider now $T>0$ with $T< T_+$. By continuity with respect to initial data, there exists $n_0 \in \mathbf{N}$ such that
    \begin{equation*}
        d_{\infty}(u(t), u_{n}(t)) \leq 1, \qquad \forall t \in [0,T].
    \end{equation*}
    for every $n \geq n_0$. By Lemma \ref{bozza_cont_norm}, there exists a constant $C'>0$ (independent of $T$) such that, for any $n \geq n_0$,
    \begin{equation*}
        |u(t)|_{\mathcal{E}}^2 \leq C' (2+|u_n(t)|^2_{\mathcal{E}}) \leq C'(2+Ke^{KT_+}), \qquad \forall t \in [0,T].
    \end{equation*}
    Being this inequality uniform in $T$, we conclude that $T_+=+\infty$. By the time reversal symmetry, we have that also $T_-=+\infty$. 
\end{proof}
\subsection{\texorpdfstring{Conservation of energy in $\mathcal{E}$}{Conservation of energy in mathcal E}}
\label{section:en_cons}
We conclude this section by showing the conservation of energy for solutions with initial data in $\mathcal{E}$. We begin by showing that a solutions with initial datum $u_0 \in \mathcal{E}$ conserves the energy $K_{u_0}$ defined in \eqref{bozza_K_u0}. Secondly, we show that, as long as a solution belongs to the set $U_0 = \{u \in \mathcal{E} \  | \ \inf_{x \in \mathbf{R}} |u(x)| >0 \}$, also the energy $\mathcal{K}$ defined in \eqref{bozza_energy_non_vanishing} is conserved.
\begin{lemma}
    Let $u_0 \in \mathcal{E}$, and let $u$ be the corresponding global solution to (\ref{bozza_NLS}), given by Proposition \ref{bozza_intro_prop_cauchy}. Then 
    $$K_{u_0}(u(t)) = K_{u_0}(u_0), \qquad \forall t \in \mathbf{R}.$$
    \label{bozza_energy_cons_law}
\end{lemma}
\begin{proof}
    Given $u_0 \in \mathcal{E}$, consider $u$ the solution to (\ref{bozza_NLS}) satisfying $u(0)=u_0$. Consider a sequence $\{u_{0,n}\} \subset X^2_{\gamma}\cap \bigl(u_0+H^1(\mathbf{R})\bigr)$, given by Lemma \ref{bozza_lemma_X^2}, such that $||u_0-u_{0,n}||_{H^1(\mathbf{R})} \to 0$ as $n \to \infty$. Call $u_n$ the solution to (\ref{bozza_NLS}) with $u_n(0) = u_{0,n}$. For any $t \in \mathbf{R}$, we have 
    $$K_{u_0}(u(t)) = \lim_{n \to \infty}K_{u_0}(u_{n}(t)) = \lim_{n\to \infty}K_{u_0}(u_{0,n}) = K_{u_0}(u_0),$$
    where the first equality follows from the continuity with respect to initial data (\ref{bozza_cont_in_data}) and the continuity of the energy (Lemma \ref{bozza_energy_cont}). The second equality follows from the conservation of energy in Corollary \ref{bozza_corollary_reference_field_K}; the third equality follows again from the continuity of the energy.
\end{proof}

\begin{lemma}
        Let $u_0 \in \mathcal{E}$ be such that $\inf_{x \in\mathbf{R}}|u_0(x)|>0$. Suppose $u \in C^0(\mathbf{R}, \mathcal{E})$ is the unique global solution to (\ref{bozza_NLS})  that satisfies $u(0) =u_0$, as given by Proposition \ref{bozza_intro_prop_cauchy}. Then, there exists a time interval $\{0\} \ni I \subset \mathbf{R}$, with $I \neq \{0\}$, for which $\inf_{x \in\mathbf{R}}|u(t,x)|>0$ for all $t\in I$. Given any time interval $I$ with this property, we have the conservation law
        $$\mathcal{K}(u(t)) = \mathcal{K}(u_0), \qquad \forall t \in I.$$
        \label{bozza_lemma_en_cons_nonvanishing}
    \end{lemma}
    \begin{proof}
    Since the solution map $t \to u(t)$ belongs to $C^{0}(\mathbf{R}, \mathcal{E})$, the map $t \to \inf_{x \in \mathbf{R}}|u(t,x)|$ is a continuous function from $\mathbf{R}$ to $\mathbf{R}_+$. If the initial datum $u_0 \in \mathcal{E}$ is nowhere vanishing, there exists a time interval $\{0\} \ni I \subset \mathbf{R}$, with $I \neq \{0\}$, such that $\inf_{x \in \mathbf{R}}|u(t,x)|>0$ for all $t \in I$.\\
    Given $t \in I$, we have $\inf_{x \in \mathbf{R}}|u(t,x)| >0$. Moreover, we have $u(t)-u_0 \in H^1(\mathbf{R})$, from Proposition \ref{bozza_u-u0_prop}. By Lemma \ref{bozza_rel_energies_1}, and the conservation of energy in Lemma \ref{bozza_energy_cons_law}, there exists $k \in \mathbf{Z}$ such that $\mathcal{K}(u(t))-\mathcal{K}(u_0) = vk\pi$. Since the map $t \to \mathcal{K}(u(t))$ is continuous for all $t \in I$ (see Lemma \ref{bozza_en_K_continuity}), we have $\mathcal{K}(u(t))-\mathcal{K}(u_0) = 0$ for all $t \in I$.
    \end{proof}

\section{Stationary solutions}
In this section we study the existence of stationary solutions to equation (\ref{bozza_NLS}). From now on we will suppose that the velocity $v$ is a positive real number ($v>0$). Similar results can be obtained for the case $v<0$.\\ By stationary solution we mean an element $u \in  \mathcal{E}$ that satisfies
\begin{equation}
   \int_{\mathbf{R}} \partial_xu\partial_x\overline{\phi}dx + \int_{\mathbf{R}} iv\partial_xu\overline{\phi}dx +\gamma u(0)\overline{\phi}(0) - \int_{\mathbf{R}} (1-|u|^2)u\overline{\phi}dx = 0, \qquad \forall\phi \in C^{\infty}_0(\mathbf{R}).
   \label{bozza_stationary_sol}
\end{equation}
We have the following Lemma.
\begin{lemma}
    Let $\gamma >0$ and let $u \in \mathcal{E}$ be a solution of (\ref{bozza_stationary_sol}). Then,
    \begin{equation}
        \begin{split}
            &u \in C^{\infty}(\mathbf{R}\backslash\{0\}) \cap C^0(\mathbf{R}); \\ &
        \partial_x^2u-iv\partial_xu+(1-|u|^2)u = 0 \quad  \text{on} \quad \mathbf{R}\backslash\{0\};\\ &
        \partial_xu(0^+)-\partial_xu(0^-) = \gamma u(0).
        \end{split}
        \label{bozza_lemma_stationary_sol}
    \end{equation} 
    \label{bozza_lemma_stationary_sol_2}
    \end{lemma}
    \begin{proof}
        Since $u \in \mathcal{E}$ we have $u \in C^0(\mathbf{R})$. By considering $\phi \in C^{\infty}_0(\mathbf{R}\backslash\{0\})$ in (\ref{bozza_stationary_sol}), we have that $\partial_x^2u-iv\partial_xu+(1-|u|^2)u=0$ in $\mathbf{R}\backslash\{0\}$ in the sense of distribution. We conclude that $u$ is smooth and it is a classical solution of $\partial_x^2u-iv\partial_xu+(1-|u|^2)u=0$ in $\mathbf{R}\backslash\{0\}$. Integrating by parts the first term in equation in (\ref{bozza_stationary_sol}) we obtain the jump condition $\partial_xu(0^+)-\partial_xu(0^-) = \gamma u(0)$.
    \end{proof}
     Following the reasoning in the work of Mari\c{s} \cite{maris2003}, we can deduce some necessary conditions for the existence of stationary solutions, which are reported in the next proposition. 
    \begin{proposition}
        Let $\gamma>0$. If $v \geq \sqrt{2}$, then no stationary solution exists. If $u \in \mathcal{E}$ is a stationary solution for $v \in (0,\sqrt{2})$, then it is nowhere vanishing and $ |u(x)| \in \Bigl[\frac{v}{\sqrt{2}},1\Bigr]$ for all $x \in \mathbf{R}$.
        \label{bozza_prop_stat_nonvanishing}
    \end{proposition} 
    \begin{proof}
        Let $v >0$ and assume that $u \in \mathcal{E}$ is a stationary solution. Since $|u(x)|\to 1$ as $x \to +\infty$, there exists $A>0$ such that $|u(x)|>0$ for all $x \in (A,+\infty)$. Within this interval, we write $u(x) = (1+r(x))e^{i\theta(x)}$, for $r,\theta$ real valued functions such that $r(x) \in (-1,+\infty)$ for all $x \in (A,+\infty)$ and $r \to 0$ and $\partial_x\theta \to 0$ as $x \to +\infty$. If we insert this expression in (\ref{bozza_lemma_stationary_sol}), we obtain, after some manipulations,
        \begin{equation}
            \partial_x\theta = \frac{v}{2}\Bigl(1-\frac{1}{(1+r)^2}\Bigr)
            \label{bozza_eq_theta}
        \end{equation}
        and
        \begin{equation}
            -\partial_x^2r-(1+r)+(1+r)^3-\frac{v^2}{4}\Bigl(1+r-\frac{1}{(1+r)^3}\Bigr) =0.
            \label{bozza_eq_r}
        \end{equation}
        We multiply (\ref{bozza_eq_r}) by $2\partial_xr$ and we integrate. Using the fact that $r\to 0$ at infinity, we obtain
        \begin{equation}
            (\partial_xr)^2 = r^2(r+2)^2\Bigl(\frac{1}{2}-\frac{v^2}{4}\frac{1}{(1+r)^2}\Bigr) = \bigl(f(r)\bigr)^2,
            \label{bozza_eq_r_prime}
        \end{equation}
        where we defined \begin{equation}
            f(r) := r(r+2)\sqrt{\frac{1}{2}-\frac{v^2}{4}\frac{1}{(1+r)^2}}.
            \label{bozza_f_ode}
        \end{equation}
        By the positivity of the left-hand side of (\ref{bozza_eq_r_prime}), we conclude that $r$ cannot take value in $(-1,-1+\frac{v}{\sqrt{2}})\backslash\{0\}$. This implies that $|u(x)| \geq \min\{1,\frac{v}{\sqrt{2}}\}$ for all $x \in (A,+\infty)$. This allows us to conclude that $|u(x)| \geq \min\{1,\frac{v}{\sqrt{2}}\}$ on the half-line $(0,+\infty)$. We can repeat the same argument on the negative real axis, and conclude that $|u(x)| \geq \min\{1,\frac{v}{\sqrt{2}}\}$ for all $x \in \mathbf{R}$. In particular, we can write $u(x) = (1+r(x))e^{i\theta(x)}$ for all $x \in \mathbf{R}$, with $r, \theta$ real valued and globally defined functions such that $r(x) \to 0$ and $\partial_x\theta(x) \to 0$ as $|x| \to +\infty$. Moreover, $r$ and $\theta$ are continuous functions and, from the jump condition on $u$ in (\ref{bozza_lemma_stationary_sol}), we have $\partial_x\theta \in C^0(\mathbf{R})$ and $r \in C^{\infty}(\mathbf{R}\backslash\{0\})\cap C^0(\mathbf{R})$.\\ Suppose now $|v|> \sqrt{2}$. Since $r$ cannot take values in $(-1,-1+\frac{v}{\sqrt{2}})\backslash\{0\}$, the condition $r(x) \to 0$ as $|x| \to \infty$ implies that $r(x) = 0$ for all $x \in \mathbf{R}$. Thus, there exists $\theta_0 \in \mathbf{R}$ such that $u = e^{i\theta_0}$. This is in contradiction with the assumption that $u$ is a stationary solution, as the jump condition in (\ref{bozza_lemma_stationary_sol}) is not satisfied. We conclude that for $|v| > \sqrt{2}$ there are no stationary solutions. Assume $|v| \leq \sqrt{2}$. In this case we know that $r \geq -1+\frac{v}{\sqrt{2}}$. We want to show that $r(x) \leq 0$ for all $x \in \mathbf{R}$. By contradiction, suppose this is not true. Consider the function $x \to \psi_v(x) = -(1+x)+(1+x)^3-\frac{v^2}{4}(1+x-\frac{1}{(1+x)^3})$, which is strictly increasing  and positive on $(0,+\infty)$. If we suppose that $r$ achieves a positive maximum in $x_0 \in \mathbf{R}\backslash\{0\}$, then $\partial_x^2r(x_0)\leq0$. On the other hand, from (\ref{bozza_eq_r}) we have $\partial_x^2r(x_0) = \psi_v(r(x_0)) >0$, which is a contradiction. Suppose instead that $r$ achieves a positive maximum at $x_0=0$. Then, it is $\partial_xr(0^+)\leq0$ and $\partial_xr(0^-)\geq0$. We obtain, by the jump condition on $u \in \mathcal{E}$, that $$\gamma = \frac{\partial_xr(0^+)-\partial_xr(0^-)}{1+r(0)} \leq 0,$$
        which is a contradiction. We conclude that for $v \leq \sqrt{2}$ we have $r(x)\leq 0$ for all $x \in \mathbf{R}$. In the case $v = \sqrt{2}$, we conclude $r = 0$. As before, this leads to $u = e^{i\theta_0}$, for $\theta_0 \in \mathbf{R}$, which is a contradiction. Thus, there are no stationary solutions for $v = \sqrt{2}$. If $0<v < \sqrt{2}$, we have $-1+\frac{v}{\sqrt{2}}\leq r(x) \leq 0$.
    \end{proof}
    In the regime $v \in (0,\sqrt{2})$, the existence and the explicit expression of stationary solutions has been determined by Hakim \cite{hakim1997nonlinear} and Mari\c{s} \cite{maris2003} with two different approaches. The analysis of Mari\c{s} appears to be convenient for the study of the stability of stationary states. For the sake of the exposition, in the following we report part of this analysis.\\
    As explained in Lemma \ref{bozza_lemma_stationary_sol_2}, a stationary solution solves, for a given velocity $v$, the traveling wave equation \eqref{bozza_traveling_w_equation} on $\mathbf{R} \backslash\{0\}$. At $x=0$, instead, it satisfies the jump condition in \eqref{bozza_jump_intro}, determined by the delta potential. As detailed in Proposition \ref{bozza_intro_hakim}, one way to construct stationary solutions is the following. First we modify the traveling wave profile $r(x)$ in \eqref{bozza_tw_mod}, by inserting a displacement parameter $\xi$. More precisely, we define the function $r(x;\xi)$ as $r(x;\xi) = r(x-\xi)$ for $x<0$ and $r(x;\xi) = r(x+\xi)$ for $x>0$. Secondly, we define the phase function $\theta(x;\xi)$ by \eqref{bozza_tw_phase}. Eventually, the field $u(x;\xi)=(1+r(x;\xi))e^{i\theta(x;\xi)}$ solves \eqref{bozza_traveling_w_equation} on $\mathbf{R}\backslash\{0\}$, and it is a stationary solution provided that $\xi$ satisfies equation \eqref{bozza_matching}. The latter equation ensures the correct jump condition for $u(x;\xi)$ at $x=0$, as shown by Hakim \cite{hakim1997nonlinear}. \\
    In \cite{maris2003}, Mari\c{s} obtains the stationary solutions with a slightly different approach, which we now describe. Consider the parameter $a \in [-1+\frac{v}{\sqrt{2}},0]$. Then, for $a \in [-1+\frac{v}{\sqrt{2}},0)$ define
    \begin{equation}
        r_a(x) = -1+\sqrt{\frac{v^2}{2}+(1-\frac{v^2}{2})\tanh^2\bigl(\frac{1}{2}\sqrt{2-v^2}(x\mp c(a))\bigr)}, \qquad \text{for} \quad \pm x\geq0,  
        \label{bozza_eq_def_r_a}
    \end{equation}
     where 
     \begin{equation}
         c(a) = \frac{1}{\sqrt{2-v^2}}\ln\frac{\sqrt{2-v^2}-\sqrt{2(a+1)^2-v^2}}{\sqrt{2-v^2}+\sqrt{2(a+1)^2-v^2}}.
         \label{bozza_c(a)}
     \end{equation}
    While for $a = 0$ define $r_a = 0$.\\
    The functions $r_a$ have the same structure as the functions $r(x;\xi)$ in \eqref{bozza_intro_stat_sol}, but now the displacement $c(a)$ depends on the free parameter $a$.
    Moreover, the functions $r_a$ are the unique solutions to the following Cauchy problem: 
     \begin{equation}
        \begin{cases}
            \partial_xr(x) = f(r(x)) \qquad \text{on} \quad (-\infty,0]\\
           \partial_xr(x) = -f(r(x)) \qquad \text{on} \quad [0,+\infty)\\
           r(0) = a,
        \end{cases} 
        \label{bozza_cauchy_stationary}
    \end{equation}
    where $a \in [-1+\frac{v}{\sqrt{2}},0]$ and the function $f(r)$ is defined in (\ref{bozza_f_ode}). This consideration allows us to deduce the following properties \cite[Lemma 2.2]{maris2003}:
       \begin{itemize}
           \item [i)] each $r_a$ is symmetric about the origin. It is increasing for $x \geq 0$ and decreasing for $x \leq 0$, and tends to zero as $|x| \to \infty$;\\
           \item[ii)] $\inf_{x \in \mathbf{R}}r_a(x) = a;$\\
           \item[iii)] $r_a \in C^{\infty}(\mathbf{R}\backslash\{0\})$, with $\partial_xr_a(0^+)-\partial_xr_a(0^-) = -2f(r_a(0))$. \\
           \item[iv)]If $ a=0$ or $a = -1+\frac{v}{\sqrt{2}}$ we have $r_a \in C^{\infty}(\mathbf{R})$.
       \end{itemize} 
       Following the previous discussion, we now introduce the fields $u_a$.
       \begin{definition}
           For each $a \in [-1+\frac{v}{\sqrt{2}},0]$ we define the element $u_a \in \mathcal{E}$ as
       \begin{equation}
         u_a(x) := (1+r_a(x))e^{i\theta_a(x)},
         \label{bozza_def_u_a}
     \end{equation}
     where $\theta_a:\mathbf{R} \to \mathbf{R}$ is the unique solution to 
     $$\partial_x\theta_a(x) = \frac{v}{2}\Bigl(1-\frac{1}{(1+r_a(x))^2}\Bigr), \qquad \text{with} \qquad \theta_a(0) = 0.$$
       \end{definition}
    The functions $u_a$ solve the traveling wave equation \eqref{bozza_traveling_w_equation} in $\mathbf{R}\backslash\{0\}$. By tuning the parameter $a$, we can impose the correct jump condition in $x=0$, and hence obtain stationary solutions. This is described in the following proposition, which is the analogous of Proposition \ref{bozza_intro_hakim}.
    \begin{proposition}[\cite{maris2003}, Chapter 3]
        Let $\gamma >0$ and $v \in (0,\sqrt{2})$. A function $u \in \mathcal{E}$ is a stationary solution of (\ref{bozza_NLS}) if and only if there exists $a \in \Bigl(-1+\frac{v}{\sqrt{2}},0\Bigr)$ and $\phi \in \mathbf{R}$ such that:
        \begin{equation}
             \quad  u(x) = u_a(x)e^{i\phi} \qquad \text{and} \qquad   \quad \gamma = -\frac{2f(a)}{a+1}
              \label{bozza_eq_maris}
        \end{equation}
        where $f$ is given in (\ref{bozza_f_ode}).
        \label{bozza_prop_r_a}
    \end{proposition}
    \begin{remark}
        By means of Proposition \ref{bozza_prop_r_a} we can label the stationary solutions with the parameter $a \in (-1+\frac{v}{\sqrt{2}},0)$. This parameter is linked to the minimal density of a stationary solution by the relation $\inf_{x \in \mathbf{R}}|u_a(x)| = 1+a$, thanks to point ii) above. As we will see in the next subsection, this property will be useful in giving a variational characterization of the stationary states.
    \end{remark}
    \begin{proof}[Proof of Proposition \ref{bozza_prop_r_a}]
        Set $\gamma >0$ and $v \in (0,\sqrt{2})$. Consider $u_a \in \mathcal{E}$ for $a \in \Bigl(-1+\frac{v}{\sqrt{2}},0\Bigr)$. On $(0,+\infty)$ we have $\partial_x^2u_a-iv\partial_xu_a+(1-|u_a|^2)u_a=0$. The same is true on $(-\infty,0)$. Then, 
        \begin{equation*}
        \frac{\partial_xu_a(0^+)-\partial_xu_a(0^-)}{u_a(0)} = \frac{\partial_xr_a(0^+)-\partial_xr_a(0^-)}{1+r(0)} = -\frac{2f(a)}{a+1}.
    \end{equation*}
    Thus $u_a$ is a stationary solution, provided that $\gamma = -\frac{2f(a)}{a+1}$. The same is true for $e^{i\phi}u_a$, for any $\phi \in \mathbf{R}$.\\Suppose now that $u \in \mathcal{E}$ is a stationary solution. From Proposition \ref{bozza_prop_stat_nonvanishing}, we have that $\frac{v}{\sqrt{2}} \leq |u(x)|\leq 1$, for all $x \in \mathbf{R}$. We write $u(x) = (1+r(x))e^{i\theta(x)}$, where $r$ and $\theta$ are continuous and real valued functions, with $r(x) \to 0$  and $\partial_x\theta(x) \to 0$ as $|x| \to \infty$. The jump condition $\partial_xu(0^+)-\partial_xu(0^-) = \gamma u(0)$, implies that $\partial_x\theta$ is continuous on $\mathbf{R}$ and that $\partial_xr(0^+)-\partial_xr(0^-) = \gamma (1+r(0))$. As in the proof of Proposition \ref{bozza_prop_stat_nonvanishing}, the identity   $\partial_x^2u-iv\partial_xu+(1-|u|^2)u=0$ on $\mathbf{R}\backslash\{0\}$ implies
    \begin{equation*}
        (\partial_xr)^2 = f^2(r), \qquad \text{on} \quad (-\infty,0)\cup(0,+\infty),
    \end{equation*}
    and
    \begin{equation}
        \partial_x\theta = \frac{v}{2}\Bigl(1-\frac{1}{(1+r)^2}\Bigr) \qquad \text{on} \quad \mathbf{R}.
        \label{bozza_comp_theta}
    \end{equation}
    Let $a = r(0)$, with $a \in [-1+\frac{v}{\sqrt{2}},0]$. If $a =0$ or $a = -1+\frac{v}{\sqrt{2}}$, then $\partial_xr(0) \to 0$ as $x \to 0$, i.e. $\partial_xr$ is a continuous function on $\mathbf{R}$. This would require $\gamma =0$, which is a contradiction. It necessarily is $r(0) = a \in (-1+\frac{v}{\sqrt{2}},0)$. Define
    \begin{equation*}\begin{split} 
        x_1 &= \inf\{x <0 | \ r \neq 0, \ r \neq -1+\frac{v}{\sqrt{2}} \ \text{on} \ (x,0)\} \\ 
        y_1 &= \sup\{y >0 | \ r \neq 0, \ r \neq -1+\frac{v}{\sqrt{2}} \ \text{on} \ (0,y)\}.
    \end{split}
    \end{equation*} 
    By continuity we have $x_1 <0$ and $y_1 >0$. Since function $\partial_xr$ is continuous on $(x_1,0)$, and since $f(r) \neq 0$ if $r \notin \{-1+\frac{v}{\sqrt{2}},0\}$, we have either $\partial_xr(x)=f(r(x))$ for all $x \in (x_1,0)$ or $\partial_xr(x) = -f(r(x))$ for all $x \in (x_1,0)$. The same holds in the interval $(0,y_1)$. From the relation 
    \begin{equation}
        0<\gamma = \frac{\partial_xu(0^+)-\partial_xu(0^-)}{u(0)}= \frac{\partial_xr(0^+)-\partial_xr(0^-)}{1+r(0)},
        \label{bozza_comp_jump}
    \end{equation}
    we deduce that necessarily is $\partial_xr= f(r)$ on $(x_1,0)$ and $\partial_xr=-f(r)$ on $(0,y_1)$. All other combinations would give $\gamma \leq 0$. By uniqueness, we deduce that $x_1=-\infty$ and $y_1 = +\infty$, and that $r = r_a$. Using (\ref{bozza_comp_theta}), we deduce that there exists $\phi \in \mathbf{R}$ such that $u = e^{i\phi}u_a$. Moreover, from the relation in (\ref{bozza_comp_jump}) we obtain that $a \in (-1+\frac{v}{\sqrt{2}},0) $ has to be a solution of $\gamma = -\frac{2f(a)}{a+1}$. 
    \end{proof}
    In view of Proposition \ref{bozza_prop_r_a}, the existence of stationary solutions for (\ref{bozza_NLS}) can be seen as the existence of roots for the equation $\gamma = -\frac{2f(a)}{a+1}$.  
    For given $\gamma >0$ and $v \in (0,\sqrt{2})$, we define the map $k_v:\Bigl[-1+\frac{v}{\sqrt{2}},0\Bigr] \to \mathbf{R}$ as 
    \begin{equation}
        k_v(a) = -\frac{2f(a)}{a+1}
        \label{bozza_def_k_v}
    \end{equation} and we look for solutions of $\gamma = k_v(a)$. The function $k_v(a)$ is increasing for $a \in [-1+\frac{v}{\sqrt{2}}, a_*]$ and decreasing for $a \in [a_*,0]$, with $k(-1+\frac{v}{\sqrt{2}}) = k_v(0) = 0$, and has a maximum at $a_* = -1+\sqrt{\frac{-1+\sqrt{1+4v^2}}{2}}$. Moreover, for a given $v$, the maximum value that $k_v$ attains is
    \begin{equation}
        \varphi(v) := k_v(a_*) = \frac{(1+\sqrt{1+4v^2}-2v^2)\sqrt{2-v^2}}{2v\sqrt{1+v^2+\sqrt{1+4v^2}}}.
        \label{bozza_def_varphi}
    \end{equation}
    The function $\varphi$ is continuous and strictly decreasing on $(0,\sqrt{2}]$, with $\lim_{v \downarrow0}\varphi(v) = +\infty$ and $\varphi(\sqrt{2}) =0$. Hence, its inverse $\varphi^{-1}$ exists and it is strictly decreasing, with $\varphi^{-1}(0) = \sqrt{2}$ and $\lim_{\gamma \to +\infty}\varphi^{-1}(\gamma) = 0$. From the discussion above, we deduce the following proposition, which was obtained in \cite[Proposition 3.1]{maris2003}.
    \begin{proposition}
        The following assertions hold:
        \begin{itemize}
            \item [i)] for a fixed velocity $v \in (0,\sqrt{2})$, there are exactly two stationary solutions if $\gamma \in (0,\varphi(v))$, where $\varphi$ is given in (\ref{bozza_def_varphi}). If $\gamma = \varphi(v)$ there exists only one stationary solution. If $\gamma > \varphi(v)$ there are no stationary solutions;
            \item[ii)] conversely, fix $\gamma >0$. If $0<v < \varphi^{-1}(\gamma)$, we have exactly two stationary solutions of velocity $v$. There is only one stationary solution at velocity $v = \varphi^{-1}(\gamma)$ and there are no stationary solutions of velocity $v > \varphi^{-1}(\gamma)$.
        \end{itemize}
        \label{bozza_maris}
    \end{proposition}
        In the case  $v \in (0,\sqrt{2})$ and $\gamma \in(0,\varphi(v))$, the equation $\gamma = k_v(a)$ has two solutions in $(-1+\frac{v}{\sqrt{2}},0)$, which we denote as $a_1$ and $a_2$, with 
        \begin{equation}
            a_2<a_1.
            \label{bozza_remark_a12}
        \end{equation} 
        In particular, it holds $a_2 \in (-1+\frac{v}{\sqrt{2}},a_*)$ and $a_1 \in (a_*,0)$.
        Using Proposition \ref{bozza_prop_r_a}, these solutions determine the two stationary solutions to (\ref{bozza_NLS}), which read
        \begin{equation}
            u_{a_1}(x) = (1+r_{a_1}(x))e^{i\theta_{a_1}(x)}, \qquad  \text{and} \qquad u_{a_2}(x) = (1+r_{a_2}(x))e^{i\theta_{a_2}(x)}.
            \label{bozza_u_a1_and_u_a2}
        \end{equation}
    \begin{remark}
    The statements in Proposition \ref{bozza_maris} imply the ones in Corollary \ref{bozza_corollary_intro}. In particular, it holds $v_{cr}(\gamma) = \varphi^{-1}(\gamma)$. Moreover, consider the relation in \eqref{bozza_matching}. If we insert the identity $\xi = -c(a)$, where $c(a)$ is given in \eqref{bozza_c(a)}, then equation \eqref{bozza_matching} becomes $\gamma = -\frac{2f(a)}{a+1}$. In the regime $v < v_{cr}(\gamma)$, this implies that the roots $\xi_1,\xi_2$ of \eqref{bozza_matching} satisfy $\xi_2 = -c(a_2)$ and $\xi_1 = -c(a_1)$. We conclude the identities
    \begin{equation}
        u(x;\xi_1) = u_{a_1}(x) \qquad \text{and} \qquad u(x;\xi_2) = u_{a_2}(x).
    \end{equation}
    \label{bozza_remark_xi_1_a_1}
    \end{remark}
    In the following, we focus on the case $v \in (0,\sqrt{2})$ and $\gamma \in (0,\varphi(v))$, where two stationary solutions exist. We prove that the solution $u_{a_1}$ is stable with respect to the \eqref{bozza_NLS} dynamics. This is expected since, as we will see, the energy of $u_{a_1}$ is lower than that of $u_{a_2}$. 
    \subsection{Variational characterization of stationary states}
        In this subsection we give a variational characterization of the stationary states in terms of the functional $\mathcal{K}$ defined in (\ref{bozza_energy_non_vanishing}). In particular, we show that the state $u_{a_1}$ in \eqref{bozza_u_a1_and_u_a2} is a local minimizer of $\mathcal{K}$. This characterization, will be useful in the next section, where we prove the stability of $u_{a_1}$.\\

    We begin with the following lemma, which shows that the energy $\mathcal{K}$ is unbounded from below in its domain.
    \begin{lemma}
        Denote $U_0 = \{u \in \mathcal{E}, \ \inf_{x \in \mathbf{R}}|u(x)| >0\}$. For any $v >0$ and $\gamma >0$ we have
        \begin{equation*}
            \inf_{u \in U_0}\mathcal{K}(u) = -\infty.
        \end{equation*}
        \label{bozza_K_unbounded}
    \end{lemma}
    \begin{proof}
        Consider $\varepsilon >0$ and a function $u \in \mathcal{E}$ such that $\inf_{x \in \mathbf{R}}|u(x)| = u(0) = \varepsilon$. We write $u(x) = \rho(x)e^{i\theta(x)}$, for $\rho,\theta:\mathbf{R} \to \mathbf{R}$ such that $\partial_x\rho, \partial_x\theta \in L^2(\mathbf{R})$ and $\rho^2-1 \in L^2(\mathbf{R})$. For $\Lambda >0$, we define a new function $u_{\Lambda}$ defined as
        \begin{equation*}u_{\Lambda}(x) = 
            \begin{cases}
                u(x+\frac{\Lambda}{2})e^{i\phi(-\Lambda/2)} \quad \text{for} \quad x \in (-\infty,-\frac{\Lambda}{2}],\\
                \varepsilon^{i\phi(x)} \quad \text{for} \quad x \in [-\frac{\Lambda}{2},\frac{\Lambda}{2}],\\
                u(x-\frac{\Lambda}{2})e^{i\phi(\Lambda/2)} \quad \text{for} \quad x \in [\frac{\Lambda}{2},+\infty),
            \end{cases}
        \end{equation*}
        where $\phi$ is a real valued function such that $\partial_x\phi \in L^2_{loc}(\mathbf{R})$. We choose $\phi$ such that $\partial_x\phi = -\frac{2}{v}$. Then, $u_{\Lambda} \in \mathcal{E}$ and we have
        \begin{equation*}
            \mathcal{K}(u_{\Lambda}) = \mathcal{K}(u)+\mathcal{K}_2(\varepsilon^{i\phi(x)}),
        \end{equation*}
        where
        \begin{equation*}
            \begin{split}
                \mathcal{K}_2(\varepsilon^{i\phi(x)}) &= \frac{1}{2}\int_{[-\frac{\Lambda}{2},\frac{\Lambda}{2}]}\varepsilon^2|\partial_x\phi(x)|^2dx + \frac{1}{4}\int_{[-\frac{\Lambda}{2},\frac{\Lambda}{2}]}(1-\varepsilon^2)^2dx - \frac{v}{2}\int_{[-\frac{\Lambda}{2},\frac{\Lambda}{2}]}(\varepsilon^2-1)\partial_x\phi(x) dx \\ & = \bigl[\varepsilon^2\frac{2}{v^2} + \frac{1}{4}(1-\varepsilon^2)^2 - (1-\varepsilon^2) \bigr]\Lambda.
            \end{split}
        \end{equation*}
        Choose $\varepsilon$ small enough, in particular such that $(\frac{2}{v^2}+\frac{1}{2})\varepsilon^2 + \frac{1}{4}\varepsilon^4 < 3/4$. Then, we have $\mathcal{K}(u_{\Lambda}) \to -\infty$ as $\Lambda \to +\infty$.
    \end{proof}
    In order to give a variational characterization of the stationary states, it's convenient to consider the functional $\mathcal{K}$ under a constraint. Following \cite{maris2003}, we consider $\mathcal{K}$ restricted to the set of functions with prescribed minimal density, and we minimize $\mathcal{K}$ in this set. Therefore we define, for $a \in [-1+\frac{v}{\sqrt{2}},0]$, the function 
        \begin{equation}
            h(a) := \inf\{\mathcal{K}(u)| \ u \in \mathcal{E}, \ \inf_{x\in \mathbf{R}}|u(x)| = a+1\}.
            \label{bozza_def_h(a)}
        \end{equation} 
        We have the following proposition (recall the definition of $u_a$ from \eqref{bozza_def_u_a}).
    \begin{proposition} [\cite{maris2003}]
    The map $a \to h(a)$ is differentiable from $[-1+\frac{v}{\sqrt{2}},0]$ to $\mathbf{R}$, with $h(0) = \gamma /2$ and derivative
        \begin{equation}
            h'(a) = (a+1)(\gamma -k_v(a)), \label{bozza_formula_h'}
        \end{equation}
        where $k_v(a)$ is defined in (\ref{bozza_def_k_v}). Moreover, $h(a) = \mathcal{K}(u_a)$ for all $a \in [-1+\frac{v}{\sqrt{2}},0]$ and, up to phase shifts, $u_a$ is the unique minimizer.
        \label{bozza_prop_h(a)}
    \end{proposition}
    \begin{proof}
        Set $a \in [-1+\frac{v}{\sqrt{2}},0)$. Being $a+1<1$, if $u \in \mathcal{E}$ satisfies $\inf_{x\in \mathbf{R}}|u(x)| = a+1$, then there exists $x_0 \in \mathbf{R}$ such that $|u(x_0)| = a+1$. For this reason, we can restrict the minimization problem to functions $u \in \mathcal{E}$ that satisfy $\inf_{x \in \mathbf{R}}|u(x)| = |u(0)|$. In other words we have, for $a \in [-1+\frac{v}{\sqrt{2}},0)$,  $$h(a) = \inf\{\mathcal{K}(u)| \ u \in \mathcal{E}, \ \inf_{x \in \mathbf{R}}|u(x)| = |u(0)| = a+1\}.$$ Consider $u \in \mathcal{E}$ such that $\inf_{x \in \mathbf{R}}|u(x)| = |u(0)| = a+1$. We write $u(x) = \rho(x)e^{i\theta(x)}$, for $\rho,\theta:\mathbf{R}\to \mathbf{R}$ such that $\partial_x\rho,\partial_x\theta \in L^2(\mathbf{R})$, $\rho^2-1 \in L^2(\mathbf{R})$ and $\inf_{x \in \mathbf{R}}\rho(x) = \rho(0) = a+1$. By using this expression for $u$, we have  $\mathcal{K}(u) = G(u)+\frac{\gamma}{2}(a+1)^2$, where
        \begin{equation*}
            G(u) = \frac{1}{2}\int_{\mathbf{R}} (\partial_x\rho)^2 + \rho^2 (\partial_x\theta)^2 dx + \frac{1}{4} \int_{\mathbf{R}}(1-\rho^2)^2dx - \frac{v}{2}\int_{\mathbf{R}}(\rho^2-1)\partial_x\theta dx.
        \end{equation*} Using Cauchy-Schwarz and Young's inequality, we have
        \begin{equation}
                \begin{split}
                    G(u) &\geq \frac{1}{2}\int_{\mathbf{R}} (\partial_x\rho)^2 + \rho^2 (\partial_x\theta)^2 dx + \frac{1}{4} \int_{\mathbf{R}}(1-\rho^2)^2dx - \frac{v}{2}\Bigl|\int_{\mathbf{R}}\frac{(\rho^2-1)}{\rho}\rho\partial_x\theta dx \Bigr| \\ & \geq  \frac{1}{2}\int_{\mathbf{R}} (\partial_x\rho)^2 dx + \frac{1}{4} \int_{\mathbf{R}}(1-\rho^2)^2dx -\frac{v^2}{8}\int_{\mathbf{R}} \frac{(\rho^2-1)^2}{\rho^2}dx =: G_2(\rho).
                \label{bozza_comp_G}
                \end{split}
        \end{equation}
        In particular, equality in (\ref{bozza_comp_G}) holds if and only if 
        \begin{equation}
            \partial_x\theta = \frac{v}{2}\bigl(1-\frac{1}{\rho^2}\bigr).
            \label{bozza_comp_theta_n}
        \end{equation}
        Assuming $u(x) = \rho(x) e^{i\theta(x)}$, with $\theta(x)$ satisfying (\ref{bozza_comp_theta_n}), we have $G(u) = G_2(\rho)$, and we are left with the problem of minimizing $G_2(\rho)$ among all positive $\rho$ that satisfy $\inf_{x \in \mathbf{R}}\rho(x)=\rho(0) = a+1$. This problem has been solved in \cite[Lemma 2.2 - iv)]{maris2003}. There, it is proven that the infimum of $G_2(\rho)$ with the constraint $\rho(0)=a+1$ is achieved uniquely by $\rho(x) = 1+r_a(x)$, where $r_a$ is defined in (\ref{bozza_eq_def_r_a}). We conclude that $\mathcal{K}(u_a) = h(a)$, for all $a \in [-1+\frac{v}{\sqrt{2}},0)$, and that $u_a$ is the unique minimizer, up to phase shifts.\\
        On the other hand, if $a =0$, then $G(u)$ has, up to phase shifts, $u = 1$ as a unique minimizer. We conclude that $\mathcal{K}(u_{a=0}) = h(0) = \gamma/2$, where we recall that $u_{a=0}=1$.\\
        Finally, the explicit expression of the map $a \to h(a)$ has been computed in \cite[Remark 3.2]{maris2003}, where\footnote{The functional $G_2$ in (\ref{bozza_comp_G}) coincides with the functional that is minimized in \cite{maris2003}, up to a normalization factor of 2. Similarly, the map $a \to h(a)$ coincides with the one defined in \cite{maris2003}, up to a normalization factor of 2 and an additional constant of value $\gamma$.} $a \to h(a)$ is shown to be a differentiable map with $h'(a)$ given by (\ref{bozza_formula_h'}).
    \end{proof}
    \begin{figure}
    \centering
    \includegraphics[width=\linewidth]{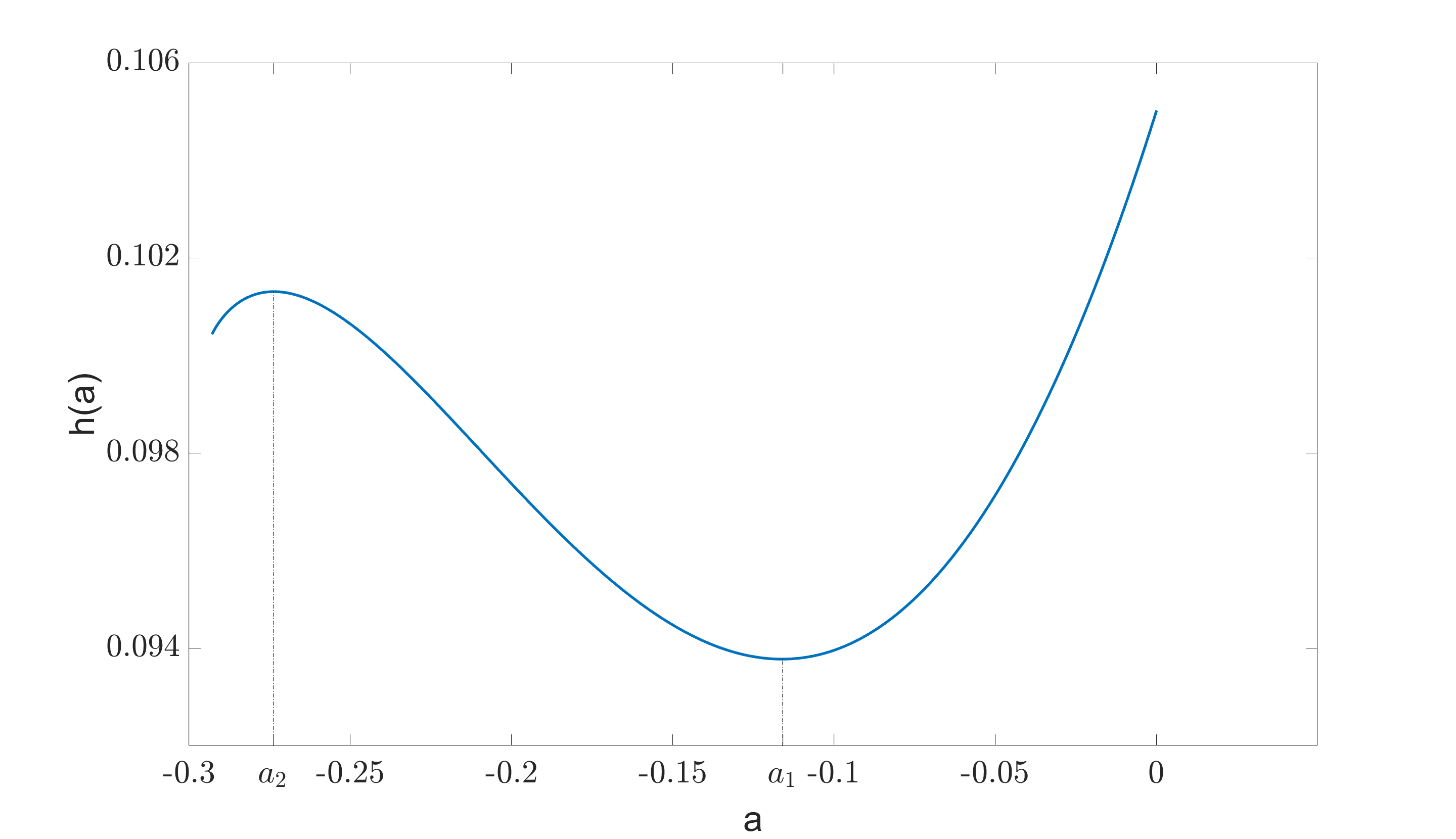}
    \caption{Plot of the function $h(a) = \mathcal{K}(u_a)$, for a velocity $v =1$ and a potential strength $\gamma = 0.21$. In this case $v_{cr}(\gamma) \sim 1.078$. We observe the presence of two critical points: the local maximum at $a_2$ and the local minimum at $a_1$.}
    \label{fig:h(a)}
\end{figure}
    Proposition \ref{bozza_prop_h(a)} tells us that, given $a \in [-1+\frac{v}{\sqrt{2}},0]$, the function $u_a$ is the unique minimizer of the problem \eqref{bozza_def_h(a)}, up to phase shifts. Moreover, it gives us an explicit expression of $h'(a)$ in \eqref{bozza_formula_h'}. From the behavior of the map $a \to k_v(a)$ in (\ref{bozza_def_k_v}), it's possible to infer the shape of the function $h(a)$. As explained in the following corollary from \cite[Remark 3.2]{maris2003}, this allows us to characterize the stationary state $u_{a_1}$ as a local minimum of $\mathcal{K}$.
    \begin{corollary}
        Let $v \in (0,\sqrt{2})$ and let $\gamma \in (0,\varphi(v))$. Let $a_1,a_2 \in (-1+\frac{v}{\sqrt{2}},0)$ be the two solutions of the equation $\gamma = k_v(a) $ with $a_2<a_1$, as defined in \eqref{bozza_remark_a12}. Then, the map $a \to h(a)$ is increasing on $[-1+\frac{v}{\sqrt{2}},a_2]$, it is decreasing on $[a_2,a_1]$, and it is increasing on $[a_1,0]$ (see Figure \ref{fig:h(a)}). In particular, we have $h(a_2) > h(a_1)$, or equivalently, $\mathcal{K}(u_{a_2})>\mathcal{K}(u_{a_1})$. Moreover, if we define
        \begin{equation}
            V_2 =\{u \in \mathcal{E}, \ \inf_{x \in \mathbf{R}}|u(x)| \geq 1+a_2\},
        \end{equation}
            then, up to phase shifts, $u_{a_1} \in \mathcal{E}$ is the unique minimizer of the problem: 
            \begin{equation}
                \inf\{\mathcal{K}(u), \ u \in V_2\}.
                \label{bozza_minimization_problem}
            \end{equation}
        \label{bozza_corollary_u_a1}
    \end{corollary}
    By means of Proposition \ref{bozza_prop_h(a)} and Corollary \ref{bozza_corollary_u_a1}, we can deduce the following dynamical property.
    \begin{lemma}
        Under the same assumptions of Corollary \ref{bozza_corollary_u_a1}, consider an initial datum $\psi \in \mathcal{E}$ such that $\psi \in V_2$ and $\mathcal{K}(\psi) < \mathcal{K}(u_{a_2})$. Call $u(t) \in C^0(\mathbf{R}, \mathcal{E})$ the unique solution to (\ref{bozza_NLS}) with $u(0) = \psi$. Then, it holds $$u(t) \in V_2 \quad  \text{for all} \quad t \in \mathbf{R}.$$ Consequently, we have $\mathcal{K}(u(t)) = \mathcal{K}(\psi)$ for all $t \in \mathbf{R}$.
        \label{bozza_pot_well}
    \end{lemma}
    \begin{proof} The existence of a state $\psi \in V_2$ satisfying $\mathcal{K}(\psi)<\mathcal{K}(u_{a_2})$ is guaranteed by Corollary \ref{bozza_corollary_u_a1}. Moreover, by the minimization property of $u_{a_2}$, we must have $\inf_{x \in \mathbf{R}}|\psi(x)| > 1+a_2$. Consider $\psi$ as above and call $u \in C^0(\mathbf{R}, \mathcal{E})$ the unique solution to (\ref{bozza_NLS}) with $u(0) = \psi$. By contradiction, suppose that there exists a time $t^*>0$ such that $u(t^*) \notin V_2$. We define $T\geq0$ as
        \begin{equation*}
            T := \sup\{t \geq0|\ \  u(s) \in V_2, \ \  \forall s \in [0,t) \}.
        \end{equation*}
        By the contradiction hypothesis, $T < +\infty$. By continuity of the map $t \to \inf_{x \in \mathbf{R}}|u(t,x)|$ for the solution $u(t) \in C^0(\mathbf{R}, \mathcal{E})$, we have $T >0$. Moreover, we have $\inf_{x \in \mathbf{R}}|u(T,x)| = 1+a_2$, again by continuity. By the minimization property of $u_{a_2}$, this contradicts the energy conservation (Lemma \ref{bozza_lemma_en_cons_nonvanishing}), since $\mathcal{K}(\psi) < \mathcal{K}(u_{a_2})$ by hypothesis, while $\mathcal{K}(u(T))\geq \mathcal{K}(u_{a_2}) = h(a_2)$. The same contradiction is reached if we suppose the existence of $t^*<0$ such that $u(t^*) \notin V_2$. We conclude $u(t) \in V_2$ for all $t \in \mathbf{R}$. Finally, $\mathcal{K}(u(t)) = \mathcal{K}(\psi)$ for all $t \in \mathbf{R}$ follows by the conservation of energy.
    \end{proof}
    \section{\texorpdfstring{Orbital stability of $u_{a_1}$}{Orbital stability of u a 1}}
    In this section we prove the orbital stability of the stationary state $u_{a_1}$. Our analysis relies on the variational characterization in Corollary \ref{bozza_corollary_u_a1}), on the dynamical property in Lemma \ref{bozza_pot_well} and on the argument of Cazenave and Lions for the orbital stability of constrained minima \cite{cazenave_lions}. The latter requires us to show the compactness of the minimizing sequences for the problem \eqref{bozza_minimization_problem}, which we do following the approach of Bethuel et al. in \cite{bethuel_existence}.\\
    
    Throughout this section the velocity $v \in (0,\sqrt{2})$ and the potential strength $\gamma \in (0,\varphi(v))$ will be fixed. Let's recall the definition of $a_1,a_2 \in [-1+\frac{v}{\sqrt{2}},0]$ as solutions to $\gamma =k_v(a)$, with $a_2<a_1$ (see \eqref{bozza_remark_a12}), and the definition of the set $$V_2 =\{u \in \mathcal{E}, \ \inf_{x \in \mathbf{R}}|u(x)| \geq 1+a_2\},$$ from Corollary \ref{bozza_corollary_u_a1}.\\
    For a given nowhere vanishing field $u \in \mathcal{E}$, with $u(x) = \rho(x) e^{i\theta(x)}$ for all $x \in \mathbf{R}$, define the energy density
    $$e(u) := \frac{1}{2}(\partial_x\rho)^2 +  \frac{1}{2}\rho^2(\partial_x\theta)^2+\frac{1}{4}(\rho^2-1)^2$$
    and the momentum density
    $$p(u) :=\frac{1}{2} \frac{|u|^2-1}{|u|^2} \Re(iu\partial_x\overline{u}) =  \frac{1}{2}(\rho^2-1)\partial_x\theta.$$
    Then, we have the following lemma.
    \begin{lemma}
        There exists $\varepsilon \in (0,1)$ such that, for any $u \in V_2$, it holds 
            \begin{equation}
                |vp(u)| \leq (1-\varepsilon)e(u) \quad \text{in} \ \  \mathbf{R} \qquad 
            \text{and} \qquad  \varepsilon E_{\gamma}(u)\leq \mathcal{K}(u)
            \end{equation}
        \label{bozza_lemma_K>E}
    \end{lemma}
    \begin{proof}
    Since $1+a_2 \in (\frac{v}{\sqrt{2}},1)$, we can consider $\varepsilon \in (0,1)$ such that $1+a_2 = \frac{1}{1-\varepsilon}\frac{v}{\sqrt{2}}.$ Given $u \in V_2$, with $u(x) = \rho(x) e^{i\theta(x)}$ for all $x \in \mathbf{R}$, we have by means of Young's inequality
    \begin{equation}
       |vp(u)| = \bigl|\frac{v}{2}(\rho^2-1)\partial_x\theta\bigr| \leq \frac{v}{\sqrt{2}\rho}\Bigl(\frac{1}{4}(\rho^2-1)^2 + \frac{1}{2}\rho^2(\partial_x\theta)^2\Bigr) \leq \frac{v}{\sqrt{2}\rho}e(u) \leq (1-\varepsilon)e(u),
       \label{bozza_ineq_e_p}
    \end{equation}
    Integrating the expression above on the real line, we conclude $|v\mathcal{P}(u)| \leq (1-\varepsilon)E_{\gamma}(u)$, so that $\mathcal{K}(u) \geq E_{\gamma}(u) - |v\mathcal{P}(u)| \geq \varepsilon E_{\gamma}(u)$.
    \end{proof}
    We now study the properties of minimizing sequences for the problem \eqref{bozza_minimization_problem}. Notice that, by definition, a minimizing sequence $\{u_n\}_{n \in \mathbf{N}}$ satisfies $u_{n} \in V_2$ for all $n \in \mathbf{N}$. This is a necessary restriction, since the energy $\mathcal{K}$ is unbounded from below in its domain (see Lemma \ref{bozza_K_unbounded}). Moreover, by the minimization property of $u_{a_1}$ from Corollary \ref{bozza_corollary_u_a1}, a minimizing sequence satisfies $\mathcal{K}(u_n) \to \mathcal{K}(u_{a_1})$ as $n \to \infty$.\\
    We begin with the following Lemma.
    \begin{lemma}
        Consider $\{u_n\}_{n \in \mathbf{N}}$ a minimizing sequence for \eqref{bozza_minimization_problem}. There exists an element $u \in \dot{H}^1(\mathbf{R})\cap L^2_{loc}(\mathbf{R})$ and a subsequence $\{u_{\sigma(n)}\}$ that satisfy
        \begin{equation*}
            \partial_xu_{\sigma(n)} \rightharpoonup \partial_xu \quad \text{in} \quad L^2(\mathbf{R}), \qquad  u_{\sigma(n)} \rightharpoonup u \quad \text{in} \quad L^2_{loc}(\mathbf{R}), \qquad u_{\sigma(n)} \to u \quad \text{in} \quad L_{loc}^{\infty}(\mathbf{R}).
        \end{equation*}
        \label{bozza_lemma_weak_limit}
    \end{lemma}
    \begin{proof}
        Consider $\{u_n\}_{n \in \mathbf{N}}$ a minimizing sequence. By means of Lemma \ref{bozza_lemma_K>E}, we know that the sequence $\{\partial_xu_n\}$ is uniformly bounded in $L^2(\mathbf{R})$. Since $L^2(\mathbf{R})$ is a reflexive Banach space, there exists $g \in L^2(\mathbf{R})$ such that, up to subsequences, $\partial_xu_n \rightharpoonup g$ in $L^2(\mathbf{R})$. Moreover, the sequence $\{|u_n(0)|\} \subset \mathbf{R}_+$ is uniformly bounded. We deduce the existence of a constant $c \in \mathbf{C}$ such that, up to subsequences, $u_n(0) \to c $ as $n \to \infty.$ We denote by $\{u_{\sigma(n)}\}$ the subsequence of $\{u_n\}$ such that
        \begin{equation*}
            \partial_xu_{\sigma(n)} \rightharpoonup g \quad \text{in} \quad L^2(\mathbf{R}), \qquad \text{and} \qquad  u_{\sigma(n)}(0) \to c \qquad \text{as} \quad n \to +\infty.
        \end{equation*}
    We define $$u(x) = \int_{0}^xg(s)ds + c.$$
    Then, $\partial_xu = g \in L^2(\mathbf{R})$ and $u(0) = c$. For any $x_0 \in \mathbf{R}$, we have
    \begin{equation*}
        |u_{\sigma(n)}(x_0)-u(x_0)| \leq \bigl|\int_{0}^{x_0}\partial_xu_{\sigma(n)}(s)-g(s)ds\bigr| + \bigl|u_{\sigma(n)}(0)-c\bigr| \to 0 , \qquad \text{as} \quad n \to \infty.
    \end{equation*}
    Hence, $u_{\sigma(n)}$ converges to $u$ pointwise on $\mathbf{R}$. Similarly, one can show that the sequence $\{u_{\sigma(n)}\}$ satisfies $u_{\sigma(n)} \rightharpoonup u$ in $L^2_{loc}(\mathbf{R})$ as $n \to \infty$. This implies that $u_{\sigma(n)} \rightharpoonup u$ in $H^1_{loc}(\mathbf{R})$.\\
    Finally, for any given $A>0$, one can show that the sequence $\{u_{\sigma(n)}\}$ is uniformly bounded in $H^1([-A,A])$. By Rellich's compactness theorem \cite[Theorem 8.8]{brezis}, and the pointwise convergence, we have $u_{\sigma(n)} \to u$ in $L^{\infty}([-A,A])$, up to subsequences. We want to show that the convergence holds for the sequence $\{u_{\sigma(n)}\}$ itself. Given $A>0$, suppose by contradiction that $\{u_{\sigma(n)}\}$ does not converge uniformly to $u$ in $[-A,A]$. This means that there exists $\varepsilon >0$ such that $||u_{\sigma(n)}-u||_{L^{\infty}([-A,A])} \geq \varepsilon$ infinitely often. We can extract a subsequence $\{u_{\pi(\sigma(n))}\}$ such that $||u_{\pi(\sigma(n))}-u||_{L^{\infty}([-A,A])} \geq \varepsilon$ for every $n \in \mathbf{N}$. This subsequence is still uniformly bounded in $H^1([-A,A])$ and converges to $u$ pointwise. By Rellich's compactness theorem, we can extract a further subsequence that converges to $u$ uniformly in $[-A,A]$, and this is a contradiction.
    \end{proof}
    In the following lemma we state further properties of the weak limit $u$ from Lemma \ref{bozza_lemma_weak_limit}.
    \begin{lemma}
        Let $\{u_n\}_{n \in \mathbf{N}}$ be a minimizing sequence for \eqref{bozza_minimization_problem}. Let $u \in \dot{H}^1(\mathbf{R})\cap L^2_{loc}(\mathbf{R})$ be its weak limit, given by Lemma \ref{bozza_lemma_weak_limit}. Then $u \in V_2$ and, for any $A>0$, we have
        \begin{equation*}
            \int_{-A}^Ae(u) \ dx + \frac{\gamma}{2}|u(0)|^2 \leq \liminf_{n \to \infty}\int_{-A}^Ae(u_n) \ dx + \frac{\gamma}{2}|u_n(0)|^2,
        \end{equation*}
        and
        \begin{equation*}
            \int_{-A}^A p(u) \  dx = \lim_{n \to \infty}\int_{-A}^Ap(u_n) \ dx.
        \end{equation*}
        \label{bozza_lemma_inequality_e_and_p}
    \end{lemma}
    \begin{proof}
    Set $A>0$. The first inequality follows by weak lower semi-continuity of the $L^2(\mathbf{R})$-norm and the $L^{\infty}_{loc}(\mathbf{R})$ convergence. Moreover, by means of Fatou's lemma, the same inequality holds on the real line, i.e.
    \begin{equation}
        \int_{\mathbf{R}}e(u) \ dx + \frac{\gamma}{2}|u(0)|^2 \leq \liminf_{n \to \infty}\int_{\mathbf{R}}e(u_n) \ dx + \frac{\gamma}{2}|u_n(0)|^2.
    \end{equation}
    Since $\{u_n\}_{n \in \mathbf{N}}$ is a minimizing sequence and by means of Lemma \ref{bozza_lemma_K>E}, this implies $u \in \mathcal{E}$. Since $u_n \to u$ in $L^{\infty}_{loc}(\mathbf{R})$, we have $u \in V_2$. For the second equality, we write
    \begin{equation}
       \begin{split}
            |\int_{-A}^A p(u_n)-p(u) dx| &=\frac{1}{2} \Bigl|\int_{-A}^A \frac{|u_n|^2-1}{|u_n|^2} \Re(iu_n\partial_x\overline{u}_n) - \frac{|u|^2-1}{|u|^2} \Re(iu\partial_x\overline{u}) dx \Bigr|\\ & \leq \frac{1}{2} \Bigl|\int_{-A}^A  \Bigl(\frac{|u_n|^2-1}{|u_n|^2}-\frac{|u|^2-1}{|u|^2} \Bigr) \Re(iu_n\partial_x\overline{u}_n) dx \Bigr| + \frac{1}{2}\Bigl|\int_{-A}^{A} \frac{|u|^2-1}{|u|^2} \Re\bigl(iu\partial_x\overline{u}-iu_n\partial_x\overline{u}_n\bigr)dx\Bigr|.
       \end{split}
       \label{bozza_mom_convergence}
    \end{equation}
    For the first contribution, we write
    \begin{equation*}
        \Bigl|\frac{|u_n|^2-1}{|u_n|^2}-\frac{|u|^2-1}{|u|^2} \Bigr| = \frac{\Bigl||u_n|^2-|u|^2\Bigr|}{|u|^2|u_n|^2} \leq (1+a_2)^{-4} \Bigl||u_n|^2-|u|^2\Bigr| \to 0, \qquad \text{in} \quad L^{\infty}([-A,A]),
    \end{equation*}
    as $n \to \infty$, by the $L^{\infty}_{loc}(\mathbf{R})$-convergence. At the same time, the integral $\int_{-A}^A |u_n\partial_x\overline{u}_n|dx$ is uniformly bounded, which implies that the first contribution converges to zero. For the second contribution in (\ref{bozza_mom_convergence}), we have
    \begin{equation*}
        \begin{split}
            &\Bigl|\int_{-A}^A\frac{|u|^2-1}{|u|^2} \Re\bigl(iu\partial_x\overline{u}-iu_n\partial_x\overline{u}_n\bigr)dx\Bigr| = \Bigl|\Re i\int_{-A}^A\frac{|u|^2-1}{|u|^2} \bigl(u\partial_x(\overline{u}-\overline{u}_n)-(u_n-u)\partial_x\overline{u}_n\bigr)dx\Bigr| \\ & \leq \Bigl|\int_{-A}^A\frac{|u|^2-1}{|u|^2}u\partial_x(\overline{u}-\overline{u}_n)dx\Bigr| + \int_{-A}^{A}\Bigl| \frac{|u|^2-1}{|u|^2} (u_n-u)\partial_x\overline{u}_n\Bigr|dx.
        \end{split}
    \end{equation*}
    The first term goes to zero as $n \to \infty$ since $\partial_xu_n \rightharpoonup \partial_xu$ in $L^2(\mathbf{R})$. The second term goes to zero as $n \to \infty$ since $u_n \to u$ in $L^{\infty}_{loc}(\mathbf{R})$.
    \end{proof}
    In the next lemma we consider a minimizing sequence $\{u_n\}_{n \in \mathbf{N}}$ and its weak limit $u \in V_2$. We would like to show that, up to phase shifts, $u$ coincides with $u_{a_1}$. In order to do that, we show that the sequence $\{u_n\}$ doesn't lose energy $e$ nor momentum $p$ at infinity. In particular,  we show that any such loss has the effect of decreasing the energy $\mathcal{K}(u)$ of the limit map. By the minimization property of $u_{a_1}$, we conclude that the sequence stays compact and that $u = u_{a_1}e^{i\phi}$, for $\phi \in \mathbf{R}$. The approach we follow is based on the concentration-compactness argument as developed in \cite[Theorem 3]{bethuel_existence} for the case of traveling waves to the standard Gross-Pitaevskii equation. Before we state the lemma, recall the definition of the distance $d_A$ in \eqref{bozza_d_A}.
    \begin{lemma}
        Let $A >0$. Let $\{u_n\}_{n \in \mathbf{N}}$ be a minimizing sequence for \eqref{bozza_minimization_problem}, and let $u \in V_2$ be its weak limit. Then, there exists $\phi \in \mathbf{R}$ such that $u = e^{i\phi}u_{a_1}$ and, up to subsequences, $$d_{A}(u_n, e^{i\phi}u_{a_1}) \to 0, \qquad \text{as} \quad n \to \infty.$$
        \label{bozza_lemma_conv_min_seq}
    \end{lemma}
    \begin{proof}
        Let $\{u_n\}_{n \in \mathbf{N}}$ be a minimizing sequence and let $u \in V_2$ be its weak limit, given by Lemma \ref{bozza_lemma_weak_limit}. Choose $\mu >0$. By Lemma \ref{bozza_lemma_inequality_e_and_p}, there exists $A_{\mu}>0$ and $n_{\mu} \in \mathbf{N}$ such that
        \begin{equation*}
            \Bigl|\int_{[-A_{\mu},A_{\mu}]}vp(u_n)dx - v\mathcal{P}(u)\Bigr| \leq \mu
        \end{equation*}
        and
        \begin{equation*}
            \int_{[-A_{\mu},A_{\mu}]}e(u_n)dx + \frac{\gamma}{2}|u_n(0)|^2 \geq E_{\gamma}(u)-\mu,
        \end{equation*}
        for all $n \geq n_{\mu}$. As in Lemma \ref{bozza_lemma_K>E}, consider $\varepsilon \in (0,1)$ such that $1+a_2 = \frac{1}{1-\varepsilon}\frac{v}{\sqrt{2}}$. We have, 
        \begin{equation}
            \begin{split}
                \mathcal{K}(u_n) &= \int_{[-A_{\mu},A_{\mu}]}\bigl(e(u_n) - vp(u_n) \bigr) dx + \frac{\gamma}{2}|u_n(0)|^2 +  \int_{ \mathbf{R}\backslash[-A_{\mu},A_{\mu}]}\bigl(e(u_n) - vp(u_n) \bigr) dx \\ & \geq \mathcal{K}(u) -2\mu +    \int_{ \mathbf{R}\backslash[-A_{\mu},A_{\mu}]}\bigl(e(u_n) - vp(u_n) \bigr) dx \\ &
                \geq \mathcal{K}(u) - 2\mu + \varepsilon\int_{\mathbf{R}\backslash[-A_{\mu},A_{\mu}]} e(u_n)dx, \qquad  \qquad \text{for all} \quad   n \geq n_{\mu},
                \label{bozza_K_inequality}
            \end{split}
        \end{equation}
        where in the last inequality we used Lemma (\ref{bozza_lemma_K>E}) and the fact that $\{u_n\}_{n \in \mathbf{N}} \subset V_2$. We conclude that $\mathcal{K}(u_n) \geq \mathcal{K}(u) - 2\mu$ for $n \geq n_{\mu}$.  Since $\mu >0$ is arbitrary, it must be $\mathcal{K}(u) = \mathcal{K}(u_{a_1})$. 
        Since $u \in V_2$ (see Lemma \ref{bozza_lemma_inequality_e_and_p}), we conclude $u = u_{a_1}e^{i\phi}$ for $\phi \in \mathbf{R}$.\\ Since $\mathcal{K}(u_n) \to \mathcal{K}(u)$ as $n \to \infty$, we can consider $m_{\mu} > n_{\mu}$ such that $\mu \geq \mathcal{K}(u_n)-\mathcal{K}(u)$ for all $n \geq m_{\mu}$. Using (\ref{bozza_K_inequality}), we obtain \begin{equation*}
            \int_{\mathbf{R}\backslash[-A_{\mu},A_{\mu}]}e(u_n)dx \leq \frac{3\mu}{\varepsilon}, \qquad \forall n\geq m_{\mu}.
        \end{equation*}
        Using Lemma \ref{bozza_lemma_K>E}, this implies
        $$\Bigl|\int_{\mathbf{R}\backslash[-A_{\mu},A_{\mu}]}vp(u_n) dx\Bigr| \leq \frac{1-\varepsilon}{\varepsilon}3\mu \qquad \forall n\geq m_{\mu}.$$ We conclude that
        \begin{equation*}
            \bigl|v\mathcal{P}(u_n)-v\mathcal{P}(u)\bigr| \leq \Bigl|\int_{[-A_{\mu},A_{\mu}]}vp(u_n)dx - v\mathcal{P}(u)\Bigr| +  \Bigl|\int_{\mathbf{R}\backslash[-A_{\mu},A_{\mu}]}vp(u_n)dx \Bigr| \leq \mu + \frac{1-\varepsilon}{\varepsilon}3\mu.
        \end{equation*}
        for all $n \geq m_{\mu}$. This implies that $\mathcal{P}(u_n) \to \mathcal{P}(u)$ as $n \to \infty$. Consequently, we have 
        \begin{equation}
            E_{\gamma}(u_n) \to E_{\gamma}(u), \qquad \text{as} \quad n \to \infty.
            \label{bozza_convergence_energy}
        \end{equation}
        Together with the weak convergence of Lemma \ref{bozza_lemma_weak_limit}, this implies $\partial_xu_n \to e^{i\phi}\partial_xu_{a_1}$ in $L^2(\mathbf{R})$ as $n \to \infty$, up to subsequences. Using (\ref{bozza_convergence_energy}), we then deduce that
        \begin{equation}
            \int_{\mathbf{R}}(1-|u_n|^2)^2dx \to  \int_{\mathbf{R}}(1-|u|^2)^2dx \qquad \text{as} \quad n \to \infty.
        \end{equation} 
        Using the uniform convergence $u_n \to u$ in $L^{\infty}_{loc}(\mathbf{R})$ from Lemma \ref{bozza_lemma_weak_limit} and using (\ref{bozza_convergence_energy}), we have, up to subsequences, $$\bigl(1-|u_n|^2\bigr) \rightharpoonup \bigl(1-|u|^2\bigr) \qquad \text{in} \quad L^2(\mathbf{R}), \qquad \text{as} \quad n \to \infty.$$ We conclude that $(1-|u_n|^2) \to (1-|u|^2)$ in $L^2(\mathbf{R})$ and that $d_{A}(u_n, e^{i\phi}u_{a_1}) \to 0 $ as $n \to \infty$ for any $A >0$.
    \end{proof}
    We are now ready to prove Proposition \ref{bozza_intro_stability}. Recall that the state $u(x;\xi_1)$ coincides with $u_{a_1}$ (see Remark \ref{bozza_remark_xi_1_a_1}).
    \begin{proof}[Proof of Proposition \ref{bozza_intro_stability}]
        We argue by contradiction and we suppose that $u_{a_1}$ is not orbitally stable. In this case we can find $A >0$ and $\varepsilon>0$, a sequences of times $\{t_n\} \subset \mathbf{R}$ and a sequence of initial data $\{u_{0,n}\} \subset \mathcal{E}$ such that 
        \begin{equation}
            d_A(u_{0,n}, u_{a_1}) \to 0 \qquad \text{as} \quad n \to \infty,
            \label{bozza_stability_contr}
        \end{equation}
        and
        \begin{equation}
             \inf_{\phi \in \mathbf{R}}d_A(u_n(t_n), e^{i\phi}u_{a_1}) \geq \varepsilon \qquad \forall n \in \mathbf{N},  
            \label{bozza_stability_contr_2}
        \end{equation}
        where $u_n(t)$ is the unique solution to (\ref{bozza_NLS}) with initial datum $u_{0,n}$.\\
        First, we want to show that $\{u_{0,n}\}\subset \mathcal{E}$ is a minimizing sequence for the problem in \eqref{bozza_minimization_problem}. Since $E_{\gamma}$ is continuous for the distance $d_A$ (Lemma \ref{bozza_lemma_energy0}),  we have, using (\ref{bozza_stability_contr}),
        \begin{equation}
            E_{\gamma}(u_{0,n}) \to E_{\gamma}(u_{a_1}), \qquad \text{as} \quad n \to \infty.
        \end{equation}
        Moreover, using Lemma \ref{bozza_lemma_modulus}, we have that $|u_{0,n}|$ converges uniformly to $|u_{a_1}|$ in $\mathbf{R}$. The uniform convergence implies that there exists $m \in \mathbf{N}$ such that $u_{0,n} \in V_2$ for any $n \geq m$. Up to extracting a subsequence, we can suppose that $u_{0,n} \in V_2$ for every $n \in \mathbf{N}$. In this way, the functions $u_{0,n}$ are nowhere vanishing and we can evaluate the momentum $\mathcal{P}(u_{0,n})$. From the continuity of $\mathcal{P}$ in Lemma \ref{bozza_mom_continuity}, we deduce
        \begin{equation}
            \mathcal{P}(u_{0,n}) \to \mathcal{P}(u_{a_1}), \qquad \text{as} \quad n \to +\infty.
        \end{equation}
        We conclude that $\{u_{0,n}\}_{n \in \mathbf{N}} \subset V_2$ and $\mathcal{K}(u_{0,n}) \to \mathcal{K}(u_{a_1})$ as $n \to \infty$. Hence, $\{u_{0,n}\}_{n \in \mathbf{N}}$ is a minimizing sequence.\\
        Since the inequality $\mathcal{K}(u_{a_1}) < \mathcal{K}(u_{a_2})$ holds true (see Corollary \ref{bozza_corollary_u_a1}), there exists $M>0$ such that $\mathcal{K}(u_{0,n}) < \mathcal{K}(u_{a_2})$ for any $n \geq M$. By Lemma \ref{bozza_pot_well}, we have  that the sequence of evolved states $\{u_n(t_n)\}_{n \in \mathbf{N}}$ satisfies $u_n(t_n) \in V_2$ and $\mathcal{K}(u_n(t_n)) = \mathcal{K}(u_{0,n})$ for all $n \geq M$. We conclude that $\{u_n(t_n)\}_{n \geq M}$ is a minimizing sequence for \eqref{bozza_minimization_problem}.  By Lemma \ref{bozza_lemma_conv_min_seq}, there exists $\phi \in \mathbf{R}$ such that $d_A(u_n(t_n), e^{i\phi}u_{a_1}) \to 0$ as $n \to \infty$, and this is in contradiction with (\ref{bozza_stability_contr_2}).
        \end{proof}

        We can now prove Corollary \ref{bozza_corollary_intro_momentum_exchange}.
        \begin{proof}[Proof of Corollary \ref{bozza_corollary_intro_momentum_exchange}]
            Let $A>0$ and $\varepsilon>0$ be given. For $\delta >0$ consider $u_0 \in U_0$ such that $d_A(u_0,u_{a_1}) <\delta$. Let $\psi(t)$ be the solution to \eqref{bozza_NLS} with initial datum $u_0$. We want to prove that, if $\delta>0$ is small enough, we have $|\mathcal{P}(\psi(t))-\mathcal{P}(u_0)|< \varepsilon$ for all $t \in \mathbf{R}$.\\
            By choosing $\delta >0$ small enough, we have $u_0 \in V_2$, thanks to Lemma \ref{bozza_lemma_modulus}. By the continuity of $\mathcal{K}$ with respect to $d_A$ and Corollary \ref{bozza_corollary_u_a1}, we have $\mathcal{K}(u_0)<\mathcal{K}(u_{a_2})$, , for $\delta>0$ small enough. This implies that $\psi(t) \in V_2$ for all $t \in \mathbf{R}$, by Lemma \ref{bozza_pot_well}.
            In particular, the momentum $\mathcal{P}(\psi(t))$ is well-defined for all times. Then, for a given $t \in \mathbf{R}$, we write
             \begin{equation}
                |\mathcal{P}(\psi(t))-\mathcal{P}(u_0)| \leq |\mathcal{P}(\psi(t))-\mathcal{P}(u_{a_1})| + |\mathcal{P}(u_{a_1})-\mathcal{P}(u_0)|.
                \label{bozza_momentum_split}
            \end{equation}                         
            By the continuity of $\mathcal{P}$ with respect to $d_A$, and the stability of $u_{a_1}$, both contributions to the right-hand-side of \eqref{bozza_momentum_split} can be made smaller than $\varepsilon/2$ by choosing $\delta>0$ small enough.
        \end{proof}
\appendix 
\section{Proof of Proposition \ref{bozza_prop_sa_extension}}
\label{bozza_appendix}
In this Appendix we prove Proposition \ref{bozza_prop_sa_extension}.
    We consider the operator $(A,D(A))$ given by $$A = -\partial_x^2+iv\partial_x, \qquad D(A) = \{f \in H^2(\mathbf{R}) \ | \  f(0) =0 \ \}.$$
    In Fourier space, this operator reads $(B,D(B))$ where $$B = p^2-vp, \qquad D(B) = \Bigl\{\hat{f} \in L^2(\mathbf{R}) \ | \ \int_{\mathbf{R}}|p^2\hat{f}|^2dp < \infty, \quad  \text{and} \quad  \int_{\mathbf{R}}\hat{f}(p)dp =0\Bigr\}.$$
    The first step is to compute the adjoint operator $(B^*,D(B^*))$. To this purpose, we introduce the following proposition, whose proof is based on \cite[Proposition 8.6]{teta2018primer}.
    \begin{proposition}
    Set $\lambda > \frac{v^2}{4}$. Then
        $$Ran(B+\lambda)^{\perp} = \{u \in L^2(\mathbf{R}) \ | \ u = q \hat{G}_v^{\lambda}, \ q \in \mathbf{C}\},$$
        where $\hat{G}_v^{\lambda}$ has been defined in (\ref{bozza_green_function_Fourier}).
    \end{proposition}
    \begin{proof}
        The set $Ran(B+\lambda)^{\perp}$ consists of functions $w \in L^2(\mathbf{R})$ satisfying $$\int_{\mathbf{R}}\overline{w}(p)(p^2-vp+\lambda)u(p)dp = 0, \qquad \forall u \in D(B).$$ A solution is $w = q \hat{G}^{\lambda}$ for any $q \in \mathbf{C}$. We need to show that these solutions are unique. For $w \in L^2(\mathbf{R})$, define $g(p) = \overline{w}(p)(p^2-vp+\lambda)$ and consider the problem
        \begin{equation}
            \int_{\mathbf{R}}g(p)u(p)dp = 0, \qquad \forall u \in D(B).
            \label{bozza_range_orth}
        \end{equation}
        We show that, in the class of step functions, the constant function is the unique solution. Consider $\Omega_1,..., \Omega_n$ a sequence of intervals of $\mathbf{R}$ such that $\Omega_i\cap\Omega_j = \emptyset$ if $i \neq j$ and $\cup_i\Omega_i = \mathbf{R}$. Let $g(p) = \sum_{i=1}^n c_i\chi_{\Omega_i}(p)$, where $c_i$ are constants in $\mathbf{C}$ and $\chi_{\Omega_i}$ denotes the characteristic function of the interval $\Omega_i\subset \mathbf{R}$. We proceed by contradiction and, without loss of generality, we assume $c_1 \neq c_2$. Then, equation (\ref{bozza_range_orth}) reads
        \begin{equation*}
            c_1\int_{\Omega_1}u(p)dp + c_2 \int_{\Omega_2}u(p)dp + \sum_{i \neq 1,2}c_i\int_{\Omega_i}u(p)dp =0.
        \end{equation*}
        Consider $u = \chi_{(-R,R)}(b_1\chi_{\Omega_1}+b_2\chi_{\Omega_2})$, where $R>0$ is large enough so that $m_i:= |(-R,R)\cap \Omega_i| >0$ for $i = 1,2$ and where $b_1,b_2 \in \mathbf{R}\backslash\{0\}$ satisfy $b_1m_1+b_2m_2=0$. Then we have
        \begin{equation*}
            0 = c_1\int_{\Omega_1}u(p)dp + c_2 \int_{\Omega_2}u(p)dp = c_1m_1b_1+c_2m_2b_2 =(c_1-c_2)m_1b_1. 
        \end{equation*}
        We conclude $c_1=c_2$, and that the constant function is the unique solution to (\ref{bozza_range_orth}) in the class of step functions. Using a density argument, one can show that the constant function is the unique solution also in the class $L^2_{loc}(\mathbf{R})$ and then conclude the proof.
    \end{proof}
    We can now compute the adjoint operator $(B^*, D(B^*))$.
    \begin{proposition}
        We have the following characterization:
        \begin{equation*}
            \begin{split}
                & D(B^*)  = \{\nu \in L^2(\mathbf{R}) \ | \ \nu = w+ q\hat{G}_v^{\lambda}, \quad w \in D(\hat{H}_0), \quad q \in \mathbf{C} \} \\ & (B^*+\lambda)\nu = (\hat{H}_0+\lambda)w,
            \end{split}
        \end{equation*}
        where $(\hat{H}_0,D(\hat{H}_0))$ is the free operator defined in (\ref{bozza_H_0_free}).
    \end{proposition}
    \begin{proof}
        First we notice that $$\langle \hat{G}_v^{\lambda}, (B+\lambda)u \rangle = 0, \qquad \forall u \in D(B).$$
        This implies that $\hat{G}_v^{\lambda} \in D(B^*)$ and $(B^*+\lambda)\hat{G}_v^{\lambda}=0$. Then, if $w \in D(\hat{H}_0)$, we have $\langle w, (B+\lambda)u \rangle = \langle (\hat{H}_0+\lambda)w, u\rangle$ for all $ u \in D(B)$. This implies that $w \in D(B^*)$ and $(B^*+\lambda)w = (\hat{H}_0+\lambda)w$. We conclude that if $\nu = w + q \hat{G}^{\lambda}$, then $\nu \in D(B^*)$ and $(B^*+\lambda)\nu = (\hat{H}_0+\lambda)w$. Conversely, suppose that $\nu \in D(B^*)$. Since $Ran(\hat{H}_0+\lambda) = L^2(\mathbf{R})$, there exists $w \in D(\hat{H}_0)$ such that
        \begin{equation*}
            (B^*+\lambda)\nu = (\hat{H}_0+\lambda)w.
        \end{equation*}
        For any $u \in D(B)$, we have $\langle w, (B+\lambda)u \rangle = \langle (\hat{H}_0+\lambda)w, u\rangle = \langle (B^*+\lambda)\nu, u \rangle$. We conclude that $w \in D(B^*)$ and $(B^*+\lambda)w = (B^*+\lambda)\nu$. Thus we have, for any $u \in D(B)$, $\langle \nu-w,(B+\lambda)u \rangle = 0$. This implies that $\nu-w \in Ran(B+\lambda)^{\perp}$, and using the previous proposition we can conclude the proof.
    \end{proof}
    The strategy now is to define a suitable restriction of the operator $(B^*,D(B^*))$ that is self-adjoint. The first candidate is a symmetric restriction of $B^*$, which is defined as follows. For any $\gamma \in \mathbf{R} \backslash\{0\}$, we define $(B_{\gamma},D(B_{\gamma}))$ as
    \begin{equation*}
        D(B_{\gamma}) = \Bigl\{ u \in D(B^*) \ | \ q = -\frac{\gamma}{\sqrt{2\pi}}\int_{\mathbf{R}} u(p) dp \Bigr\}, \qquad B_{\gamma} = B^*|_{D(B_{\gamma})}.
    \end{equation*}
    The operator $(B_{\gamma},D(B_{\gamma}))$ is symmetric: for $\nu_1,\nu_2 \in D(B_{\gamma})$
    \begin{equation*}
        \begin{split}
            \langle \nu_1, (B_{\gamma}+\lambda)\nu_2 \rangle &= \langle w_1+q_1\hat{G}_v^{\lambda}, (\hat{H}_0+\lambda)w_2\rangle = \langle (\hat{H}_0+\lambda)w_1, w_2 \rangle + q_1\int_{\mathbf{R}}\overline{w_2(p)}dp \\ & = \langle (\hat{H}_0+\lambda)w_1, \nu_2 \rangle - \overline{q_2}\int_{\mathbf{R}}w_1(p)dp +  q_1\int_{\mathbf{R}}\overline{w_2(p)}dp \\ & = \langle (\hat{H}_0+\lambda)w_1, \nu_2 \rangle - \overline{q_2}\int_{\mathbf{R}}\nu_1(p)dp +  q_1\int_{\mathbf{R}}\overline{\nu_2(p)}dp = \langle (B_{\gamma}+\lambda)\nu_1,\nu_2\rangle.
        \end{split}
    \end{equation*}
    We will make use of the following self-adjointess criterion (see \cite[Proposition 4.5]{teta2018primer}).
    \begin{proposition}
       Let $(H,D(H))$ be a symmetric operator defined in a Hilbert space $\mathcal{H}$. If there exists $z \in \mathbf{C}$, $\Im z \neq 0$, such that 
       \begin{equation*}
           Ran(H-z) = Ran(H-\overline{z}) = \mathcal{H},
       \end{equation*}
       then $(H,D(H))$ is self-adjoint.
    \end{proposition}
    \begin{proof}[Proof of Proposition \ref{bozza_prop_sa_extension}]
        Our goal is to show the self-adjointness of $(B_{\gamma},D(B_{\gamma}))$ by means of the self-adjointness criterion above. For any $z \in \mathbf{C}$, with $\Im z \neq 0$, and $f \in L^2(\mathbf{R})$, we want to solve the problem
        \begin{equation}
            (B_{\gamma} - z)u = f, \quad u \in D(B_{\gamma}).
            \label{bozza_range_sa}
        \end{equation}
        Writing $u = w^{\lambda}+q\hat{G}_v^{\lambda}$, we have $(B_{\gamma}-z)u = (B_{\gamma}+\lambda)u - (z+\lambda)u = (\hat{H}_0-z)w^{\lambda} - (z+\lambda)q\hat{G}_v^{\lambda}$, where $q = -\frac{\gamma}{\sqrt{2\pi}}\int_{\mathbf{R}}u(p)dp$. We write (\ref{bozza_range_sa}) as
        \begin{equation*}
            (\hat{H}_0 - z)w^{\lambda} -(z+\lambda)q\hat{G}_v^{\lambda} = f.
        \end{equation*}
        We write $(\hat{H}_0-z)^{-1}(p) = \sqrt{2\pi}\ \hat{G}^{-z}(p) = \frac{1}{p^2-vp-z}$ and we use the identity $$\hat{G}_v^{-z}-\hat{G}_v^{\lambda} = \frac{1}{\sqrt{2\pi}}\frac{z+\lambda}{(p^2-vp-z)(p^2-vp+\lambda)} = \sqrt{2\pi}(z+\lambda)\hat{G}_v^{-z}\hat{G}_v^{\lambda}.$$
        We can write (\ref{bozza_range_sa}) as
        \begin{equation}
            u + \frac{\gamma}{\sqrt{2\pi}}\Bigl(\int_{\mathbf{R}}u(p)dp \Bigr)
\  \hat{G}_v^{-z} = \sqrt{2\pi}\hat{G}_v^{-z} f.
\label{bozza_eq_u_sa}
        \end{equation}
        If we integrate over $\mathbf{R}$, we get
        \begin{equation*}
            \int_{\mathbf{R}}u(p)dp = \frac{\sqrt{2\pi}}{1+\gamma G_v^{-z}(0)}\int_{\mathbf{R}}\hat{G}_v^{-z}(p)f(p)dp.
        \end{equation*}
        where $G_v^{-z}$ is defined in (\ref{bozza_green}). Inserting this into equation (\ref{bozza_eq_u_sa}) we obtain
        \begin{equation}
            u = (\hat{H}_0-z)^{-1}f - \frac{\gamma}{1+\gamma G_v^{-z}(0)}\Bigl(\int_{\mathbf{R}}\hat{G}_v^{-z}(p)f(p)dp \Bigr) \hat{G}_v^{-z}.
            \label{bozza_resolvent_fourier_gamma}
        \end{equation}
        By the self-adjointess criterion, we conclude that $(B_{\gamma},D(B_{\gamma}))$ is self-adjoint in $L^2(\mathbf{R})$. The expression of the resolvent $(B_{\gamma}-z)^{-1}f$ corresponds to the left-hand side of (\ref{bozza_resolvent_fourier_gamma}), for $f \in L^2(\mathbf{R})$.\\
        One can verify that the operator $(B_{\gamma}, D(B_{\gamma}))$ corresponds to the Fourier transform of the original operator $(H_{\gamma}, D(H_{\gamma}))$, namely $B_{\gamma} = \mathcal{F}H_{\gamma}\mathcal{F}^{-1}$. By the unitarity of the Fourier transform, the operator $(H_{\gamma}, D(H_{\gamma}))$ is self-adjoint (see \cite[Proposition 4.17]{teta2018primer}). We can compute the kernel of the resolvent operator $(H_{\gamma}-z)^{-1}$, for $z \in \mathbf{C}$ with $\Im z \neq 0$, by the inverse Fourier transform of (\ref{bozza_resolvent_fourier_gamma}). We obtain, for $g \in L^2(\mathbf{R})$,
        \begin{equation*}
            \bigl((H_{\gamma}-z)^{-1}g\bigr)(x) = \bigl(\mathcal{F}^{-1}(B_{\gamma}-z)^{-1}\mathcal{F}g\bigr)(x) = \bigl((H_0-z)^{-1}g\bigr)(x) - \frac{\gamma}{1+\gamma/2k}G^{-z}_v(x) \int_{\mathbf{R}}\hat{G}^{-z}_v(p)\hat{g}(p)dp.
        \end{equation*}
        If we write $$\int_{\mathbf{R}}\hat{G}_v^{-z}(p)\hat{g}(p)dp = \int_{\mathbf{R}}\overline{G^{\overline{-z}}_v}(y)g(y)dy = \int_{\mathbf{R}}G^{-z}_{-v}(y)g(y)dy,$$
        we obtain that the integral kernel of $(H_{\gamma}-z)^{-1} $ reads
        \begin{equation}
            (H_{\gamma}-z)^{-1}(x,y) = (H_0-z)^{-1}(x-y) - \frac{\gamma}{1+\gamma/2k} G^{-z}_{-v}(y)G^{-z}_v(x).
            \label{bozza_comput_kernel}
        \end{equation}
        In particular, $(H_{\gamma}-z)^{-1}$ is a compact perturbation of $(H_{0}-z)^{-1}$. By Weyl's theorem, $\sigma_{ess}(H_{\gamma}) = \sigma_{ess}(H_0) = [-\frac{v^2}{4}, +\infty)$. Finally, $H_{\gamma}$ does not have any eigenvalue for $\gamma>0$. Indeed, suppose $H_{\gamma}u = E u$, for $u \in D(H_{\gamma})$ with $||u||_{L^2(\mathbf{R})}=1$. First suppose that $E \geq -\frac{v^2}{4}$. We have
       \begin{equation*}
           -\partial_x^2u(x)+iv\partial_xu(x) = E u(x), \qquad x\neq0.
       \end{equation*}
       If we consider $\eta$ such that $u = e^{i\frac{v}{2}x}\eta$, we obtain
       \begin{equation*}
           -\partial_x^2\eta = (E+\frac{v^2}{4})\eta, \qquad x\neq0,
       \end{equation*}
       and we conclude that $\eta$, and hence $u$, cannot belong to $L^2(\mathbf{R})$. Now suppose $E < -\frac{v^2}{4}$. On the one hand we have $\langle u, H_{\gamma}u\rangle = E$. On the other hand, using point v) of Remark \ref{bozza_remark_sa}, we have
       \begin{equation*}
          \begin{split}
               \langle u, H_{\gamma}u \rangle &= \int_{\mathbf{R}}|\partial_xu|^2 + ivu(x)\partial_x\overline{u}(x) dx + \gamma |u(0)|^2 \geq \int_{\mathbf{R}}|\partial_xu|^2dx - |v| \bigl|\bigl|\partial_xu\bigr|\bigr|_{L^2(\mathbf{R})} + \gamma |u(0)|^2 \\ & \geq  \int_{\mathbf{R}}|\partial_xu|^2dx - \frac{v^2}{4} - ||\partial_xu||^2_{L^2(\mathbf{R})} + \gamma |u(0)|^2\geq -\frac{v^2}{4},
          \end{split}
       \end{equation*}
       which is a contradiction. We conclude that $\sigma(H_{\gamma}) = \sigma_c(H_{\gamma}) = [-\frac{v^2}{4}, +\infty)$ and that formula (\ref{bozza_comput_kernel}) holds for $z \in \mathbf{C}\backslash\sigma(H_\gamma)$.\\
    \end{proof}
\section{Proof of Proposition \ref{bozza_persistence_regularity}}
\label{section: appendix_II}
In this Appendix we prove Proposition \ref{bozza_persistence_regularity}. 
\begin{proof}
    Let $u_0 \in X^2_{\gamma}\cap \mathcal{E}$ and consider $R>0$ such that $|u_0|_{\mathcal{E}}< R$ and $||\tilde{H}_{\gamma}u||_{L^2(\mathbf{R})} < R$. For $\tilde{T}>0$ small enough, our goal is to find $\tilde{w} \in C^0((-\tilde{T},\tilde{T}),H^1(\mathbf{R})) \cap C^1((-\tilde{T},\tilde{T}),L^2(\mathbf{R}))$ such that $\tilde{w} = \Phi_{\tilde{T},u_0}(\tilde{w})$, where $\Phi_{\tilde{T},u_0}$ has been defined in (\ref{bozza_Phi}). We set $\tilde{R}>0$ and $\tilde{T}>0$, and we define the space
    \begin{equation*}
        \begin{split}
            \tilde{W}_{\tilde{R}}(\tilde{T}) = \{\tilde{w} \in &C^0((-\tilde{T},\tilde{T}),H^1(\mathbf{R})) \cap C^1((-\tilde{T},\tilde{T}),L^2(\mathbf{R})),  \\ & \text{s.t.} \quad \tilde{w}(0) =0, \quad  \text{and} \quad \forall t \in (-\tilde{T},\tilde{T}) \quad  ||\tilde{w}(t)||_{H^1(\mathbf{R})} < \tilde{R}, \quad ||\partial_t\tilde{w}(t)||_{L^2(\mathbf{R})} < \tilde{R}\}
        \end{split}
    \end{equation*}
    For $\tilde{w} \in \tilde{W}_{\tilde{R}}(\tilde{T})$, we set $u(t) = \tilde{w}(t)+T_{\gamma}(t)u_0$ and we have $u \in C^1((-\tilde{T},\tilde{T}), L^2(\mathbf{R}))$ by Proposition \ref{bozza_prop_X2}. Then, we have $F\circ u \in C^1((-\tilde{T},\tilde{T}), L^2(\mathbf{R}))$ with
    \begin{equation*}
        \partial_t (F\circ u)(t) = \partial_t u -2|u|^2\partial_tu -u^2\partial_t\overline{u}.
    \end{equation*}
    For $t \in (-\tilde{T},\tilde{T})$, it holds
    \begin{equation*}
        \partial_t\Phi_{\tilde{T},u_0}(\tilde{w})(t) = i \int_0^te^{-isH_{\gamma}}\partial_t(F \circ u)(t-s) ds + i e^{-itH_{\gamma}}F(u_0) \in L^2(\mathbf{R}),
    \end{equation*}
    where we use $\tilde{w}(0)=0$. In particular, $\partial_t\Phi_{\tilde{T},u_0}(\tilde{w})\in C^0((-\tilde{T}, \tilde{T}))$ and there exist $C_R>0$, depending only on $R$, and $C_{\tilde{R}}>0$, depending only on $\tilde{R}$, such that
    $$\Bigl|\Bigl|\partial_t\Phi_{\tilde{T},u_0}(\tilde{w})(t)\Bigr|\Bigr|_{L^2(\mathbf{R})} \lesssim C_R(1+\tilde{T}C_{\tilde{R}}).$$
    Since $\Phi_{\tilde{T},u_0}(\tilde{w})(0)=0$, and using inequality (\ref{bozza_contraction_0}), we can choose $\tilde{R}$ big enough and then $\tilde{T}$ small enough, so that $\Phi_{\tilde{T},u_0}(\tilde{w}) \in \tilde{W}_{\tilde{R}}(\tilde{T}).$ Similarly, for $\tilde{w}_1,\tilde{w}_2 \in \tilde{W}_{\tilde{R}}(\tilde{T})$, we have
    \begin{equation*}
        \begin{split}
            \Bigl|\Bigl|\partial_t\Phi_{\tilde{T},u_0}(\tilde{w}_1)(t)-&\partial_t\Phi_{\tilde{T},u_0}(\tilde{w}_2)(t)\Bigr|\Bigr|_{L^2(\mathbf{R})} \\ &\lesssim \tilde{T}C_RC_{\tilde{R}}\Bigl(||\tilde{w}_1-\tilde{w}_2||_{(L^{\infty}(-\tilde{T},\tilde{T}),H^1(\mathbf{R}))}+||\partial_t\tilde{w}_1-\partial_t\tilde{w}_2||_{(L^{\infty}(-\tilde{T},\tilde{T}),L^2(\mathbf{R}))}\Bigr).
        \end{split}
    \end{equation*}
    Together with inequality (\ref{bozza_contraction}), we obtain that $\Phi_{\tilde{T},u_0}$ is a contraction on $\tilde{W}_{\tilde{R}}(\tilde{T})$ for $\tilde{T}$ small enough. By the fixed point theorem, there exists a unique $\tilde{w} \in \tilde{W}_{\tilde{R}}(\tilde{T})$ such that $\Phi_{\tilde{T},u_0}(\tilde{w})=\tilde{w}$. We conclude that $u(t) = \tilde{w}(t)+T_{\gamma}(t)u_0$ is the unique solution to (\ref{bozza_NLS}) with $u(0) = u_0$ for $t \in (-\tilde{T},\tilde{T})$.\\
    For $t \in (-\tilde{T},\tilde{T})$, we have
    \begin{equation*}
        \begin{split}
            \lim_{\tau \to 0}\frac{e^{-i\tau H_{\gamma}}-1}{\tau}\Phi_{\tilde{T},u_0}(\tilde{w})(t) &=  \lim_{\tau \to 0} \Bigl[ \frac{\Phi_{\tilde{T}+\tau,u_0}(\tilde{w})(t)-\Phi_{\tilde{T},u_0}(\tilde{w})(t)}{\tau} - \frac{i}{\tau}\int_{t}^{t+\tau}e^{-i(t+\tau-s)H_{\gamma}}F(u(s))ds \Bigr] \\ & = \partial_t \Phi_{\tilde{T},u_0}(\tilde{w})(t)-iF(u(t)).
        \end{split}
    \end{equation*}
    This implies that $\tilde{w}(t) \in D(H_{\gamma})$ for any $t \in (-\tilde{T},\tilde{T})$, and 
    \begin{equation*}
        \partial_t u(t) = -i\tilde{H}_{\gamma}u(t)+iF(u(t)), \qquad \forall t \in (-\tilde{T},\tilde{T}).
    \end{equation*}
    In particular, $u(t) \in X^2_{\gamma}$ for every $t \in (-\tilde{T},\tilde{T})$. By iterating the fixed point argument, we can find a maximal solution $u$ defined on the interval $(-\tilde{T}_-,\tilde{T}_+)$ with $\tilde{T}_{\pm} \in (0,+\infty]$. If $\tilde{T}_{\pm}< +\infty$, we have
    \begin{equation}
        |u|_{\mathcal{E}}+||\tilde{H}_{\gamma}u||_{L^2(\mathbf{R})} \to +\infty, \quad \text{as} \quad t \to \pm \tilde{T}_{\pm}.
        \label{bozza_alternative}
    \end{equation}
    If $(-T_-,T_+)$ represents the maximal time of existence given by Corollary \ref{bozza_corollary}, for $T_{\pm} \in (0,+\infty]$, we have $\tilde{T}_{\pm} \leq T_{\pm}$, by uniqueness. Suppose $\tilde{T}_+ < T_+$. From the Duhamel formula in Proposition \ref{bozza_duhamel}, we have
    \begin{equation*}
        \partial_t u = -ie^{-itH_{\gamma}}\tilde{H}_{\gamma}u_0 + i e^{-itH_{\gamma}}F(u_0) + i \int_{0}^t e^{-isH_{\gamma}} \partial_t(F\circ u) (t-s)ds.
    \end{equation*}
    Using the fact that $|u|_{\mathcal{E}}$ remains bounded on $[0,\tilde{T}_+)$, and using Lemma \ref{bozza_L_infty}, we can find a constant $C>0$ such that
    \begin{equation*}
        ||\partial_t u||_{L^2(\mathbf{R})} \leq C + C \int_{0}^{t} ||\partial_t u||_{L^2(\mathbf{R})}.
    \end{equation*}
    By the Gr\"{o}nwall lemma, $\partial_t u$ is uniformly bounded in $L^2(\mathbf{R})$ on the interval $[0,\tilde{T}_+)$. By (\ref{bozza_eq_regular}), we have that $\tilde{H}_{\gamma}u$ stays bounded on $[0,\tilde{T}_+)$, and this contradicts (\ref{bozza_alternative}). We conclude that $\tilde{T}_{\pm} = T_{\pm}$.
\end{proof}
\section*{Acknowledgments.}
\noindent The first author is partially supported by the PRIN 2022 Project 2022YXWSLR "Boundary analysis for dispersive and viscous fluids", by Istituto Nazionale di Alta Matematica through the GNAMPA Research Group, by the Italian Ministry of University and Research (MUR)
through the Excellence Department Project awarded to GSSI, CUP D13C22003740001. The second author is grateful to Raffaele Scandone for fruitful and kind discussions on the theory of point interaction.
\bibliographystyle{siam}
\bibliography{biblio_1dsuperflow}
\end{document}